\documentclass[a4paper]{amsart}
\usepackage{mathmarco}

\usepackage[maxalphanames=99,url=false,maxbibnames=99,giveninits=true,style=alphabetic]{biblatex}

\AtEveryBibitem{
  \ifentrytype{article}{\clearfield{doi}\clearfield{issn}\clearfield{isbn}}{}
  \ifentrytype{incollection}{\clearfield{doi}\clearfield{issn}\clearfield{isbn}}{}
  \ifentrytype{book}{\clearfield{doi}\clearfield{issn}\clearfield{isbn}}{}
}

\usepackage{etoolbox}

\makeatletter

\@ifpackageloaded{hyperref}{%
  \@ifpackageloaded{biblatex}{%
    \DeclareRobustCommand{\SkipTocEntry}[5]{}% hyperref only 
  }%
}{}

\let\origsubsection\subsection
\renewcommand{\subsection}{%
  \@ifstar{\subsection@star}{\origsubsection}% If star: use modified, else: use original
}%
\newcommand{\subsection@star}[1]{%
  \addtocontents{toc}{\protect\SkipTocEntry}
  
  \origsubsection*{#1}
  
}

\newcommand{\addresseshere}{%
  \enddoc@text\let\enddoc@text\relax
}

\makeatother

\title{Proper cocycles, measure equivalence and $L_p$-Fourier multipliers}

\author[Wang,  Xia and Yao]{Simeng Wang,  Runlian Xia  \and Gan Yao}
	
\address[Simeng Wang]{Institute for Advanced Study in Mathematics, Harbin Institute of Technology,  Harbin 150001, China.}
\email{simeng.wang@hit.edu.cn}

\address[Runlian Xia]{School of mathematics and Statistics, University of Glasgow, University Avenue, Glasgow G12 8QQ, UK}
\email{Runlian.Xia@glasgow.ac.uk}

\address[Gan Yao]{Institute for Advanced Study in Mathematics, Harbin Institute of Technology,  Harbin 150001, China.}
\email{gan.yao3@outlook.com}

\begin{document}

\begin{abstract}
  We develop a new transference method for completely bounded \( L_p \)-Fourier multipliers via proper cocycles arising from probability measure-preserving group actions. This method extends earlier results by Haagerup and Jolissaint, which were limited to the case \( p = \infty \).  Based on this approach, we present a new and simple proof of the main result in \cite{MR4813921} regarding the pointwise convergence of noncommutative Fourier series on amenable groups, refining the associated estimate for maximal inequalities. In addition, this framework yields a transference principle for \( L_p \)-Fourier multipliers from lattices in linear Lie groups to their ambient groups, establishing a noncommutative analogue of Jodeit's theorem. As a further application, we construct a natural analogue of the Hilbert transform on \( SL_2(\mathbb{R}) \).
  \end{abstract}
  
  \maketitle
  \tableofcontents
  \section{Introduction} 
  In noncommutative harmonic analysis, the study of Fourier multipliers on group algebras is of paramount importance. This investigation traces its origins to the pioneering work of Haagerup \cite{MR520930}, which settled a long-standing conjecture concerning Grothendieck's approximation property for the reduced C*-algebras of free groups using Fourier multipliers.
  In subsequent work, de Cannière and Haagerup \cite{MR0784292} and Cowling and Haagerup \cite{MR0996553} investigated  completely bounded Fourier multipliers on simple Lie groups, leading to profound results on the weak amenability of these groups. These approximation properties also play an essential role in the modern theory of von Neumann algebras. Indeed, the aforementioned investigations in \cite{MR0996553} derived one of the first results of non-isomorphisms of von Neumann algebras generated by lattices; the approximation properties also constitute a core ingredient in Popa's deformation/rigidity theory, and the study
  of strong solidity and uniqueness of Cartan subalgebras (see e.g. \cite{MR2334200,MR2680430,MR3087388,MR3259044}).
  
  As in the classical setting, Fourier multipliers on noncommutative $L_p$-spaces over groups arise naturally but are significantly more delicate to analyze. One of the first remarkable achievements in this direction is due to Lafforgue and de la Salle \cite{MR2838352}, who demonstrated that the completely bounded Fourier multipliers on the noncommutative $L_p$ space over $SL_3(\mathbb{Z})$ exhibits a rigid phenomenon. This result implies, in particular, that these \( L_p \)-spaces lack the completely bounded approximation property for large \( p \). This line of research was subsequently extended to other simple Lie groups by Haagerup, de Laat, and collaborators \cite{MR3047470,MR3453357,MR3781331,MR3035056}. Junge, Mei, and Parcet \cite{MR3283931,MR3776274} established the complete boundedness of noncommutative Riesz transforms and certain Hörmander--Mikhlin type multipliers associated with finite-dimensional cocycles. Parcet, Ricard, and de la Salle \cite{MR4408121} provided a Hörmander--Mikhlin type criterion for \( SL_n(\mathbb{R}) \) with \( n \geq 3 \), which was further generalized to higher-rank simple Lie groups by Conde-Alonso, González-Pérez, Parcet, and Tablate \cite{condealonso2024hormandermikhlintheoremhighrank}. Recently, Caspers \cite{MR4669315} established noncommutative Calderón--Torchinsky theorems on Lie groups. Moreover, Mei, Ricard, and Xu \cite{MR4408132} developed a Hörmander--Mikhlin theory for free groups and amalgamated free products of groups.\par
  
    In the aforementioned studies, a variety of transference techniques for Fourier multipliers have become a fundamental tool in multiple contexts, for instance the transference between Fourier and Schur multipliers \cite{MR2866074,MR3378821} as well as the Coifman--Weiss transference method \cite{MR481928}. Amongst others, a basic ingredient in the aforementioned seminal work \cite{MR0996553} on the isomorphism problem for group von Neumann algebras is a transference method that lifts Fourier multipliers from a lattice to its ambient group, which immediately implies that the lattices and the ambient group share precisely the same Cowling--Haagerup constant for weak amenability. This means, roughly speaking, that these lattices admit uniform bounds for appropriate approximate identities consisting of completely bounded Fourier multipliers on the Fourier algebras (see also \cite{MR3476201}). 
  Inspired by this work, Cowling and Zimmer \cite{MR1007408} and Jolissaint \cite{MR1756981} introduced an analogous transference method for proper cocycles arising from group actions on probability spaces, thereby showing that weak amenability---and more generally other approximation properties---remains invariant under measure equivalence (see also \cite{MR3310701}).
  
  The corresponding $L_p$ theory of the aforementioned transference method arises naturally in the study of noncommutative $L_p$-spaces and the Fourier multipliers acting on them. Indeed, a central motivation for developing the theory of noncommutative Fourier multipliers on $L_p$-spaces is to understand the $L_p$-analogue of Cowling--Haagerup's theory---for instance the rigidity phenomena of the associated noncommutative $L_p$-spaces over higher rank lattices when $p$ is close to $2$. In view of Cowling--Haagerup's approach, it is thus essential to understand the relationships between $L_p$ Fourier multipliers on different lattices and those on their ambient groups. Moreover, as remarked by \cite[Appendix B]{MR3482270}, the extension of Fourier multipliers from lattices to ambient groups may be viewed as a noncommutative counterpart of Jodeit's theorem, which in the classical setting addresses the extension from $\mathbb Z ^d$ to $\mathbb R^d$. It is worth noting that extending such methods to the $L_p$ setting is far from straightforward. Prior results rely heavily on Gilbert's characterization of completely bounded Fourier multipliers on the Fourier algebra, a tool which is unavailable in the $L_p$ framework.
  
  This paper aims to establish the aforementioned transference theory in the setting of Fourier multipliers on noncommutative $L_p$-spaces over groups, introducing new examples of completely bounded Fourier multipliers arising from measure equivalence, lattices, and Hilbert transforms. In the following let $G$ be a locally compact group. The group von Neumann algebra of $G$, denoted by $\mathcal{L}(G)$,  is the von Neumann algebra generated by the left regular representation $\lambda$ on $ L_2 (G) $, that is,
  \[
  \mathcal{L}(G) :=\overline{\rm span}^{w*} 
  \big\{ \lambda_g : g \in G \big\} \subset B(L_2 (G)).
  \] 
  We write $\lambda(f)=\int_G f(g)\lambda_g\dd g \in \mathcal{L}(G)$ for $f\in L_1 (G)$. When $G$ is unimodular, the von Neumann algebra $\mathcal{L}(G)$ admits a normal semifinite faithful trace satisfying 
  \[
  \tau(\lambda(f))=f(e),\quad \forall f\in C_c(G)*C_c(G),
  \]
  where $C_c(G)$ is the space of all continuous functions with compact support.  The associated noncommutative \( L_p \)-spaces \( L_p(\mathcal{L}(G)) \) with respect to $\tau$ naturally generalize classical \( L_p \)-spaces on the Pontryagin dual group (see the next section for precise definitions).
  Given $m\in L_\infty(G)$, the Fourier multiplier with symbol \( m \) is defined as the (potentially unbounded) linear operator $T_m: \lambda (C_c(G))\subset \mathcal L (G)\rightarrow \mathcal L (G) $ given by 
  $$
  T_m (\lambda (f))= \int_G m(g) f(g)\lambda_g \dd g.
  $$
  Establishing the boundedness of such Fourier multipliers on noncommutative \( L_p \)-spaces poses significant challenges that are absent in the commutative setting. We summarize our results in the following four axes. 
  \subsection*{Transference via  cocycles and measure equivalence}
Consider a measure-preserving action of $G$ on a finite measure space $(X,\mu )$, and let $H$ be another locally compact group. A Borel map $\alpha:G\times X\to H$ is called a \emph{Borel cocycle} if it satisfies
\begin{equation}\label{Eqn: cocycle identity left}
  \alpha(g_1 g_2, x) = \alpha(g_1, g_2 x) \alpha(g_2, x).
\end{equation}
This identity is the natural compatibility condition for defining an action of $G$ on the spaces of Borel functions $\phi:X\to H$ via $(g,\phi)\mapsto \phi(g\cdot )\alpha(g,x) $. Such cocycles arise frequently in various settings of group dynamics, including prominent examples from the theory of lattices and measure equivalence, which will be discussed later. Haagerup \cite{MR0996553,MR3476201} and Jolissaint \cite{MR1756981,MR3310701} observed a close connection between Fourier multipliers on $H$ and those on $G$. More precisely, given a symbol $m\in L_\infty(H)$, one may define an associated symbol $\widetilde{m} \in L_\infty (G)$ by 
\begin{equation}\label{eqn: formula of tilde m}
	\widetilde{m}(g)=\frac{1}{\mu(X)}\int_X m(\alpha(g,x))\dd\mu(x),\quad g\in G.
\end{equation}
It has been shown in \cite{MR1756981,MR3310701} that $\widetilde{m} $ is a completely bounded (resp. completely positive) Fourier multiplier on  $\mathcal L(G)$ if $m$ is so on $\mathcal L(H)$ (see also \cite{MR0996553,MR3476201} for partial results for lattices); moreover, this transference keeps certain geometric or analytic properties when $\alpha$ is proper. These observations have noteworthy applications in the study of approximation properties and measure equivalence invariants, as discussed in previous paragraphs.
The following theorem establishes analogous results for Fourier multipliers on the associated noncommutative $L_p$-spaces.
\begin{theorem}\label{Thm: Transference from L(H) to L(G)}
	Let $1\leq p\leq \infty$. Let $H$ be a discrete group, $G$ be a unimodular locally compact second countable  group and $(X,\mu)$ be a finite measure space. Let $\theta$ be a $\mu$-preserving action of $G$ on $ (X,\mu)$. Assume that the crossed product $L_\infty(X)\rtimes_\theta G$ is QWEP. Let $\alpha:G\times X\to H$ and $m, \widetilde{m}$ be as above. Then we have
	\[
	\norm{T_{\widetilde{m}}:L_p(\mathcal{L}(G))\to L_p(\mathcal{L}(G))}_{\mathrm{cb}}\leq \norm{T_{m}:L_p(\mathcal{L}(H))\to L_p(\mathcal{L}(H))}_{\mathrm{cb}}.
	\]
\end{theorem}
Our approach to this theorem is notably different from the case $p=\infty$ studied in \cite{MR3476201,MR1756981,MR3310701}; the latter relies on Gilbert's theorem, which is unavailable for $1<p<\infty$. Moreover, the space $L_p$ here can be indeed replaced by other Orlicz spaces. The key to proving the theorem is the construction of a new special injective $*$-homomorphism from $\mathcal{L}(G)$ into the crossed product $L_\infty(X)\rtimes (H\times G)$, which encodes the information of the cocycle $\alpha$. We will provide full details of the construction of this homomorphism and prove the theorem in Section \ref{Sec: embeddeding vNa}.

This approach leads to several applications, which will be discussed in the following paragraphs. Let us first mention an immediate corollary. It is well-known that the above cocycle exists when $G$ and $H$ are measure equivalent (see Section \ref{Subsec: Measure equivalence of groups} for more details). It has been shown in \cite{MR3310701} that weak amenability is an invariant under measure equivalence. Now we may establish an analogous result in the $L_p$-setting. For $1\leq p\leq \infty$, we say that $G$ is \emph{$L_p$-weakly amenable} if $G$ admits a sequence of symbols $(m_n)_{n\geq 0}\subset A(G)$ such that $m_n\to 1$ uniformly on compact sets and 
\[
\sup_n \norm{T_{m_n}:L_p(\mathcal{L}(G))\to L_p(\mathcal{L}(G))}_{\mathrm{cb}}<\infty ,
\]
where $A(G)$ denotes the Fourier algebra of $G$. The smallest possible supremum of the above quantity, for which such a sequence $(m_n)$ exists, is denoted by $\Lambda_{p,\mathrm{cb}}(G)$, which serves as the $L_p$-analogue of the celebrated Cowling--Haagerup constant. When $G$ is a discrete group, $L_p$-weak amenability is nothing but the completely bounded approximation property of $L_p (\mathcal L (G))$, which is known to be a stronger property than the Haagerup--Kraus approximation property (AP), as demonstrated in \cite{MR1971296}. The study of $L_p$-weak amenability for higher rank lattices for $p$ close to $2$ remains a challenging problem in the field of operator spaces. It is widely recognized among experts that a resolution of this problem would possibly provide an answer to Connes' rigidity problem for higher-rank lattices.
The following result asserts that $L_p$-weak amenability is invariant under measure equivalence between discrete groups, assuming a natural QWEP condition. 
\begin{corollary}\label{Coro: transferred multiplier of two lattices in non-amenable groups}
Let $1\leq p\leq \infty$. Let $\Gamma$ and $\Delta$ be two discrete groups that are measure equivalent with a measure equivalence coupling $(\Sigma,\sigma)$. Assume that $L_\infty (\Sigma/\Delta) \rtimes \Gamma $ and $L_\infty (\Gamma\backslash\Sigma) \rtimes \Delta $ are QWEP. Then $\Gamma$ is $L_p$-weakly amenable if and only if $\Delta$ is $L_p$-weakly amenable, and in this case, $\Lambda_{p,\mathrm{cb}}(\Gamma)=\Lambda_{p,\mathrm{cb}}(\Delta)$.
\end{corollary}

\subsection*{Pointwise convergence of noncommutative Fourier series}
The pointwise convergence of Fourier series and the study of associated maximal inequalities constitute a central topic in harmonic analysis. In the noncommutative setting, the work of Cowling, Haagerup, and others on approximation properties of groups (e.g., \cite{MR520930,MR0784292,MR0996553}) can in fact be interpreted as a theory of mean convergence for Fourier series, analogous to the classical Fejér and Bochner--Riesz summation methods. The noncommutative analogue of almost everywhere convergence was introduced by Lance in his study of noncommutative ergodic theory \cite{MR428060}; this type of convergence is typically referred to as \emph{(bilaterally) almost uniform} convergence, abbreviated as \emph{b.a.u.} and \emph{a.u.}, respectively, see Definition \ref{Def: b.a.u and a.u}. The corresponding problem of pointwise convergence of Fourier series on group von Neumann algebras is considerably subtle. A systematic study was recently carried out by Hong, Wang, and the first author in \cite{MR4813921}. Due to various intrinsic difficulties in the theory of noncommutative maximal inequalities, their approach is lengthy and relies on a delicate adaptation of Bourgain's bootstrap method to the noncommutative setting. \par

Our aforementioned transference method yields, in contrast, a significantly simpler proof of one of the main results in \cite{MR4813921}.
More precisely, recall that a locally compact second countable group $G$ is amenable if it admits a sequence of unital positive definite functions $(m_n)_{n \geq 0} \subset A(G)$ such that $m_n \to 1$ uniformly on compact sets. It is well known that $T_{m_n} f \to f$ in norm for all $f \in L_p(\mathcal{L}(G))$ with $1 \leq p < \infty$. In this setting, using the measure equivalence between $G$ and the group of integers $\mathbb{Z}$, we construct such a sequence $(m_n)_{n \geq 0}$ for which, for all $f \in L_p(\mathcal{L}(G))$ and $1 < p < \infty$,
$$
T_{m_n} f \to f \quad \text{b.a.u.\ (a.u.\ if } p \geq 2\text{), as } n \to \infty.
$$
Our new proof also refines the estimates for the corresponding maximal inequalities obtained in \cite{MR4813921}. In particular, by extrapolation, the above b.a.u.\ convergence extends to the space $L\log^2 L(\mathcal{L}(G))$, which is the commonly recognized critical Orlicz space for noncommutative strong-type maximal inequalities, whereas in \cite{MR4813921} only the b.a.u.\ convergence on the larger space $L\log^{22} L(\mathcal{L}(G))$ could be established.
\subsection*{Noncommutative Jodeit theorem}
The classical Jodeit theorem asserts that any \( L_p \)-bounded Fourier multiplier on \( \mathbb{Z}^n \) is the restriction of an \( L_p \)-bounded Fourier multiplier on \( \mathbb{R}^n \). More precisely, given an \( L_p \)-bounded Fourier multiplier \( m \in \ell_\infty(\mathbb{Z}^n) \), which can be viewed as a discrete measure on \( \mathbb{R}^n \), the piecewise linear extension
\[
\widetilde{m} = \chi_{[-1/2,1/2]} * m * \chi_{[-1/2,1/2]}
\]
defines an \( L_p \)-bounded Fourier multiplier on \( \mathbb{R}^n \). As previously noted, it is natural to consider the problem of extending multipliers from a lattice to a locally compact group, which serves as a basic ingredient in Cowling--Haagerup's seminal work \cite{MR0996553}.
Let \( \Gamma \subset G \) be a lattice in a locally compact second countable group \( G \), and let \( \Omega \) be a symmetric fundamental domain. Given a function \( m: \Gamma \to \mathbb{C} \), Haagerup~\cite{MR3476201} considered the function \( \widetilde{m}: G \to \mathbb{C} \) defined by
\begin{equation}\label{Eqn: Haagerup transferred multiplier}
	\widetilde{m} = \chi_\Omega * m * \chi_\Omega,
\end{equation}
and proved that \( \widetilde{m} \) defines a completely bounded Fourier multiplier on \( \mathcal{L}(G) \), provided that \( m \) defines a completely bounded Fourier multiplier on \( \mathcal{L}(\Gamma) \).
The corresponding \( L_p \)-analogue of this result remained open prior to our work; see \cite{MR279533,MR390652} for partial results in the case of abelian groups, and \cite{MR2838352,MR3482270,MR4612556} for related results concerning Schur multipliers and amenable groups.
As observed earlier, the extended multiplier \( \widetilde{m} \) coincides with the one defined in \eqref{eqn: formula of tilde m} when one considers the natural cocycle \( G \times \Omega \to \Gamma \) associated with the lattice structure. In particular, this yields a complete resolution of the noncommutative Jodeit theorem for hyperlinear lattices (for example, any lattice that is a finitely generated subgroup of \( \mathrm{GL}_n(\mathbb{K}) \) for some field \( \mathbb{K} \)).
 \begin{corollary}\label{Cor: transferred multiplier for non-amenable groups}
 	Let $\Gamma\subset G$ be a hyperlinear lattice in a locally compact second countable group $G$ with a fundamental domain $\Omega$. Let $m$ and $\widetilde{m}$ be as above, and $1\leq p\leq \infty$. Then 
 	  \begin{equation}\label{ineq: jodeit for non-amenable group}
 		\norm{T_{\widetilde{m}}:L_p(\mathcal{L}(G))\to L_p(\mathcal{L}(G))}_{\mathrm{cb}}\leq \norm{T_{m}:L_p(\mathcal{L}(\Gamma))\to L_p(\mathcal{L}(\Gamma))}_{\mathrm{cb}}.
 	\end{equation}
 \end{corollary}
As in \cite{MR0996553}, this leads to the following important observation concerning $L_p$-weak amenability.
\begin{corollary}\label{Cor:LpWA lattices}
Let $\Gamma\subset G$ be a hyperlinear lattice in a locally compact second countable group $G$, and $1\leq p\leq \infty$. If $\Gamma$ is $L_p$-weakly amenable, then $G$ is $L_p$-weakly amenable, and in this case, $\Lambda_{p,\mathrm{cb}}(G)\leq \Lambda_{p,\mathrm{cb}}(\Gamma)$.
\end{corollary}
The converse of the statement corresponds to a form of noncommutative de Leeuw's restriction theorem, which remains open except in certain local or asymptotically invariant settings (see \cite{MR3482270,MR4721791}).
\subsection*{A continuous analogue of Hilbert transforms on $SL_2 (\mathbb R)$} Hilbert transforms represent a fundamental example of a Fourier multiplier in harmonic analysis. Accordingly, the study of its noncommutative analogues is of paramount importance. Several instances of noncommutative Hilbert transforms have been investigated for free groups and, more generally, for groups acting on tree-like structures such as $SL_2(\mathbb{Z})$. In particular, a canonical Hilbert transform on $SL_2(\mathbb{Z})$ is studied in \cite{gonzalezperez2022noncommutativecotlaridentitiesgroups}:
\begin{equation}\label{eqn: Hilbert transform on PSL_2(Z)}
  m(\begin{pmatrix}
    a&b\\
    c&d
  \end{pmatrix})=\sgn(ac+bd),\quad \text{for} \begin{pmatrix}
    a&b\\
    c&d
  \end{pmatrix}\in SL_2(\mathbb{Z}) ,
\end{equation}
and is shown to define a completely bounded Fourier multiplier on $L_p(\mathcal{L}(SL_2(\mathbb{Z})))$ for every $1 < p < \infty$, as a consequence of a Cotlar-type identity.
However, this formula (or any genuine idempotent $m$ satisfying $m^2 = 1$) does \emph{not} define a completely bounded Fourier multiplier on $L_p$ for $p \neq 2$ when $SL_2(\mathbb{Z})$ is replaced by $SL_2(\mathbb{R})$; see \cite{MR4882285}. The results therein indicate that any symbol exhibiting discontinuities---i.e., “jumps”---on $SL_2(\mathbb{R})$ fails to yield an $L_p$-bounded Fourier multiplier. This suggests the need for constructing alternative, continuous extensions of $m$ that may serve as proper analogues of the Hilbert transform on $SL_2(\mathbb{R})$. As discussed earlier, a canonical approach to extending $L_p$-Fourier multipliers from a lattice to its ambient group is via noncommutative Jodeit theorems. Thus, by applying Corollary~\ref{Cor: transferred multiplier for non-amenable groups}, we obtain a completely bounded Fourier multiplier $\widetilde{m}$ on $L_p(\mathcal{L}(SL_2(\mathbb{R})))$, which naturally extends the Hilbert transform on $\mathcal{L}(SL_2(\mathbb{Z}))$ and may serve as an analogue for a Hilbert transform on $\mathcal{L}(SL_2(\mathbb{R}))$.
In Section~\ref{Sec: Transference from SL2(Z) to SL2(R)}, we also compute the explicit form of this symbol $\widetilde{m}$. Interestingly, due to the non-uniform geometry of the lattice inclusion $SL_2(\mathbb{Z}) \subset SL_2(\mathbb{R})$, the resulting symbol may exhibit improved decay properties compared to the Euclidean setting---such as H\"ormander--Mikhlin-type behavior---which we discuss further in Remark~\ref{Rmk: decay of symbol tilde m}.

\section{Preliminary results and definitions}
\subsection{Noncommutative $L_p$ spaces and noncommutative $\ell_\infty$-valued $L_p$-spaces} 
Let $\mathcal{M}$ be a von Neumann algebra with a normal semi-finite faithful trace $\tau$. Let $S_+(\mathcal{M})$ be the set of all $x\in \mathcal{M}_+$ such that $\tau(s(x))<\infty$, where $s(x)$ is the support projection of $x$. Denote the linear span of $S_+(\mathcal{M})$ by $S(\mathcal{M})$. Given $1\leq p<\infty$, we define 
\[
\norm{x}_p=(\tau(\abs{x}^p))^{\frac{1}{p}},\quad x\in S(\mathcal{M}),
\]
where $\abs{x}=(x^*x)^{\frac{1}{2}}$. The noncommutative $L_p$ space $L_p(\mathcal{M},\tau)$ is the completion of the normed space $(S(\mathcal{M}),\norm{\cdot}_p)$. In the case of no confusion, we will denote $L_p(\mathcal{M},\tau)$ simply by $L_p(\mathcal{M})$. We refer to \cite{MR1999201} for more information on noncommutative $L_p$ spaces. A map $T:L_p(\mathcal{M})\to L_p(\mathcal{M})$ is \emph{completely bounded} if the completely bounded norm
\[
\norm{T:L_p(\mathcal{M})\to L_p(\mathcal{M})}_{\mathrm{cb}}\coloneqq \sup_k   \| \id_{\mathbb{M}_k} \otimes T : {L_p(\mathbb{M}_k \bar\otimes \mathcal{M})\rightarrow L_p(\mathbb{M}_k \bar\otimes \mathcal{M})}\|
\]
is finite. We also denote $\norm{T:L_p(\mathcal{M})\to L_p(\mathcal{M})}_{\mathrm{cb}}$ by $\norm{T}_{\mathrm{cb}(L_p(\mathcal{M}))}$, and write the inequality ${\norm{T(f)}_p\leq_{\mathrm{cb}} C\norm{f}_p}$ if $\norm{T}_{\mathrm{cb}(L_p(\mathcal{M}))}\leq C$.\par

To define maximal operators, we need the noncommutative vector-valued $L_p$-spaces $L_p(\mathcal{M};\ell_\infty)$ first given by Pisier in \cite{MR1648908} and then generalized by Junge in \cite{MR1916654}. Here $L_p(\mathcal{M};\ell_\infty)$ is the space of all $x=(x_n)_{n\in \mathbb{N}}\subset L_p(\mathcal{M})$ which admit a factorization of the form $x_n=a y_n b$, where $a,b\in L_{2p}(\mathcal{M})$ and $y=(y_n)_{n\in \mathbb{N}}\subset \mathcal{M}$. The norm is given by 
\[
  \norm{x}_{L_p(\mathcal{M};\ell_\infty)}=\inf\{\norm{a}_{2p}\sup_{n}\norm{y_n}_\infty\norm{b}_{2p}\},
\]
where the infimum is taken over all possible decomposition $x=ayb$. We will adopt the convention that the norm $\norm{x}_{L_p(\mathcal{M};\ell_\infty)}$ is denoted by $\norm{{\sup}_n^+ x_n}_{p}$. To compare with the commutative case, we remark that if $(x_n)_{n\in \mathbb{N}}$ is a sequence of self-adjoint operators in $L_p(\mathcal{M})$, then $(x_n)_{n\in \mathbb{N}}$ belongs to $L_p(\mathcal{M};\ell_\infty)$ if and only if there is a positive element $a\in L_p(\mathcal{M})$ such that $-a\leq x_n\leq a$ for any $n\in \mathbb{N}$. If we take the infimum over the $p$-norms of all these $a$'s, we recover the $L_p(\mathcal{M};\ell_\infty)$ norm of $x$, i.e.,
\[
  \norm{{\sup}_n^+ x_n}_{p}=\inf\{\norm{a}_{p}:a\in L_p(\mathcal{M})_+,-a\leq x_n\leq a,\forall n\in \mathbb{N}\}.
\]
We refer to \cite[Section 2]{MR2276775} for more details. We say that a family of linear maps $\Phi_n:L_p(\mathcal{M})\to L_p(\mathcal{M})$ is of \emph{strong type $(p,p)$} with constant $C$ if 
\[
  \norm{{\sup}_n^+ \Phi_n(x)}_{p}\leq C\norm{x}_p,\quad x\in L_p(\mathcal{M}).
\]

\subsection{Group von Neumann algebras, left Hilbert algebras and crossed products}
Let $C_c(G)$ be the vector space consisting of all continuous functions with compact supports. The vector space $C_c(G)$ is a Hilbert algebra with respect to the following product, involution and inner product:
\begin{align*}
  f_1*f_2(g)&=\int_G f_1(g_0)f_2(g_0^{-1}g)\dd g_0,\quad \forall f_1,f_2\in C_c(G),\\
  f^\#(g)&=\overline{f(g^{-1})},\quad \forall f\in C_c(G),\\
  \inner{f_1,f_2}&=\int_G f_1(g)\overline{f_2(g)}\dd g,\quad \forall f_1,f_2\in C_c(G).
\end{align*}\par

Let $(X,\mu)$ be a finite measure space and $G$ be a unimodular locally compact second countable group equipped with a Haar measure. Consider a $\mu$-preserving (anti-)action of $G$ on $X$ and denote the corresponding action on $L_\infty(X)$ by $\theta$. Let $C_c(G,L_\infty(X))$ be the vector space of all w*-continuous $L_\infty(X)$ valued functions on $G$ with compact supports. The vector space $C_c(G,L_\infty(X))$ is a Hilbert algebra with respect to the following product, involution and inner product:
\begin{align*}
  \xi*\eta(g)&=\int_{G}\theta_{g_0}(\xi(gg_0))\eta(g_0^{-1})\dd g_0,\quad \forall\xi,\eta\in C_c(G,L_\infty(X)),\\
  \xi^\#(g)&=\theta_{g}^{-1}(\overline{\xi(g^{-1})}),\quad \forall\xi\in C_c(G,L_\infty(X)),\\
   \inner{\xi,\eta}&=\int_{ G}\int_X \xi(g)(x)\overline{\eta(g)(x)}\dd\mu(x)\dd g,\quad \forall \xi,\eta\in C_c(G,L_\infty(X)).
\end{align*} 
\par
  We define $L_2(G, L_2(X))$ to be the space of all square-integrable functions that take values in $L_2(X)$. For the rest of the paper, we denote $L_2(G, L_2(X))$ by $L_2(L_2)$ for simplicity. The set of all right bounded vectors $\mathfrak{B}'$ and the set of all left bounded vectors $\mathfrak{B}$ are
  \begin{align*}
   \mathfrak{B}' &= \left\{ \eta \in L_2(L_2) : \sup_{\xi\in C_c(G,L_\infty(X))} \frac{\| \xi*\eta \|_{L_2(L_2)}}{\| \xi \|_{L_2(L_2)}} < \infty \right\}, \\
    \mathfrak{B} &= \left\{ \xi\in L_2(L_2) : \sup_{\eta\in \mathfrak{B}'} \frac{\| \xi*\eta \|_{L_2(L_2)}}{\| \eta \|_{L_2(L_2)}} < \infty \right\}.
  \end{align*}
 Note that to each $\eta \in \mathfrak{B}'$, there corresponds a unique bounded operator $\pi_r(\eta) \in B(L_2(L_2))$ such that  
\[
\pi_r(\eta)\xi = \xi * \eta, \quad \forall \xi \in C_c(G, L_\infty(X)),
\]
and to each $\xi \in \mathfrak{B}$, there corresponds a unique bounded operator $\pi_l(\xi) \in B(L_2(L_2))$ such that  
\[
\pi_l(\xi)\eta = \xi * \eta, \quad \forall \eta \in \mathfrak{B}'.
\]
Also observe that $C_c(G, L_\infty(X)) \subset \mathfrak{B}$. The crossed product $L_\infty(X) \rtimes_\theta G$ is defined as the von Neumann algebra generated by $\pi_l(C_c(G, L_\infty(X)))$ in $B(L_2(L_2))$. The sets $\mathfrak{B}$ and $\mathfrak{B}'$ satisfy  
\[
\pi_l(\mathfrak{B}) \subset L_\infty(X) \rtimes_\theta G \quad \text{and} \quad \pi_r(\mathfrak{B}') \subset (L_\infty(X) \rtimes_\theta G)'.
\]
Moreover, there exists a trace $\tau$ on the positive cone $(L_\infty(X) \rtimes_\theta G)_+$ satisfying
\[
  \tau(x) =
  \begin{cases}
    \| \xi \|_{L_2(L_2)}^{2}, & \text{if } x = \pi_l(\xi)^* \pi_l(\xi),\quad\xi \in \mathfrak{B}, \\
    +\infty,       & \text{otherwise}.
  \end{cases}
\]
We refer to \cite[Section 1, Chapter \RN{6}]{MR1943006} and \cite[Section 1, Chapter \RN{10}]{MR1943006} for more information on Hilbert algebras and crossed products.\par 

\section{Transference via cocycles and crossed product embeddings}\label{Sec: embeddeding vNa}
In this section, we establish the transference of Fourier multipliers via cocycles and prove Theorem \ref{Thm: Transference from L(H) to L(G)}. Consider a locally compact second countable unimodular group $G$ and a discrete group $H$. Suppose $\theta$ is a $\mu$-preserving action of $G$ on $(X,\mu)$, and let $\alpha : G \times X \to H$ be a measurable cocycle. We construct a special injective, trace-preserving $*$-homomorphism
\[
\iota : \mathcal{L}(G) \hookrightarrow L_\infty(X) \rtimes_\theta (H \times G),
\]
which encodes the cocycle data and intertwines the Fourier multipliers on $H$ with those on $G$.\par 

 By taking the product of the trivial action of $H$ and the action $\theta$, we obtain an action of $H \times G$ on $L_\infty(X)$, which we continue to denote by $\theta$ without ambiguity. Explicitly, the action is given by
\[
\theta_{(h, g)} f(x) = f(g^{-1} x), \quad \text{for all } (h, g) \in H \times G.
\]
Since $H$ acts trivially, the crossed product admits the decomposition
\[
L_\infty(X) \rtimes_\theta (H \times G) = (L_\infty(X) \rtimes_\theta G) \,\overline{\otimes}\, \mathcal{L}(H).
\]

We now proceed to construct the homomorphism $\iota$ in detail. Given a cocycle $\alpha : G \times X \to H$,  for each $g \in G$, define the measurable sets
\[
A_{h, g} \coloneqq \{x \in X : \alpha(g, x) = h\},\quad h\in H,
\]
so that the space $X$ admits a countable partition into disjoint Borel sets:
\begin{equation}\label{Eqn: partition X}
  X = \bigsqcup_{h \in H} A_{h, g}.
\end{equation}
Since the function $\gamma : X \times H \times G \to H$ given by $\gamma(x, h, g) = \alpha(g, x) h^{-1}$ is measurable, the level set $\gamma^{-1}(e_H)$ is measurable. Consequently, the function
\[
A(h, g, x) \coloneqq \mathbbm{1}_{A_{h, g}}(x),\quad (h,g,x)\in H\times G\times X,
\]
belongs to $L_0(H \times G \times X)$, the space of complex-valued measurable functions on $H \times G \times X$. We define the map
\begin{equation}\label{Def: definition of iota_0}
  \iota_0 : C_c(G) \to L_0(H \times G \times X), \quad \iota_0(f)(h,g,x) = f(x)A(h,g,x).
\end{equation}
In the sequel, consider the Hilbert algebra $C_c(H \times G, L_\infty(X))$ associated with the action $\theta$, and let $\mathfrak{B}$ and $\mathfrak{B}'$ denote the sets of all left- and right-bounded vectors in $L_2(H\times G, L_2(X))$, respectively. The following lemmas will be useful when passing to the operator algebraic setting.
\begin{lemma}\label{Lem: image of C_c(G) under iota_0}
 The image of $\iota_0$ lies in $\mathfrak{B} \cap \mathfrak{B}'$.
\end{lemma}
\begin{proof}
  
To verify that $\iota_0(f) \in \mathfrak{B} \cap \mathfrak{B}'$ for $f \in C_c(G)$, observe that
\[
\begin{split}
\norm{\iota_0(f)}_{L_2(L_2)}^2 &= \sum_{h \in H} \int_G \int_X \abs{f(g) \mathbbm{1}_{A_{h, g}}(x)}^2 \, d\mu(x) \, \dd g \\
&= \sum_{h \in H} \int_G \abs{f(g)}^2 \mu(A_{h, g}) \, \dd g 
= \mu(X) \norm{f}_{L_2(G)}^2,
\end{split}
\]
so $\iota_0(f) \in L_2(L_2)$.\par 

Let $\eta \in L_2(L_2)$. We then compute the left action:
\[
    \begin{split}
      (\iota_0(f)*\eta)(h,g, x) &= \sum_{h_0\in H}\int_{G} (\theta_{(h_0, g_0)}(\iota_0(f)(hh_0,gg_0)) \eta(h_0^{-1},g_0^{-1}))(x) \dd g_0 \\
      &= \sum_{h_0\in H}\int_{G} f(gg_0)\mathbbm{1}_{A_{hh_0,gg_0}}(g_0^{-1}x) \eta(h_0^{-1},g_0^{-1})(x) \dd g_0 \\
      &= \sum_{h_0\in H}\int_{G} f(gg_0)\mathbbm{1}_{A_{hh_0,gg_0}}(g_0^{-1}x) \eta(\alpha(gg_0,g_0^{-1}x)^{-1}h,g_0^{-1})(x) \dd g_0 \\
      &= \int_G f(gg_0) \eta(\alpha(gg_0,g_0^{-1}x)^{-1}h,g_0^{-1})(x) \dd g_0 \\
      &= \int_G f(g_0) \eta(\alpha(g_0,g_0^{-1}gx)^{-1}h,g_0^{-1}g)(x) \dd g_0.
    \end{split}
  \]
Finally, applying the Minkowski inequality yields
\[
    \begin{split}
      \norm{\iota_0(f)*\eta}_{L_2(L_2)} &= \left( \sum_{h\in H}\int_{G\times X}\abs{\int_G f(g_0) \eta(\alpha(g_0,g_0^{-1}gx)^{-1}h,g_0^{-1}g)(x) \dd g_0 }^2 \dd\mu(x)\dd g \right)^\frac{1}{2} \\
      &\leq \int_G \abs{f(g_0)} \left( \sum_{h\in H}\int_{ G\times X} \abs{\eta(\alpha(g_0,g_0^{-1}gx)^{-1}h,g_0^{-1}g)(x)}^2 \dd\mu (x)\dd g \right)^\frac{1}{2} \dd g_0 \\
      &= \int_G \abs{f(g_0)} \left( \sum_{h\in H}\int_{ G\times X} \abs{\eta(h,g)(x)}^2 \dd\mu (x)\dd g \right)^\frac{1}{2} \dd g_0 \\
      &= \norm{\eta}_{L_2(L_2)} \norm{f}_{L_1(G)}.
    \end{split}
  \]
  This proves that $\iota_0(f) \in \mathfrak{B}$ and hence $\iota_0(C_c(G)) \subset \mathfrak{B}$.\par 
Now we verify that $\iota_0(f)\in \mathfrak{B}'$:
\[
    \begin{split}
      (\xi*\iota_0(f))(h,g, x) &= \sum_{h_0\in H}\int_{G} (\theta_{(h_0, g_0)}(\xi(hh_0,gg_0)) (\iota_0(f))(h_0^{-1},g_0^{-1}))(x) \dd g_0 \\
      &= \sum_{h_0\in H}\int_{G} \xi(hh_0,gg_0)(g_0^{-1}x)f(g_0^{-1})\mathbbm{1}_{A_{h_0^{-1},g_0^{-1}}}(x) \dd g_0 \\
      &= \sum_{h_0\in H}\int_{G} \xi(h\alpha(g_0^{-1},x)^{-1},gg_0)(g_0^{-1}x)f(g_0^{-1})\mathbbm{1}_{A_{h_0^{-1},g_0^{-1}}}(x) \dd g_0 \\
      &= \int_{G} \xi(h\alpha(g_0^{-1},x)^{-1},gg_0)(g_0^{-1}x)f(g_0^{-1}) \dd g_0 \\
      &= \int_{G} f(g_0) \xi(h\alpha(g_0,x)^{-1},gg_0^{-1})(g_0x) \dd g_0.
    \end{split}
  \]
Applying the Minkowski inequality again, we get
\[
\begin{split}
\norm{\xi*\iota_0(f)}_{L_2(L_2)} &= \left(\sum_{h\in H} \int_{ G\times X}\abs*{\int_{G}f(g_0) \xi(h\alpha(g_0,x)^{-1},gg_0^{-1})(g_0x) \dd g_0}^2 \dd\mu(x)\dd g \right)^{\frac{1}{2}} \\
      &\leq \int_G \abs{f(g_0)}\left( \sum_{h\in H}\int_{G\times X}\abs{ \xi(h\alpha(g_0,x)^{-1},gg_0^{-1})(g_0x)}^2 \dd \mu(x)\dd g \right)^\frac{1}{2} \dd g_0\\
       &=\int_G \abs{f(g_0)}\left( \sum_{h\in H}\int_{G\times X}\abs{ \xi(h,g)(g_0x)}^2 \dd \mu(x)\dd g \right)^\frac{1}{2} \dd g_0.
\end{split}
\]
Since the action of $G$ on $X$ preserves the measure $\mu$, we conclude
\[
\norm{\xi * \iota_0(f)}_{L_2(L_2)} \leq \norm{\xi}_{L_2(L_2)} \norm{f}_{L_1(G)}.
\]
Therefore, $\iota_0(f) \in \mathfrak{B}'$, and we conclude that $\iota_0(C_c(G)) \subset \mathfrak{B} \cap \mathfrak{B}'$. This completes the proof of Lemma \ref{Lem: image of C_c(G) under iota_0}.
\end{proof}

\begin{lemma}\label{Lem: *-isomorphism from C_c(G)}
	The map $\iota_0:C_c (G) \to \mathfrak{B} \cap \mathfrak{B}'$ is an injective $*$-homomorphism.
\end{lemma}

\begin{proof}
	Let $g \in G$, $h \in H$, and $x \in X$. By the cocycle identity \eqref{Eqn: cocycle identity left}, we have
	\[
	\alpha(g^{-1}, gx) \alpha(g, x) = \alpha(e_G, x) = e_H.
	\]
	This implies that
	\[
	x \in A_{h, g}  \iff  gx \in A_{h^{-1}, g^{-1}},
	\]
	and hence,
	\[
	\mathbbm{1}_{A_{h, g}}(x) = \mathbbm{1}_{A_{h, g}}(x) \mathbbm{1}_{A_{h^{-1}, g^{-1}}}(gx).
	\]
	Now let $f \in C_c(G)$. We compute
	\[
	\begin{split}
		\iota_0(f^\#)(h, g, x) &= f^\#(g) \mathbbm{1}_{A_{h, g}}(x) \\
		&= \overline{f(g^{-1})} \mathbbm{1}_{A_{h,g}}(x) \mathbbm{1}_{A_{h^{-1}, g^{-1}}}(gx) \\
		&= \overline{f(g^{-1})} \mathbbm{1}_{A_{h^{-1}, g^{-1}}}(gx) \mathbbm{1}_{A_{h, g}}(g^{-1} gx) \\
		&= \theta^{-1}_{(h, g)} \left( \overline{f(g^{-1})} \mathbbm{1}_{A_{h^{-1}, g^{-1}}}(\cdot) \mathbbm{1}_{A_{h, g}}(g^{-1} \cdot) \right)(x) \\
		&=  \iota_0(f)^\#(h, g, x),
	\end{split}
	\]
	showing that $ \iota_0(f^\#) =  \iota_0(f)^\#$.
	
	We now verify that $ \iota_0$ respects convolution. For $f_1, f_2 \in C_c(G)$ and $(h,g,x)\in H\times G\times X$, we have
	\[
	\begin{split}
		( \iota_0(f_1)* \iota_0(f_2))(h,g, x) &= \sum_{h_0\in H}\int_{G}(\theta_{(h_0, g_0)}(( \iota_0(f_1))(hh_0,gg_0))( \iota_0(f_2))(h_0^{-1},g_0^{-1}))(x)\dd g_0 \\
		&= \sum_{h_0\in H}\int_{G} f_1(gg_0)\mathbbm{1}_{A_{hh_0,gg_0}}(g_0^{-1}x) f_2(g_0^{-1})\mathbbm{1}_{A_{h_0^{-1},g_0^{-1}}}(x)\dd g_0 \\
		&= \sum_{h_0\in H}\int_{G} f_1(g_0)\mathbbm{1}_{A_{h_0,g_0}}(g_0^{-1}gx) f_2(g_0^{-1}g)\mathbbm{1}_{A_{h_0^{-1}h,g_0^{-1}g}}(x)\dd g_0.
	\end{split}
	\]
	Recall from the definition of the sets $A_{h, g}$ that
	\begin{itemize}
		\item $g_0^{-1} g x \in A_{h_0, g_0} \iff x \in \{y\in X : \alpha(g_0, g_0^{-1}gy) = h_0\}$;
		\item $x \in A_{h_0^{-1} h, g_0^{-1} g} \iff x \in \{y\in X : \alpha(g_0^{-1}g, y) = h_0^{-1}h\}$.
	\end{itemize}
	Therefore, whenever both conditions hold, the cocycle identity \eqref{Eqn: cocycle identity left} yields
	\[
	h = \alpha(g_0, g_0^{-1} g x) \alpha(g_0^{-1} g, x) = \alpha(g, x),
	\]
	implying $x \in A_{h, g}$. Conversely, if $x \in A_{h, g}$ and $x \in A_{h_0^{-1} h, g_0^{-1} g}$, then
	\[
	\alpha(g_0, g_0^{-1} g x) = \alpha(g, x) \alpha(g_0^{-1} g, x)^{-1} = h (h_0^{-1} h)^{-1} = h_0,
	\]
	which implies $g_0^{-1} g x \in A_{h_0, g_0}$. Hence, we conclude
	\[
	x \in A_{h, g} \text{ and } x \in A_{h_0^{-1} h, g_0^{-1} g} \iff g_0^{-1} g x \in A_{h_0, g_0} \text{ and } x \in A_{h_0^{-1} h, g_0^{-1} g},
	\]
	and therefore,
	\[
	\mathbbm{1}_{A_{h_0, g_0}}(g_0^{-1} g x) \mathbbm{1}_{A_{h_0^{-1} h, g_0^{-1} g}}(x) = \mathbbm{1}_{A_{h_0^{-1} h, g_0^{-1} g}}(x) \mathbbm{1}_{A_{h, g}}(x).
	\]
	This leads to:\[
	\begin{split}
		( \iota_0(f_1)* \iota_0(f_2))(h,g, x) &= \sum_{h_0\in H}\int_{G} f_1(g_0)\mathbbm{1}_{A_{h_0^{-1}h,g_0^{-1}g}}(x)f_2(g_0^{-1}g)\mathbbm{1}_{A_{h,g}}(x)\dd g_0 \\
		&= \mathbbm{1}_{A_{h,g}}(x)\mathbbm{1}_{X}(x)\int_{G} f_1(g_0)f_2(g_0^{-1}g)\dd g_0 \\
		&= f_1*f_2(g)\mathbbm{1}_{A_{h,g}}(x) = ( \iota_0(f_1*f_2))(h,g, x),
	\end{split}
	\]
	proving that $ \iota_0$ is a $*$-homomorphism.\par 
	
	The injectivity of $\iota_0$ follows from the disjoint partition of $X$ in \eqref{Eqn: partition X}, and the linearity is immediate. This completes the proof  of Lemma \ref{Lem: *-isomorphism from C_c(G)}.
\end{proof}

We next verify that the map $\pi_l$ respects the $*$-structure and is faithful on $\mathfrak{B} \cap \mathfrak{B}'$. The following lemma is standard and we outline its proof for the convenience of the reader.

\begin{lemma}\label{Lem: *-representation of B cap B'}
The restriction of the map $\pi_l$ on $\mathfrak{B} \cap \mathfrak{B}'$ is a faithful $*$-homomorphism.
\end{lemma}

\begin{proof}
Let $a_1, a_2 \in \mathfrak{B}$. By \cite[Lemma 1.8', Chapter \RN{6}]{MR1943006}, we have
\[
\pi_l(a_1 a_2) = \pi_l(\pi_l(a_1) a_2) = \pi_l(a_1) \pi_l(a_2).
\]

Let $a \in \mathfrak{B} \cap \mathfrak{B}'$. By the density of $C_c(H \times G, L_\infty(X))$ in the $L_2(L_2)$ norm and applying \cite[Lemma 1.9, Chapter \RN{6}]{MR1943006}, we obtain that for any $\eta, \xi \in L_2(L_2)$,
\[
\begin{split}
\inner{\pi_l(a^\#) \eta, \xi}=\inner{\pi_l(a)^* \eta, \xi}.
\end{split}
\]
Therefore, $\pi_l(a^\#) = \pi_l(a)^*$, which shows that $\pi_l$ is a $*$-homomorphism. Its faithfulness is immediate from the definition.
\end{proof}

Now we are ready to establish the desired embedding of $\mathcal{L}(G)$ into $L_\infty(X) \rtimes_\theta (H \times G)$. Consider 
\begin{equation}\label{Def: definition of iota}
  \iota:\lambda(C_c(G)) \to L_\infty(X) \rtimes_\theta (H \times G),\quad  \lambda(f) \mapsto \pi_l \circ \iota_0(f).
\end{equation}
This map is well-defined by Lemma~\ref{Lem: image of C_c(G) under iota_0}.

\begin{proposition}\label{Prop: *-representation from LG to crossed product}
The map $ \iota$ extends to a trace-preserving, weak*-continuous, injective $*$-homomorphism from $\mathcal{L}(G)$ to $L_\infty(X) \rtimes_\theta (H \times G)$. 
\end{proposition}

\begin{proof}
By Lemmas \ref{Lem: *-isomorphism from C_c(G)} and \ref{Lem: *-representation of B cap B'}, $\iota$ is an injective $*$-homomorphism on $\lambda(C_c(G))$
 since $\pi_l:\mathfrak{B}\cap\mathfrak{B}'\to L_\infty(X) \rtimes_\theta (H \times G)$ and $\iota_0:C_c(G)\to \mathfrak{B}\cap\mathfrak{B}'$ are all injective $*$-homomorphisms.\par 

Now, without loss of generality, assume that $\mu(X) = 1$. From the definition of the trace $\tau$ and Lemmas \ref{Lem: *-isomorphism from C_c(G)} and \ref{Lem: image of C_c(G) under iota_0}, we find
\[
    \begin{split}
      \tau( \iota(|\lambda(f)|^2))=&\tau(\pi_l(\iota_0(f))^*\pi_l(\iota_0(f)))=\norm{\iota_0(f)}^2\\
      =&\sum_{h\in H}\int_{G}\int_X\overline{\iota_0(f)(h,g, x)}\iota_0(f)(h,g, x)\dd\mu(x)\dd g\\
      =&\sum_{h\in H}\int_{G}\int_X\overline{f(g)}\mathbbm{1}_{A_{h,g}}(x)f(g)\mathbbm{1}_{A_{h,g}}(x)\dd\mu(x)\dd g\\
      =&\sum_{h\in H}\int_{G}\overline{f(g)}f(g)\mu(A_{h,g})\dd g\\
      =&\int_{G}\overline{f(g)}f(g)\dd g=\tau_G(|\lambda(f)|^2)=\tau_G(|\lambda(f)|^2).
    \end{split}
  \]
Hence, $ \iota$ preserves the trace on $\lambda(C_c(G))$. Then, by a standard argument, we conclude that $ \iota$ extends to a trace-preserving, weak*-continuous, injective $*$-homomorphism from $ \mathcal{L}(G)$ to $L_\infty(X) \rtimes_\theta (H \times G)$.
\end{proof}\par

The following key result describes how symbols on $H$ can be transferred to symbols on $G$ through the cocycle $\alpha$. In the following, let $\mathbb{E}=\iota^*$ be the conditional expectation from $(L^\infty(X)\rtimes G)\,\overline{\otimes}\,\mathcal{L}(H)$ onto $\mathcal{L}(G)$. Note that our construction of $\iota$ and $\mathbb E$ relies heavily on the discreteness of $H$, which also leads to the same discreteness assumptions in the following theorem as well as in Theorem~\ref{Thm: Transference from L(H) to L(G)}. Recall that for a symbol $m \in \ell_\infty(H)$, we consider the transferred symbol $\widetilde{m} \in L_\infty(G)$ given by
\[
\widetilde{m}(g) = \frac{1}{\mu(X)} \int_X m(\alpha(g, x)) \, d\mu(x), \quad g \in G.
\]
\begin{theorem}\label{Thm: multiplier diagram}
The following diagram commutes:
\begin{center}
\begin{tikzcd}[column sep={0.5em}]
\lambda(C_c(G)) \arrow[r, "\iota"] \arrow[d, "T_{\widetilde{m}}"'] &[20] \pi_l(\mathfrak{B}) \arrow[d, "\mathrm{id} \otimes T_m"] \arrow[r,phantom,"\subset"] & {(L^\infty(X)\rtimes G)\,\overline{\otimes}\,\mathcal{L}(H)} \\
\mathcal{L}(G)                                                      &[20] \pi_l(\mathfrak{B}) \arrow[l, "\mathbb{E}"] \arrow[r,phantom,"\subset"]              & {(L^\infty(X)\rtimes G)\,\overline{\otimes}\,\mathcal{L}(H)}
\end{tikzcd}
\end{center}
i.e., the identity $T_{\widetilde{m}}=\mathbb{E}\circ(\id \otimes T_m) \circ  \iota$ holds on $\lambda(C_c(G))$.
\end{theorem}

To connect the Fourier multiplier $T_m$ on $\mathcal{L}(H)$ with an operator acting on the image of $\iota_0$ in $L_2(H \times G, L_2(X))$, we aim to construct an intertwining operator that models the action of $T_m$ at the level of functions. To this end, we define a pointwise multiplication operator that encodes the symbol $m$. Given a bounded function $m \in \ell_\infty(H)$, define $\widetilde{T_m}$ on $L_2(H \times G, L_2(X))$ by
\[
(\widetilde{T_m}(\xi))(h, g) \coloneqq m(h) \cdot \xi(h, g), \quad \text{for all } \xi \in L_2(H \times G, L_2(X)).
\]

By construction, this operator intertwines the action of $T_m$ with the homomorphism $\pi_l$ via the identity
\[
\pi_l \circ \widetilde{T_m} = (\id \otimes T_m) \circ \pi_l,
\]
which allows us to transfer the action of $T_m$ through the crossed product representation. The following lemma ensures that $\widetilde{T_m}$ maps the image of $\iota_0$ into the domain of $\pi_l$, so that both sides of the identity are well-defined.

\begin{lemma}\label{Lem: Fourier multiplier on L2(L2)}
For any $m \in \ell_\infty(H)$, we have $\widetilde{T_m}(\iota_0(C_c(G))) \subset \mathfrak{B}$. 
\end{lemma}

\begin{proof}
Let $f \in C_c(G)$. Recall from the earlier construction that
\[
\iota_0(f)(h, g)(x) = f(g) \mathbbm{1}_{A_{h, g}}(x).
\]
Then,
\[
(\widetilde{T_m} \iota_0(f))(h, g)(x) = f(g) \mathbbm{1}_{A_{h, g}}(x) m(h) = f(g) m(\alpha(g, x)) \mathbbm{1}_{A_{h, g}}(x),
\]
since $h = \alpha(g, x)$ for $x \in A_{h, g}$.

We now verify that $\widetilde{T_m}(\iota_0(f)) \in \mathfrak{B}$. For any $\eta \in L_2(L_2)$, we consider the convolution:
\[
    \begin{split}
      &(\widetilde{T_m}(\iota_0(f))*\eta)(h,g)(x)\\
      =&\sum_{h_0\in H}\int_{G} (\theta_{(h_0, g_0)}(\widetilde{T_m}(\iota_0(f))(hh_0,gg_0)) \eta(h_0^{-1},g_0^{-1}))(x)\dd g_0\\
      =&\sum_{h_0\in H}\int_{G} f(gg_0)\mathbbm{1}_{A_{hh_0,gg_0}}(g_0^{-1}x)m(\alpha(gg_0,g_0^{-1}x)) \eta(h_0^{-1},g_0^{-1})(x)\dd g_0\\
      =&\sum_{h_0\in H}\int_{G} f(gg_0)\mathbbm{1}_{A_{hh_0,gg_0}}(g_0^{-1}x)m(\alpha(gg_0,g_0^{-1}x)) \eta(\alpha(gg_0,g_0^{-1}x)^{-1}h,g_0^{-1})(x)\dd g_0\\
      =&\int_G f(gg_0)m(\alpha(gg_0,g_0^{-1}x)) \eta(\alpha(gg_0,g_0^{-1}x)^{-1}h,g_0^{-1})(x)\dd g_0\\
      =&\int_G f(g_0)m(\alpha(g_0,g_0^{-1}gx)) \eta(\alpha(g_0,g_0^{-1}gx)^{-1}h,g_0^{-1}g)(x)\dd g_0,
    \end{split}
  \]
  and hence
  \[
    \begin{split}
      &\norm{\widetilde{T_m}(\iota_0(f))*\eta}_{L_2(L_2)}\\
      =&\left( \sum_{h\in H}\int_{G\times X}\abs*{\int_G f(g_0)m(\alpha(g_0,g_0^{-1}gx)) \eta(\alpha(g_0,g_0^{-1}gx)^{-1}h,g_0^{-1}g)(x)\dd g_0 }^2\dd\mu(x)\dd g \right)^\frac{1}{2}\\
      \leq &\int_{G}\abs{f(g_0)}\left( \sum_{h\in H}\int_{ G\times X} \abs{m(\alpha(g_0,g_0^{-1}gx))\eta(\alpha(g_0,g_0^{-1}gx)^{-1}h,g_0^{-1}g)(x)}^2\dd\mu (x)\dd g \right)^\frac{1}{2}\dd g_0\\
      \leq &\norm{m}_{\ell_\infty(H)}\int_{G}\abs{f(g_0)}\left( \sum_{h\in H}\int_{G\times X} \abs{\eta(\alpha(g_0,g_0^{-1}gx)^{-1}h,g_0^{-1}g)(x)}^2\dd\mu (x)\dd g \right)^\frac{1}{2}\dd g_0\\
      =&\norm{m}_{\ell_\infty(H)}\int_{G}\abs{f(g_0)}\left( \sum_{h\in H}\int_{ G\times X} \abs{\eta(h,g)(x)}^2\dd\mu (x)\dd g \right)^\frac{1}{2}\dd g_0\\
      =& \norm{m}_{\ell_\infty(H)}\norm{\eta}_{L_2(L_2)} \norm{f}_{L_1(G)}.
    \end{split}
  \]
So $\widetilde{T_m}(\iota_0(f)) \in \mathfrak{B}$.
\end{proof}

With this preparatory result in place, we are now ready to verify the commutativity of the diagram stated in Theorem~\ref{Thm: multiplier diagram}, by explicitly computing the image of a Fourier multiplier under the composed map $\mathbb{E} \circ (\id \otimes T_m) \circ \iota$.

\begin{proof}[Proof of Theorem \ref{Thm: multiplier diagram}]
For any $f \in C_c(G)$ and $\xi \in \mathfrak{B}$, using the intertwining identity $\iota \circ \lambda = \pi_l \circ \iota_0$, we obtain
\[
  \begin{split}
   \tau_G(\lambda (f)^*\mathbb{E}(\pi_l(\xi)))= &   \tau( \iota(\lambda(f))^*\pi_l(\xi)) \\
=& \tau({\pi_l}(\iota_0(f))^*\pi_l(\xi))\\
   = & \sum_{h\in H}\frac{1}{\mu(X)}\int_{G\times X}\overline{ \iota_0 (f)}\xi (h,g,x) \dd\mu(x) \dd g\\
    =& \sum_{h\in H}\frac{1}{\mu(X)}\int_{G\times X}\overline{f(g)}\mathbbm{1}_{A_{h,g}}(x)\xi(h,g, x) \dd\mu(x)   \dd g \\
    =& \frac{1}{\mu(X)}\int_G \int_{X}\overline{f(g)}\xi(\alpha(g,x),g, x) \dd\mu(x)  \dd g\\
    =&\tau_G\left( \lambda(f)^*\lambda\left( \int_{X}\xi(\alpha(\cdot,x),\cdot, x) \dd\mu(x) \right) \right)
  \end{split}
\]
By duality, we can explicitly write
\[
  \mathbb{E}(\pi_l(\xi))=\frac{1}{\mu(X)}\int_G \int_{X}\xi(\alpha(g,x),g, x) \dd\mu(x) \lambda_g \dd g.
\]

Now recall that the crossed product decomposes as $L_\infty(X)\rtimes_\theta (H\times G)=(L^\infty(X) \rtimes_\theta G) \,\overline{\otimes}\, \mathcal{L}(H)$, so we may interpret $\id \otimes T_m$ as a map on the entire crossed product, where $\id$ denotes the identity on $L_\infty(X) \rtimes_\theta G$. By Lemma~\ref{Lem: Fourier multiplier on L2(L2)}, we have $\widetilde{T_m}(\iota_0(f)) \in \mathfrak{B}$. Combining this with the intertwining identities $\iota \circ \lambda = \pi_l \circ \iota_0$ and $\pi_l \circ \widetilde{T_m} = (\id \otimes T_m) \circ \pi_l$ again we conclude
 \[
  (\id \otimes T_m)( \iota(\lambda(f))) = \pi_l(\widetilde{T_m}(\iota_0(f))).
 \]
Applying $\mathbb{E}$ to both sides, we compute
\[
  \begin{split}
    &\mathbb{E}((\id\otimes T_m)( \iota(\lambda(f)))) \\
    =&\frac{1}{\mu(X)}\int_G \left( \int_{X}f(g)\mathbbm{1}_{A_{\alpha(g,x),g}}(x)m(\alpha(g,x))\dd\mu(x)\right)\lambda_{g}  \dd g\\
    =&\frac{1}{\mu(X)}\int_G f(g)\left( \int_{X}m(\alpha(g,x))\dd\mu(x)\right)\lambda_{g}  \dd g.
  \end{split}
\]
Define the transferred symbol $\widetilde{m} \in L_\infty(G)$ by
\[
\widetilde{m}(g) = \frac{1}{\mu(X)} \int_X m(\alpha(g, x)) \, d\mu(x).
\]
Then the above identity shows that
\[
\mathbb{E}\circ(\id \otimes T_m)( \iota(\lambda(f))) = T_{\widetilde{m}}(\lambda(f)),
\]
which yields the desired commutative diagram.
\end{proof}

Before proving Theorem \ref{Thm: Transference from L(H) to L(G)}, we recall the definition of QWEP. A $C^*$-algebra $A \subset B(H)$ is said to have Lance's weak expectation property (WEP) if there exists a unital completely positive map $\Phi:B(H)\to A^{**}$ such that $\Phi(a)=a$ for all $a\in A$. A $C^*$-algebra $A$ is said to be QWEP if it is a quotient of a $C^*$-algebra with the WEP. If a discrete group $G$ is hyperlinear, then the group von Neumann algebra $\mathcal{L}(G)$ is QWEP. We refer the reader to \cite{MR2072092, MR2391387} for further information on QWEP and hyperlinear. It was shown by Junge in an unpublished work that if a von Neumann algebra $\mathcal{M}$ is QWEP, then for any von Neumann algebra $\mathcal{N}$ and any completely bounded map $T$ on $L_p(\mathcal{N})$, the tensor product map satisfies
\begin{equation} \label{eq: QWEP}
 \| \id_\mathcal{M} \otimes T : L_p(\mathcal{M} \bar\otimes \mathcal{N}) \rightarrow L_p(\mathcal{M} \bar\otimes \mathcal{N}) \| \leq \norm{T}_{{\mathrm{cb}(L_p( \mathcal{N}))}}.
\end{equation}

By the complete $L_p$ contractivity of both $\mathbb{E}$ and $\iota$, Proposition~\ref{Prop: *-representation from LG to crossed product} and Theorem~\ref{Thm: multiplier diagram} immediately imply the following result.
\begin{theorem}\label{Thm: Transference from crossed product to L(G)}
	Let $1\leq p\leq \infty$. Let $H$ be a discrete group, $G$ be a unimodular locally compact second countable  group and $(X,\mu)$ be a finite measure space. Let $\theta$ be a $\mu$-preserving action of $G$ on $ (X,\mu)$. Let $\alpha:G\times X\to H$ and $m, \widetilde{m}$ be as above. Then we have
	\[
	\norm{T_{\widetilde{m}}:L_p(\mathcal{L}(G))\to L_p(\mathcal{L}(G))}_{\mathrm{cb}}\leq \norm{\id\otimes T_{m}:L_p( (L_\infty(X) \rtimes_\theta G) \,\overline{\otimes}\, \mathcal{L}(H))\to L_p( (L_\infty(X) \rtimes_\theta G) \,\overline{\otimes}\, \mathcal{L}(H))}_{\mathrm{cb}}.
	\]
\end{theorem}

Now Theorem~\ref{Thm: Transference from L(H) to L(G)} is in reach.
\begin{proof}[Proof of Theorem \ref{Thm: Transference from L(H) to L(G)}]
Since $L_\infty(X) \rtimes_\theta G$ is QWEP and $T_m$ is a completely bounded Fourier multiplier on $L_p(\mathcal{L}(H))$, we may apply inequality \eqref{eq: QWEP} to the operator $\id\otimes T_m$ appearing in Theorem~\ref{Thm: Transference from crossed product to L(G)} to deduce that 
\[
\norm{T_{\widetilde{m}}}_{\mathrm{cb}(L_p(\mathcal{L}(G)))} \leq \norm{T_m}_{\mathrm{cb}(L_p(\mathcal{L}(H)))}.
\]
This completes the proof of the theorem.
\end{proof}
\begin{remark}\label{Rmk: main Theorem for left cocycle} 

The results in this section may also be adapted for right cocycles, that is, maps $\beta : X \times G \to H$ satisfying the right cocycle identity
\[
\beta(x, g_1 g_2) = \beta(x, g_1) \beta(xg_1, g_2).
\]
In particular, given a symbol $m \in \ell_\infty(H)$, we may consider the associated symbol $\widetilde{m} \in L_\infty(G)$ by  
\[
\widetilde{m}(g) = \frac{1}{\mu(X)} \int_X m(\beta(x, g)) \, \dd\mu(x), \quad g \in G.
\]
Proposition~\ref{Prop: *-representation from LG to crossed product}, Theorem~\ref{Thm: multiplier diagram}, and Theorem~\ref{Thm: Transference from crossed product to L(G)} remain valid when the cocycle $\alpha$ is replaced by the right cocycle $\beta$.

\end{remark}
\begin{remark}
The transference method developed in this section extends beyond noncommutative $L_p$-spaces and applies equally to fully symmetric function spaces $E$, such as Marcinkiewicz, Orlicz, and most Lorentz spaces. Indeed, in this broader setting, the conditional expectation
\[
\mathbb{E} : E\big((L^\infty(X) \rtimes_\theta G) \,\overline{\otimes}\, \mathcal{L}(H)) \to E(\mathcal{L}(G))
\]
is still completely contractive, as established via interpolation techniques; see  \cite[Theorem 3.4]{MR1188788}.
\end{remark}
\begin{remark}
  The similar transference of Schur multipliers from $H$ to $G$ is more direct.  For a measure space $\Omega$ and a function $\psi$ on $\Omega\times \Omega$, we may consider the Schur multiplier map $M_\psi$ on the Schatten class $S_p (L_2 (\Omega))$ and we denote the corresponding completely bounded norm by $\norm{M_{\psi}}_{\mathrm{cb}(S^p(L_2(\Omega)))}$; see \cite{MR2838352}.
  Let $(X,\mu)$ be a probability measure space and suppose $\alpha: G\times X\to H$ is a measurable map. Given $\varphi$ be a bounded continuous function on $H\times H$, define the transferred symbol
\begin{equation}\label{Eqn: cocycle transfer}
  \widetilde{\varphi}(g_1,g_2)=\int_X \varphi(\alpha(g_1,x),\alpha(g_2,x))\dd\mu(x).
\end{equation}
Then we have the norm estimate:
\[
\norm{M_{\widetilde{\varphi}}}_{\mathrm{cb}(S^p(L_2(G)))}\leq \norm{M_{\varphi}}_{\mathrm{cb}(S^p(L_2(H)))}.
\]
Indeed, for each $x\in X$, define the subset $K_x\subset H\times G$ by
\[
	K_x=\{(h,g)\in H\times G:h=\alpha(g,x)\},
\]
and the embedding
\[
    \iota_x:G\times G\hookrightarrow K_x\times K_x\subset (H\times G)\times (H\times G),\quad (g_1,g_2)\mapsto ((\alpha(g_1,x),g_1),(\alpha(g_2,x),g_2)).
  \]
   Also define $\varphi'\in C((H\times G)\times (H\times G))$ by
  \[
    \varphi'((h_1,g_1),(h_2,g_2))=\varphi(h_1,h_2).
  \]
  Then the transferred symbol can be written as 
  \[
    \widetilde{\varphi}=\int_X \varphi'\circ \iota_x\dd\mu(x).
  \]
  By \cite[Theorem 1.19]{MR2838352} we obtain
  \[
    \begin{split}
      \norm{M_{\widetilde{\varphi}}}_{\mathrm{cb}(S^p(L_2(G)))}\leq &\int_X\norm{M_{\varphi'\circ \iota_x}}_{\mathrm{cb}(S^p(L_2(G)))}\dd\mu(x)\\
      =&\int_X\norm{M_{\varphi'|_{(K_x\times K_x)}}}_{\mathrm{cb}(S^p(L_2(K_x)))}\dd\mu(x)\\
      \leq &\int_X\norm{M_{\varphi'}}_{\mathrm{cb}(S^p(L_2(H\times G)))}\dd\mu(x)\\
      =&\norm{M_{\varphi'}}_{\mathrm{cb}(S^p(L_2(H\times G)))}.
    \end{split}
  \]
  Finally, by \cite[Remark 1.6]{MR2838352} we have
\[
\norm{M_{\varphi'}}_{\mathrm{cb}(S^p(L_2(H\times G)))}=\norm{M_{\varphi}}_{\mathrm{cb}(S^p(L_2(H)))}.
\]
\end{remark}

\section{Applications}\label{Sec: applications}

We now present several applications of Theorem~\ref{Thm: Transference from L(H) to L(G)}, illustrating how transference via cocycles connects geometric and analytic properties between different groups. These include results on measure equivalence, maximal inequalities for amenable groups, a noncommutative version of the Jodeit theorem, and an explicit case study involving $SL_2(\mathbb{R})$.\par 

Throughout this section, in order to ensure that the transference map in~\eqref{eqn: formula of tilde m} sends the Fourier algebra $A(H)$ to the Fourier algebra $A(G)$, we require the cocycles to satisfy a properness condition. This leads to the following definition.\par 

\begin{definition}[Proper cocycle]\label{Def: proper cocycle}
  Let $G$, $H$ be locally compact second countable groups and let $G$ act on a standard probability space $(X,\mu)$ and assume the action is Borel. 
  \begin{enumerate}[(1)]
    \item A cocycle $\alpha:G\times X\to H$ is said to be \emph{proper with respect to a family $\mathcal{A}$} of Borel sets in $X$ if 
    \begin{itemize}
      \item for every compact subset $K\subset G$ and every $A\in \mathcal{A}$,  there is a precompact set $L(K,A)\subset H$ such that, for every $g\in K$ we get that  $\alpha(g,x)\in L(K,A)$ for almost all $x\in A\cap g^{-1}A$;
      \item for every compact subset $L\subset H$ and every $A\in \mathcal{A}$, we get that the set $K(L,A)$ of all $g\in G$ such that $\mu\{x\in X:a\in A\cap g^{-1}A,\alpha(g,x)\in L\}>0$, is precompact in G.
    \end{itemize}
    \item  A family $\mathcal{A}$ of Borel sets in $X$ is said to be \emph{large} if it is closed under finite unions and under taking Borel subsets, and if for every $\varepsilon>0$ there is a set $A\in \mathcal{A}$ such that $\mu(X\setminus A)<\varepsilon$.
    \item A cocycle $\alpha:G\times X\to H$ is said to be \emph{proper} if $\alpha$ is proper with respect to some large family $\mathcal{A}$.
  \end{enumerate}
\end{definition}\par

The above definition of a proper cocycle is adapted from \cite{deprez2014permanence,MR1756981,MR3310701} for technical convenience. It is straightforward to see that the above definition is weaker than those given in the aforementioned works; that is, properness in the sense of \cite{deprez2014permanence,MR1756981,MR3310701} implies properness in our sense.\par

We now recall the following lemma from \cite[Lemma 2.1]{MR3310701}. In Assertion (2), we replace the notion of properness used in \cite{MR3310701} with the weaker one defined above. This modification does not affect the validity of the lemma, as the proof that $\widetilde{m} \in A(G)$ in \cite[Lemma 2.1]{MR3310701} relies solely on the existence of a large family $\mathcal{A}$.

\begin{lemma}[{\cite[Lemma 2.1]{MR3310701}}]\label{lem: A to A}
Let $G$ and $H$ be locally compact second countable groups. Suppose $G$ acts in a Borel measurable way on a standard probability space $(X, \mu)$, and let $\alpha: G \times X \to H$ be a measurable cocycle. Then:
\begin{enumerate}
  \item If a symbol $m\in L_\infty(H)$ is  positive definite, then so is $\widetilde{m}\in L_\infty(G)$.
  \item If $\alpha$ is proper and $m \in A(H)$, then $\widetilde{m} \in A(G)$.
\end{enumerate}
\end{lemma}
The following lemma will be used frequently in this section.
\begin{lemma}\label{lem: uniformly convergence to uniformly convergence}
  Let $G$, $H$, $(X,\mu)$ be as in the above lemma. Let $\alpha: G \times X \to H$ be a measurable proper cocycle and $(m_n)_{n\in\mathbb{N}}$ be a sequence of symbols in $L_\infty(H)$. If $(m_n)_{n\in \mathbb{N}}$ is uniformly bounded and $m_n\to 1$ uniformly on compact subsets of $H$, then $(\widetilde{m}_n)_{n\in \mathbb{N}}$ is uniformly bounded and $\widetilde{m}_n\to 1$ uniformly on compact subsets of $G$.
\end{lemma}
\begin{proof}
   Fix a compact set $K \subset G$ and $\varepsilon > 0$. Since $\alpha$ is a proper cocycle, there exists $A \in \mathcal{A}$ such that $\mu(A) > 1 - \varepsilon$ and the set
\[
\left\{ \alpha(g, x) :g\in K, x \in A \cap g^{-1} A \right\}
\]
is precompact. Then we estimate
\[
  \begin{split}
    \sup_{g\in K}\abs{\widetilde{m}_n(g)-1}=&\sup_{g\in K}\abs{\int_{X}m_n(\alpha(g,x))\dd\mu(x)-1}\\
    \leq &\sup_{g\in K}\abs{\int_{X\setminus A\cap g^{-1}A}(m_n(\alpha(g,x))-1)\dd\mu(x)}+\sup_{g\in K}\abs{\int_{A\cap g^{-1}A}(m_n(\alpha(g,x))-1)\dd\mu(x)}\\
    \leq &\, 2\varepsilon\norm{m_n-1}_{\infty}+\sup_{g\in K}\abs{\int_{A\cap g^{-1}A}(m_n(\alpha(g,x))-1)\dd\mu(x)}.
  \end{split}
\]
Since $m_n \to 1$ uniformly on compact subsets of $H$ and the range of $\alpha(K, x)$ over $A \cap g^{-1}A$ is precompact, the second term tends to $0$ as $n \to \infty$. Therefore,
\[
 \sup_{g\in K}\abs{\widetilde{m}_n(g)-1}\leq 2\varepsilon\norm{m_n-1}_{\infty}+\varepsilon
\]  
for all sufficiently large $n$. Since $\varepsilon > 0$ was arbitrary, it follows that $\widetilde{m}_n \to 1$ uniformly on $K$. \par 
\end{proof}

\subsection{Measure equivalence of groups}\label{Subsec: Measure equivalence of groups}

We begin with an application of Theorem~\ref{Thm: Transference from L(H) to L(G)} to measure equivalence of discrete groups. In particular, we prove that $L_p$-weak amenability is invariant under measure equivalence, assuming a natural QWEP condition.\par 
We say that two countable discrete groups $\Gamma$ and $\Delta$ are \emph{measure equivalent} (abbreviated as \emph{ME}) if there exists a standard infinite measure space $(\Sigma, \sigma)$ equipped with commuting, measure-preserving, free actions of $\Gamma$ and $\Delta$, such that both actions admit fundamental domains of finite measure. In this case, we call $(\Sigma, \sigma)$ an \emph{ME coupling} of $\Gamma$ and $\Delta$. For notational convenience, we denote the $\Gamma$-action on the left and the $\Delta$-action on the right.\par 
For countable discrete groups that are ME, one can construct Borel cocycles from one group to the other. A detailed construction appears in \cite[p.~1064]{MR1740986}. For the reader's convenience, we now outline this construction: Let $(\Sigma, \sigma)$ be an ME coupling of two countable groups $\Gamma$ and $\Delta$. Choose Borel cross-sections from the quotient spaces $\Gamma \backslash \Sigma$ and $\Sigma / \Delta$ into $\Sigma$, and let $Y$ and $X$ denote their respective ranges so that
\[
\Sigma = \bigsqcup_{\delta \in \Delta} X \delta = \bigsqcup_{\gamma \in \Gamma} \gamma Y,
\]
where $(X, \mu)$ and $(Y, \nu)$ are standard measure spaces, with $\mu = \sigma|_X$ and $\nu = \sigma|_Y$. We refer to $(X,\mu)$ and $(Y,\nu)$ as fundamental domains for the actions $\Sigma \curvearrowleft \Delta$ and $\Gamma \curvearrowright \Sigma$, respectively. Given $\gamma \in \Gamma$ and $x \in X$, there exist unique elements $x(\gamma, x) \in X$ and $\alpha(\gamma, x) \in \Delta$ such that
\begin{equation}\label{Eqn: construction of cocycle alpha}
\gamma x = x(\gamma, x) \alpha(\gamma, x).
\end{equation}
The map $\alpha:\Gamma\times X\to \Delta$ satisfies the (left) cocycle identity:
\begin{equation}\label{Eqn: cocycle for ME}
\alpha(\gamma_1 \gamma_2, x) = \alpha(\gamma_1, \gamma_2 \cdot x)  \alpha(\gamma_2, x)
\end{equation} 
with the natural (left) $\mu$-preserving action of $\Gamma$ on $X$ given by $\gamma \cdot x \coloneqq x(\gamma, x)$. Similarly, there is a map $\beta : Y \times \Delta \to \Gamma$ satisfies the (right) cocycle identity:
\begin{equation}\label{Eqn: cocycle for ME, right}
\beta(y, \delta_1 \delta_2) = \beta(y, \delta_1) \beta(y \cdot \delta_1, \delta_2),
\end{equation}

\begin{proof}[Proof of Corollary~\ref{Coro: transferred multiplier of two lattices in non-amenable groups}]
Without loss of generality, assume that $\Delta$ is weakly $L_p$-amenable. That is, there exists a sequence of multipliers $(m_n)_{n \geq 0} \subset A(\Delta)$ such that $m_n \to 1$ pointwise, and
\[
\norm{T_{m_n}(f)}_p\leq_{\mathrm{cb}} C_p\norm{f}_p, \quad \text{for all } f \in L_p(\mathcal{L}(\Delta)).
\]

Let $(X, \mu)$, $\alpha$ be as above. Assume without loss of generality that $\mu(X) = 1$. By Theorem~\ref{Thm: Transference from L(H) to L(G)}, the transferred symbol
\[
\widetilde{m}_n(\gamma) = \int_X m_n(\alpha(\gamma, x)) \, \dd\mu(x)
\]
defines a bounded Fourier multiplier on $L_p(\mathcal{L}(\Gamma))$ with the estimate
\[
\norm{T_{\widetilde{m}_n}(f)}_p\leq_{\mathrm{cb}} C_p\norm{f}_p, \quad \text{for all } f \in L_p(\mathcal{L}(\Gamma)).
\]
Since $\Gamma$ and $\Delta$ are discrete and measure equivalent, it follows from \cite[Lemma 3.1]{MR3310701} that the associated cocycle $\alpha$ is proper in the sense of \cite{MR3310701}, and hence also proper in the sense of Definition~\ref{Def: proper cocycle}.
 By Lemmas \ref{lem: A to A} and \ref{lem: uniformly convergence to uniformly convergence}, we also know that for each $n \in \mathbb{N}$, the function $\widetilde{m}_n$ is positive definite, belongs to $A(\Gamma)$ and converges to $1$ pointwise as $n\to \infty$. Hence, $\Gamma$ is weakly $L_p$-amenable, and we have
\[
\Lambda_{p, \mathrm{cb}}(\Gamma) \leq \Lambda_{p, \mathrm{cb}}(\Delta).
\]
Since ME is an equivalence relation, the reverse inequality also holds. 
\end{proof}
The verification of the QWEP assumption for the crossed products (or equivalently the hyperlinearity of the associated actions in the sense of \cite{MR2826401,MR2566316}) in Corollary~\ref{Coro: transferred multiplier of two lattices in non-amenable groups} is sometimes nontrivial and goes beyond the scope of this paper. We refer the reader to \cite{MR2826401,MR2566316,gao2024soficitygroupactionssets,gao2024actionslerfgroupssets} for related discussion in this direction.

\subsection{Pointwise convergence of noncommutative Fourier series}
In this part, we apply our transference method to give a simplified proof of the maximal inequality for Fourier series summation on locally compact unimodular amenable groups established in \cite{MR4813921}, and moreover improve the constant in the inequality.\par 

We recall the following analogue for the noncommutative setting of the usual almost everywhere convergence. This is the notion of almost uniform convergence introduced by \cite{MR428060}.\par

\begin{definition}\label{Def: b.a.u and a.u}
  Let $(x_n)_{n\in \mathbb{N}}$ be a sequence in $L_0(\mathcal{M})$. We say $(x_n)_{n\in \mathbb{N}}$ converges \emph{almost uniformly} (\emph{a.u.}) to $x\in L_0(\mathcal{M})$ if for any $\varepsilon>0$, there is a projection $e\in \mathcal{M}$ s.t. 
  \[
  \tau(e^\perp)< \varepsilon\quad\text{and}\quad  \lim_{n\to \infty}\norm{(x-x_n)e}_\infty=0.
  \]
  We say $(x_n)_n$ converges \emph{bilaterally almost uniformly} (\emph{b.a.u.}) to $x\in L_0(\mathcal{M})$ if for any $\varepsilon>0$, there is a projection $e\in \mathcal{M}$ s.t. 
  \[
  \tau(e^\perp)< \varepsilon\quad\text{and}\quad  \lim_{n\to \infty}\norm{e(x-x_n)e}_\infty=0.
  \]
\end{definition}
Note that in the commutative case, both notions are equivalent to the usual almost everywhere convergence in terms of Egorov's Theorem in the case of probability space.\par

\begin{corollary}\label{Cor: CPAP for amenable groups}
Let $G$ be a unimodular amenable locally compact second countable group. Then there exists a sequence of unital positive definite functions $(m_n)_{n \geq 0} \subset A(G)$ such that $m_n \to 1$ uniformly on compact sets and for all $f \in L_p(\mathcal{L}(G))$ with $1 < p <\infty$, we have
\begin{equation}\label{Ineq: strong type (p,p) on amenable group}
   \norm{{\sup}_n^+ T_{m_n}(f)}_{p}\lesssim (p-1)^{-2}\norm{f}_{p}.
\end{equation}
Moreover,
\[
T_{m_n} f \to f \quad \text{b.a.u. (and a.u.\ if $p \geq 2$) as } n \to \infty.
\]
\end{corollary}\par

To prove Corollary~\ref{Cor: CPAP for amenable groups}, we observe that any unimodular amenable locally compact group $G$ is measure equivalent to $\mathbb{Z}$ in a broader sense of \cite[Definition 1.1]{MR3117525} than the discrete case discussed earlier, as shown in \cite{MR4235753}. We do not require the precise definition of this notion; it is enough to note that in this setting, there always exists a proper cocycle $\alpha : G \times X \to \mathbb{Z}$, as established in \cite[Theorem 4.1]{deprez2014permanence}. This suffices to apply Proposition~\ref{Prop: *-representation from LG to crossed product} and Theorem~\ref{Thm: multiplier diagram}.

\begin{proof}[Proof of Corollary~\ref{Cor: CPAP for amenable groups}]

Recall that on the one-dimensional torus $\mathbb{T}$, the Ces\`aro sums can be realized as a Fourier multiplier with symbol $C_n : \mathbb{Z} \to \mathbb{R}$ given by
\[
C_n(k) = \left(1 - \frac{|k|}{n+1}\right) \chi_{[-n, n]}(k), \quad \text{for all } k \in \mathbb{Z}.
\]
It is well known that $(C_n)_{n \geq 0}$ is a sequence of unital positive definite functions in $A(\mathbb{Z})=L_1(\mathbb{T})$, and that $C_n(k) \to 1$ pointwise for each $k \in \mathbb{Z}$ as $n \to \infty$. \par 

Consider the finite measure space $(X,\mu)$ and the proper cocycle $\alpha: G \times X \to \mathbb{Z}$ introduced before this proof. Without loss of generality, we assume that $\mu(X) = 1$. Take
\[
m_n(g) =\widetilde{C}_n(g)\coloneqq \int_X C_n(\alpha(g, x)) \, \dd\mu(x).
\]
Recall $\mathbb{E}$ and $\iota$ is contractive by Proposition~\ref{Prop: *-representation from LG to crossed product}. Applying Theorem~\ref{Thm: multiplier diagram}, we obtain
\[
\begin{split}
  &\norm{(T_{m_n})_{n\geq 0} : L_p(\mathcal{L}(G)) \to L_p(\mathcal{L}(G); \ell_\infty)} \\
  =\;& \norm{\mathbb{E} \circ (\id_{L_\infty(X)\rtimes_\theta G} \otimes T_{C_n}) \circ \iota : L_p(\mathcal{L}(G)) \to L_p(\mathcal{L}(G); \ell_\infty)} \\
  \leq\;& \norm{(\id_{L_\infty(X)\rtimes_\theta G} \otimes T_{C_n})_{n\geq 0} : L_p((L^\infty(X) \rtimes_\theta G) \,\overline{\otimes}\, \mathcal{L}(\mathbb{Z})) \to L_p((L^\infty(X) \rtimes_\theta G) \,\overline{\otimes}\, \mathcal{L}(\mathbb{Z}); \ell_\infty)}.
\end{split}
\]
By \cite[Theorem 4.4]{MR3079331} and the proof of \cite[Theorem 4.2]{MR3079331}, we have
\[
\norm{(\id\otimes T_{C_n})(f)}_p\lesssim (p-1)^{-2}\norm{f}_p,\quad f\in L_p(\mathbb{T},L_p(L^\infty(X) \rtimes_\theta G)).
\]
Therefore, the sequence $(T_{m_n})_{n\geq 0}$ satisfies the desired strong-type $(p,p)$ inequality.\par 

From the cocycle relation~\eqref{Eqn: cocycle for ME}, we observe that $\alpha(e_G, x) = e_{\mathbb{Z}} = 0$ for all $x \in X$. Hence,
\[
  m_n(e_G)=\int_X C_n(\alpha(e_G,x))\dd\mu(x)=C_n(0)=1.
\]
By properness and Lemmas~\ref{lem: A to A} and \ref{lem: uniformly convergence to uniformly convergence}, we know that each $m_n$ is a positive definite function in $A(G)$ and that $m_n \to 1$ uniformly on compact subsets in $G$. Consequently, for all $f \in \lambda(C_c(G))$,
\[
\| T_{m_n}f- f\|_\infty \to 0,  \quad \text{as } n\to \infty.
\]
Since the sequence $(T_{m_n})_{n \geq 0}$ is of strong type $(p,p)$, it follows from \cite[Theorem 3.1]{MR4172864} and the density of $\lambda(C_c(G))$ in $L_p(\mathcal{L}(G))$ for $1 < p < \infty$ that
\[
T_{m_n} f \to f \quad \text{b.a.u.\ (a.u.\ if $p \geq 2$) as } n \to \infty, \quad \forall f \in L_p(\mathcal{L}(G)).
\]
This completes the proof.
\end{proof}\par 

Consequently, we obtain b.a.u.\ convergence on the commonly recognized critical Orlicz space $L \log^2 L(\mathcal{L}(G))$ for noncommutative strong maximal inequalities. Before stating this result, we first present a standard auxiliary lemma related to Banach principles. The proof proceeds verbatim as in \cite[Lemma 3.6]{MR3815246}, except that the maximal inequality used therein is replaced by \eqref{Ineq: weak maximal inequality for Orlicz space}.

\begin{lemma}\label{Lem: weak Llog^2L to b.a.u}
Let $(\Phi_n)_{n \in \mathbb{N}}$ be a sequence of positive linear maps on $L\log^2L(\mathcal{M})$ which satisfies the weak-type maximal inequality: for every $f\in L\log^2L(\mathcal{M})_+$ and $\alpha > 0$, there exists a projection $e \in \mathcal{M}$ such that
\begin{equation} \label{Ineq: weak maximal inequality for Orlicz space}
\|e \Phi_n(f) e\|_\infty \leq \alpha, \quad \text{for all } n \in \mathbb{N}, \quad \text{and} \quad \tau(e^\perp) \leq C \frac{\|f\|_{L\log^2 L}}{\alpha}.
\end{equation}
If $(\Phi_n(f))_{n \in \mathbb{N}}$ converges b.a.u.\ on a dense subspace of $L\log^2 L(\mathcal{M})$, then it converges b.a.u.\ for all $f \in L\log^2 L(\mathcal{M})$.
\end{lemma}

\begin{corollary}\label{Cor: Llog^2L b.a.u. for amenable groups}
Let $G$ be a unimodular amenable locally compact second countable group. Then there exists a sequence of unital positive definite functions $(m_n)_{n \geq 0} \subset A(G)$ such that $m_n \to 1$ uniformly on compact sets. For all $f \in L \log^2 L(\mathcal{L}(G))$, we have
\[
\norm{{\sup}_n^+ T_{m_n}(f)}_{1}\leq C\norm{f}_{L\log^2L},
\]
and
\[
T_{m_n}(f) \to f \quad \text{b.a.u.\ as } n \to \infty.
\]
\end{corollary}

\begin{proof}
Let $m_n$ denote the symbols constructed in the proof of Corollary~\ref{Cor: CPAP for amenable groups}. Since the inequality \eqref{Ineq: strong type (p,p) on amenable group} has the order $(1 - p)^{-2}$, applying the extrapolation method of \cite[Theorem 2.3]{MR2594639} yields boundedness of the sequence $(T_{m_n})_{n \geq 0}$ from $L\log^2 L$ into the noncommutative maximal $L_1$ space:
 \[
    \norm{{\sup}_n^+ T_{m_n}(f)}_{1}\leq C\norm{f}_{L\log^2L},\quad \forall f \in L\log^2L(\mathcal{L}(G)).
  \]
For $f\in L\log^2L(\mathcal{L}(G))_+$, taking a projection $e$ s.t. $\norm{eae}_{\infty}\leq \alpha$ with $T_{m_n}(f)\leq a$, the above strong-type inequality implies the weak-type estimate \eqref{Ineq: weak maximal inequality for Orlicz space}. The conclusion then follows from Lemma~\ref{Lem: weak Llog^2L to b.a.u} and the density of $\lambda(C_c(G))$ in $L\log^2 L(\mathcal{L}(G))$, as in the proof of Corollary~\ref{Cor: CPAP for amenable groups}.
\end{proof}

\subsection{Noncommutative Jodeit theorem}\label{Sec: semisimple Lie group}

In this subsection, we establish a noncommutative analogue of the classical Jodeit theorem. As a consequence, we show that $L_p$-weak amenability passes from a hyperlinear lattice to the ambient locally compact group.

A lattice $\Gamma$ in a unimodular locally compact group $G$ is a closed discrete subgroup such that $G / \Gamma$ carries a finite $G$-invariant measure. Let $\Omega \subset G$ denote a fundamental domain, that is, the range of a Borel cross-section from $G/\Gamma$ to $G$. Then we may write
\[
G = \bigsqcup_{\gamma \in \Gamma} \Omega \gamma.
\]
Since $\Gamma$ is a lattice, we have $\mu(\Omega) < \infty$, where $\mu$ is a Haar measure on $G$. Let $\nu$ denote the restriction of $\mu$ to $\Omega$. As in the discussion of measure equivalence, there is a natural $\nu$-preserving action of $G$ on $X$, which we denote by $\theta$. The associated Borel cocycle $\alpha : G \times \Omega \to \Gamma$ satisfies the identity
\begin{equation}\label{Eqn: cocycle for lattice}
\alpha(g_1 g_2, x) = \alpha(g_1, \theta_{g_2}x) \alpha(g_2, x).
\end{equation}
 A detailed discussion can be found in \cite[Section 2]{MR3476201}. \par 

\begin{proof}[Proof of Corollary~\ref{Cor: transferred multiplier for non-amenable groups}]
By \cite[Theorem 4.12, Chapter X]{MR1943006}, we have the isomorphism
\[
L_\infty(\Omega) \rtimes_\theta G \cong \mathcal{L}(\Gamma) \,\overline{\otimes}\, B(L_2(\Omega, \nu)).
\]
Since $\Gamma$ is hyperlinear, its group von Neumann algebra $\mathcal{L}(\Gamma)$ is QWEP. Moreover, the tensor product of two QWEP von Neumann algebras is again QWEP by \cite{MR1218321}, so $L_\infty(\Omega) \rtimes_\theta G$ is QWEP.\par 

Assume now that $\Gamma$ is weakly $L_p$-amenable. Then there exists a sequence of completely bounded Fourier multipliers $(m_n)_{n \geq 0} \subset A(\Gamma)$ on $L_p(\mathcal{L}(\Gamma))$ such that $m_n \to 1$ pointwise. By Theorem~\ref{Thm: Transference from L(H) to L(G)}, the functions 
\[
\widetilde{m}_n(g) = \frac{1}{\nu(\Omega)} \int_\Omega m_n(\alpha(g, x)) \, \dd\nu(x),
\]
defined using the cocycle in~\eqref{Eqn: cocycle for lattice}, give rise to a sequence of completely bounded Fourier multipliers on $L_p(\mathcal{L}(G))$ with $\norm{T_{\widetilde{m}_n}}_{\mathrm{cb}(L_p)}\leq \norm{T_{m_n}}_{\mathrm{cb}(L_p)}$. It is stated in \cite[Proof of Proposition 2.5]{MR1756981} that the associated cocycle $\alpha$ is proper in the sense of \cite{MR1756981} and hence proper in the sense of Definition \ref{Def: proper cocycle}. By Lemmas~\ref{lem: A to A} and \ref{lem: uniformly convergence to uniformly convergence}, we know that each $\widetilde{m}_n$ is a positive definite function in $A(G)$ and that $\widetilde{m}_n \to 1$ uniformly on compact subsets in $G$. Thus, $G$ is also weakly $L_p$-amenable, and we have the inequality
\[
\Lambda_{p, \mathrm{cb}}(G) \leq \Lambda_{p, \mathrm{cb}}(\Gamma),
\]
which completes the proof of Corollary~\ref{Coro: transferred multiplier of two lattices in non-amenable groups}.
\end{proof}

\begin{remark}
Note that Proposition~\ref{Prop: *-representation from LG to crossed product} and Theorem~\ref{Thm: multiplier diagram} no longer hold if $ H $ is replaced by a non-discrete group; a simple counterexample arises when taking $ G = \mathbb{Z} $ and $ H = \mathbb{R} $ with the trivial cocycle on a singleton.  However, it still makes sense to ask whether analogues of Theorem~\ref{Thm: Transference from L(H) to L(G)}, Corollary~\ref{Coro: transferred multiplier of two lattices in non-amenable groups}, and Corollary~\ref{Cor: transferred multiplier for non-amenable groups} hold in the case of non-discrete groups.

Indeed, since any lattice $\Gamma$ is measure equivalent to its ambient group $G$ in the sense of \cite{MR3117525} (see \cite[Example 1.2]{MR3117525}), a positive answer to this question would imply the inequality  
\[
\Lambda_{p, \mathrm{cb}}(G) \geq \Lambda_{p, \mathrm{cb}}(\Gamma),
\]
which would yield a variant of the noncommutative de Leeuw restriction theorem for lattices. Such a result remains open in general; at present, only partial results are known under additional assumptions of asymptotic invariance, or within certain local settings (see \cite{MR3482270,MR4721791}). 
\end{remark}

\section{A continuous analogue of Hilbert transforms on $SL_2(\mathbb{R})$} \label{Sec: Transference from SL2(Z) to SL2(R)}

In this section, we propose an analogue of Hilbert transforms on $ SL_2(\mathbb{R}) $. As mentioned in the Introduction, the problem of constructing a continuous analogue of Hilbert transforms on Lie groups has recently emerged from works on Hilbert transforms defined on $ SL_2(\mathbb{Z}) $ and on idempotent Fourier multipliers on Lie groups (see \cite{gonzalezperez2022noncommutativecotlaridentitiesgroups,MR4882285}). Our previous discussion of the noncommutative Jodiet theorems suggests a natural candidate for a Hilbert transform on $ SL_2(\mathbb{R}) $. In the following, we study this Fourier multiplier in detail, focusing in particular on its explicit computation and decay properties. As a key ingredient, we provide in Proposition~\ref{Prop: widetilde m is bi-K-invariant} an explicit formula for the transferred multiplier $ \widetilde{m} $, corresponding to any function $ m $ on $ SL_2(\mathbb{Z}) $ satisfying $ m(\gamma) = m(-\gamma) $.

Let us first compute the cocycle associated with the lattice $ SL_2(\mathbb{Z}) \subset SL_2(\mathbb{R}) $, using the standard action of $ SL_2(\mathbb{Z}) $ on the upper half-plane. In contrast to the left cocycles used in previous discussions, we will now use right cocycles in order to be consistent with the usual notation for such actions in the literature. As noted in Remark~\ref{Rmk: main Theorem for left cocycle}, this change does not lead to any essential differences in the previous results, and Theorem~\ref{Cor: transferred multiplier for non-amenable groups} remains valid in the setting of right cocycles.

We begin by recalling the classical action of $SL_2(\mathbb{R})$ on the upper half-plane $\mathcal{H} = \{x + i y : x \in \mathbb{R},\, y > 0\}$ via M\"obius transformations:
\[
\rho_g : z \mapsto \frac{a z + b}{c z + d}, \quad \text{for } g = \begin{pmatrix} a & b \\ c & d \end{pmatrix} \in SL_2(\mathbb{R}), \ z \in \mathcal{H}.
\]
It is well known that $\mathcal{H}$ admits a fundamental domain for the action of $PSL_2(\mathbb{Z})$ given by
\begin{equation}\label{Eqn: fundamental domain of PSL2(Z)}
\mathcal{H} = \bigsqcup_{\gamma \in PSL_2(\mathbb{Z})} \rho_\gamma \mathcal{F},
\end{equation}
where
\[
\mathcal{F} = \left\{ x + i y \in \mathcal{H} : \abs{x} \leq \frac{1}{2},\ x^2 + y^2 \geq 1 \right\}.
\]
See \cite[p.~223]{MR803508} for further discussion. Let $SL_2(\mathbb{R}) = ANK$ denote the Iwasawa decomposition, where
\[
A = \left\{ \begin{pmatrix} r & 0 \\ 0 & \frac{1}{r} \end{pmatrix} : r > 0 \right\}, \quad
N = \left\{ \begin{pmatrix} 1 & x \\ 0 & 1 \end{pmatrix} : x \in \mathbb{R} \right\}, \quad
K = SO(2) = \left\{ \begin{pmatrix} \cos \theta & -\sin \theta \\ \sin \theta & \cos \theta \end{pmatrix} : \theta \in [0, 2\pi) \right\}.
\]
The subgroup $AN \subset SL_2(\mathbb{R})$ can be identified with the set
\[
AN = \left\{ \begin{pmatrix} \sqrt{y} & \dfrac{x}{\sqrt{y}} \\ 0 & \dfrac{1}{\sqrt{y}} \end{pmatrix} : y > 0,\ x \in \mathbb{R} \right\}.
\]
The left Haar measure $\mu_L$ on $AN$ is given by
\[
\dd\mu_L(s) = \frac{\dd x\,\dd y}{y^2}, \quad \text{where } s = \begin{pmatrix} \sqrt{y} & \dfrac{x}{\sqrt{y}} \\ 0 & \dfrac{1}{\sqrt{y}} \end{pmatrix}.
\]
The Haar measure $\nu$ on $K = SO(2)$ is defined by
\[
\dd\nu(k) = \dd\theta, \quad \text{where } k = \begin{pmatrix} \cos \theta & -\sin \theta \\ \sin \theta & \cos \theta \end{pmatrix}.
\]
Any element $g \in SL_2(\mathbb{R})$ can be uniquely decomposed as $g = s k$ with $s \in AN$ and $k \in K$. The Haar measure $d\mu$ on $SL_2(\mathbb{R})$ then decomposes as
\[
\dd\mu(g) = \dd\mu_L(s)\, \dd\nu(k).
\]
Let $P_{AN}(g)$ and $P_K(g)$ denote the $AN$ and $K$ components of $g$ under this decomposition.

We define the map $\pi: SL_2(\mathbb{R}) \to \mathcal{H}$ by
\[
\pi(g) \coloneqq \rho_g(i), \quad g \in SL_2(\mathbb{R}).
\]
The restriction $\pi|_{AN}$ is a bijection from $AN$ onto the upper half-plane $\mathcal{H}$.
Let us now fix the subset 
\[
K_+ = \left\{ \begin{pmatrix} \cos\theta & -\sin\theta \\ \sin\theta & \cos\theta \end{pmatrix} : \theta \in [0, \pi) \right\}.
\]
This satisfies $K = K_+ \sqcup (-K_+)$. The fundamental domain for the lattice $SL_2(\mathbb{Z})\subset SL_2(\mathbb{R})$ can be realized as 
\begin{equation}\label{Eqn: decomposition of AN times K}
SL_2(\mathbb{R}) = AN K = \bigsqcup_{\gamma \in SL_2(\mathbb{Z})} \gamma \left( (\pi|_{AN})^{-1}(\mathcal{F}) \cdot K_+ \right).
\end{equation}
Indeed, for any $sk \in ANK$, the decomposition~\eqref{Eqn: fundamental domain of PSL2(Z)} implies that there exist exactly two elements $\gamma, -\gamma \in SL_2(\mathbb{Z})$ and a unique $s' \in (\pi|_{AN})^{-1}(\mathcal{F})$ such that
\[
\pi(s) = \rho_{\gamma} \pi(s') = \rho_{-\gamma} \pi(s').
\]
But since
\[
\pi(s) = \rho_{\pm \gamma} \pi(s') = \rho_{\pm \gamma} \rho_{s'} i = \rho_{\pm \gamma s'} i = \pi(P_{AN}(\pm \gamma s')),
\]
and $\pi$ is a bijection on $AN$, it follows that $s = P_{AN}(\pm \gamma s')$. Therefore,
\[
s k = P_{AN}(\pm \gamma s') k = \pm \gamma s' P_K(\pm \gamma s')^{-1} k.
\]
Since exactly one of the elements $P_K(\pm \gamma s')^{-1} k$ lies in $K_+$, the uniqueness of the decomposition~\eqref{Eqn: decomposition of AN times K} is established.\par 

The measure of the set $(\pi|_{AN})^{-1}(\mathcal{F})\subset AN$ is
\[
\mu_L\left( (\pi|_{AN})^{-1}(\mathcal{F}) \right) = \int_{-\frac{1}{2}}^{\frac{1}{2}} \int_{\sqrt{1 - x^2}}^\infty \frac{1}{y^2} \, \dd y \, \dd x = \frac{\pi}{3},
\]
and the measure of $K_+ \subset SO(2)$ is
\[
\nu(K_+) = \int_0^\pi \dd\theta = \pi.
\]
Therefore, the set $(\pi|_{AN})^{-1}(\mathcal{F}) \cdot K_+ $ has finite measure $\frac{\pi^2}{3}$.\par 

We now consider the right cocycle $\beta$ associated with the lattice structure, defined via the standard construction introduced at the beginning of Subsection~\ref{Sec: semisimple Lie group}. By~\eqref{Eqn: decomposition of AN times K}, for each $s_0 k_0\in (\pi|_{AN})^{-1}(\mathcal{F}) \cdot K_+$ and $g \in SL_2(\mathbb{R})$, there exists a unique element $\beta(s_0 k_0, g) \in SL_2(\mathbb{Z})$ such that
\[
s_0k_0 g \in \beta(s_0 k_0, g) \left( (\pi|_{AN})^{-1}(\mathcal{F}) \cdot K_+ \right),
\]
and we consider the natural right action of $SL_2(\mathbb{R})$ on the fundamental domain by
\[
s_0k_0 \cdot g \coloneqq \beta(s_0 k_0, g)^{-1} s_0 k_0 g.
\]
The map $\beta : \left( (\pi|_{AN})^{-1}(\mathcal{F}) \cdot K_+ \right) \times SL_2(\mathbb{R}) \to SL_2(\mathbb{Z})$, associated with the natural right action, satisfies the right cocycle identity.

\begin{remark}\label{Rmk: properties of beta}
It is easy to see that $\beta(s_0 k_0,g) = \gamma$ only if $\rho_{s_0k_0g}i \in \rho_\gamma \mathcal{F}$. In particular, the cocycle satisfies
\[
\beta(s_0 k_0, gk) \in \{ \pm \beta(s_0 k_0, g) \}, \quad \forall k \in K,
\]
and we may shift the $K$-component:
\[
\beta(s_0 k_0, g) = \beta(s_0, k_0 g).
\]
\end{remark}
The transferred symbol $\widetilde{m}$ is given by
\[
\begin{split}
\widetilde{m}(g)
&= \frac{1}{\mu\left( (\pi|_{AN})^{-1}(\mathcal{F}) \times K_+ \right)}
\int_{(\pi|_{AN})^{-1}(\mathcal{F}) \times K_+} m\left( \beta((s_0, k_0), g) \right) \, \dd\mu(s_0, k_0) \\
&= \frac{1}{\mu_L\left( (\pi|_{AN})^{-1}(\mathcal{F}) \right)\nu(K_+)} \int_{K_+} \int_{(\pi|_{AN})^{-1}(\mathcal{F})} m\left( \beta(s_0, k_0 g) \right) \, \dd\mu_L(s_0) \, \dd\nu(k_0).
\end{split}
\]
Define the auxiliary function $\widetilde{\widetilde{m}} \in L_\infty(SL_2(\mathbb{R}))$ by
\begin{equation}\label{eqn: formula of tilde tilde m for SL2(R)}
\widetilde{\widetilde{m}}(g) = \frac{3}{\pi} \int_{(\pi|_{AN})^{-1}(\mathcal{F})} m\left( \beta(s_0, g) \right) \, \dd\mu_L(s_0).
\end{equation}
Then $\widetilde{m}$ can then be expressed as an average over $K_+$:
\begin{equation}\label{eqn: formula of tilde m for SL2(R)}
\widetilde{m}(g) = \frac{1}{\pi} \int_{K_+} \widetilde{\widetilde{m}}(k_0 g) \, \dd\nu(k_0).
\end{equation}
The structure of the cocycle $\beta$ and the averaging formula for $\widetilde{m}$ yield symmetry properties inherited from the original symbol $m$. In particular, if $m$ is invariant under negation, then the transferred symbol $\widetilde{m}$ becomes bi-$K$-invariant, as shown in the following proposition.

\begin{proposition}\label{Prop: widetilde m is bi-K-invariant}
Let $m \in L_\infty(SL_2(\mathbb{Z}))$ satisfy $m(\gamma) = m(-\gamma)$. Then the transferred function $\widetilde{m}$ is bi-$K$-invariant.
\end{proposition}

\begin{proof}
If $m$ satisfies $m(\gamma) = m(-\gamma)$, then by Remark~\ref{Rmk: properties of beta}, both $\widetilde{\widetilde{m}}$ and $\widetilde{m}$ are right $K$-invariant and
\[
\widetilde{\widetilde{m}}(-g) = \widetilde{\widetilde{m}}(g), \quad \widetilde{m}(-g) = \widetilde{m}(g).
\]
So we can extend the integration from $K_+$ to all of $K$:
\[
\begin{split}
\int_{-K_+} \widetilde{\widetilde{m}}(k_0 g) \, \dd\nu(k_0)
&= \int_{K_+} \widetilde{\widetilde{m}}(-k_0 g) \, \dd\nu(-k_0)
= \int_{K_+} \widetilde{\widetilde{m}}(k_0 g) \, \dd\nu(k_0).
\end{split}
\]
Therefore, for any $k \in K$, we compute
\[
\begin{split}
\widetilde{m}(k g)
&= \frac{1}{\pi} \int_{K_+} \widetilde{\widetilde{m}}(k_0 k g) \, \dd\nu(k_0)
= \frac{1}{2\pi} \int_K \widetilde{\widetilde{m}}(k_0 k g) \, \dd\nu(k_0) \\
&= \frac{1}{2\pi} \int_K \widetilde{\widetilde{m}}(k_0 g) \, \dd\nu(k_0)
= \frac{1}{\pi} \int_{K_+} \widetilde{\widetilde{m}}(k_0 g) \, \dd\nu(k_0)
= \widetilde{m}(g).
\end{split}
\]
This shows that $\widetilde{m}$ is left $K$-invariant as well. Hence, $\widetilde{m}$ is bi-$K$-invariant.
\end{proof}

We now introduce a concrete example of a Fourier multiplier on $SL_2(\mathbb{Z})$ given in \cite[Proposition 5.1]{gonzalezperez2022noncommutativecotlaridentitiesgroups}. Let
\[
S = \begin{pmatrix} 0 & -1 \\ 1 & 0 \end{pmatrix}, \quad
T = \begin{pmatrix} 1 & 1 \\ 0 & 1 \end{pmatrix}, \quad R=ST = \begin{pmatrix} 0 & -1 \\ 1 & 1 \end{pmatrix}.
\]
The matrices $S$ and $R$ generate $SL_2(\mathbb{Z})$. More precisely, any element $\gamma \in SL_2(\mathbb{Z})$ can be written as a free product of $S$ and $R$ (up to sign) and
\[
SL_2(\mathbb{Z}) = \{ \pm I, \pm S, \pm R, \pm R^2, \pm SR, \pm RS, \pm SR^2, \pm SRS, \dots \}.
\]
In the sequel we focus on the function $m$ given by \eqref{eqn: Hilbert transform on PSL_2(Z)}, which as explained in \cite{gonzalezperez2022noncommutativecotlaridentitiesgroups}, can also be equivalently defined as
\[
m(\gamma) =
\begin{cases}
1, & \text{if } \gamma \text{ begins with } \pm S \text{ or } \pm I, \\
0, & \text{if } \gamma \text{ begins with } \pm R.
\end{cases}
\]
The following lemma characterizes the image of $\mathcal{F}$ under the action by the support of the function $m$ above.

\begin{lemma}\label{Lem: Region of rho_Gamma F}
Let
\[
\Gamma_S = \{ \gamma \in SL_2(\mathbb{Z}) : \gamma \text{ begins with } \pm S \text{ or } \gamma = \pm I \}.
\]
For each $n \in \mathbb{Z}$, define the vertical strip
\[
S_n = \left\{ z \in \mathcal{H} : \abs{\Re(z - n)} \leq \frac{1}{2} \right\}.
\]
Let $\rho_{\Gamma_S} \mathcal{F} = \bigcup_{\gamma \in \Gamma_S} \rho_\gamma \mathcal{F}$ and set
\[
\mathcal{A} = \left( \bigcup_{n \geq 0} S_n \right) \cap \left\{ z \in \mathcal{H} : \abs{z + 1} \geq 1 \right\}.
\]
Then
\[
\rho_{\Gamma_S} \mathcal{F} = \mathcal{A}.
\]
\end{lemma}

\begin{proof}
 We first show that $\mathcal{A} \subset \rho_{\Gamma_S} \mathcal{F}$. Recall that the generators of $SL_2(\mathbb{Z})$ act on $\mathcal{H}$ via
  \[
    \rho_T z=z+1,\quad \rho_S z=-\frac{1}{z}
  \]
  and that the image of the standard fundamental domain $\mathcal{F}$ under $\rho_S$ is given by
  \[
  \rho_S\mathcal{F}=\{x+iy:(x+1)^2+y^2\geq 1\}\cap \{x+iy:(x-1)^2+y^2\geq 1\}\cap \{x+iy:x^2+y^2\leq  1\}.
  \]
 Let $\mathcal{A}_{3}=\{x+iy:(x-\frac{1}{3})^2+y^2\leq \frac{1}{9}\}\cup \{x+iy:(x-\frac{2}{3})^2+y^2\leq \frac{1}{9}\}$. We have 
\[
  \begin{split}
   \mathcal{A}=&\left( \bigcup_{l_0\in \{0,1\}}\bigcup_{n\geq 0}\rho_{T^n S^{l_0}}\mathcal{F} \right)\bigcup\left( \bigcup_{l_0\in \{0,1\}}\bigcup_{n\geq 0}\rho_{T^n(TST)S^{l_0}}\mathcal{F} \right)\bigcup\left(\bigcup_{n\geq 0} \rho_{T^n}\mathcal{A}_3 \right)\\
    =&\mathcal{A}_1\bigcup \mathcal{A}_2\bigcup\left(\bigcup_{n\geq 0} \rho_{T^n}\mathcal{A}_3 \right).
  \end{split}
\]  
\begin{figure}[h]
  \centering
  \includegraphics[scale=0.8]{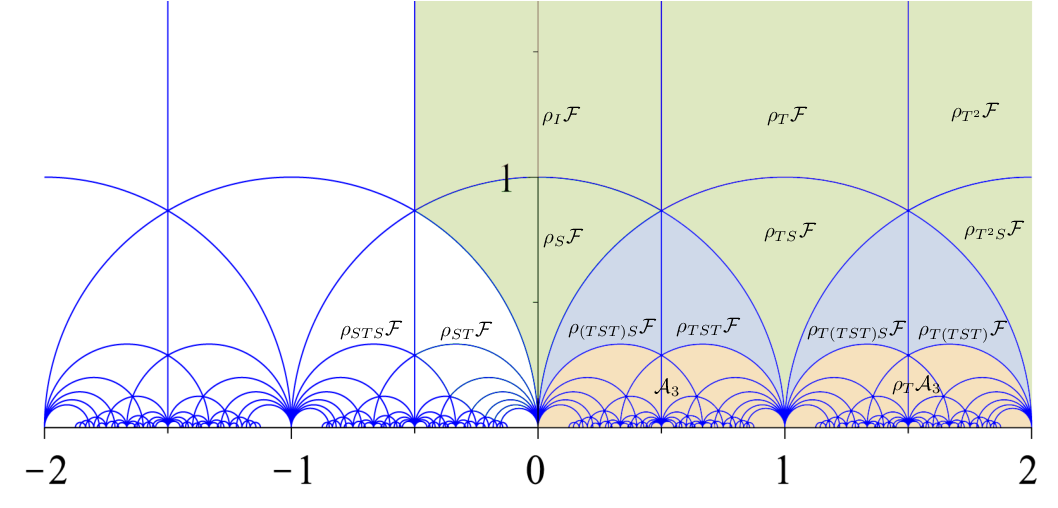}
  \caption{The decomposition of the region $\mathcal{A}$. The regions $\mathcal{A}_1$ and $\mathcal{A}_2$ are colored green and blue, respectively.}
  \label{fig: the region of A}
\end{figure}
Note that $\mathcal{A}_1 \subset \rho_{\Gamma_S} \mathcal{F}$ and $\mathcal{A}_2 \subset \rho_{\Gamma_S} \mathcal{F}$. For $z = x + i y \in \mathcal{A}_3$, a computation shows that
\[
  \begin{split}
    \Re((TST)^{-1}(x+iy))=&\frac{x-x^2-y^2}{(1-x)^2+y^2}\\
    \geq &\min\left( \frac{x-x^2+(x-\frac{1}{3})^2-\frac{1}{9}}{(1-x)^2+y^2},\frac{x-x^2+(x-\frac{2}{3})^2-\frac{1}{9}}{(1-x)^2+y^2} \right)\\
    =&\min\left( \frac{\frac{x}{3}}{(1-x)^2+y^2},\frac{\frac{1-x}{3}}{(1-x)^2+y^2} \right)>0,
  \end{split}
\]
and
\[
\Im((TST)^{-1}(z)) = \frac{y}{(1 - x)^2 + y^2} > y,
\]
since $\mathcal{A}_3 \subset \{ x + i y : (1 - x)^2 + y^2 < 1 \}$. Therefore, repeated applications of $\rho_{T^{-1}}$ and $\rho_{(TST)^{-1}}$ move any point in $\mathcal{A}_3$ upward within $\mathcal{A}$, eventually into $\mathcal{A}_1$ or $\mathcal{A}_2$. Since both are contained in $\rho_{\Gamma_S} \mathcal{F}$, we conclude that $\mathcal{A} \subset \rho_{\Gamma_S} \mathcal{F}$.\par 

To prove the reverse inclusion, suppose for contradiction that there exists $\gamma \in \Gamma_S$ such that $\rho_\gamma \mathcal{F} \not\subset \mathcal{A}$. Since $\rho_S \mathcal{A} \cup \mathcal{A} = \mathcal{H}$, it follows that
\[
\rho_S \rho_\gamma \mathcal{F} \subset \mathcal{A} \subset \rho_{\Gamma_S} \mathcal{F},
\]
which implies $S\gamma \in \Gamma_S$. Therefore, $\gamma$ would either be in $\{ \pm S, \pm I \}$ or begin with $\pm R$, contradicting the assumption that $\gamma \in \Gamma_S$ and $\rho_\gamma \mathcal{F} \not\subset \mathcal{A}$. Thus, we must have $\rho_{\Gamma_S} \mathcal{F} \subset \mathcal{A}$, and the proof is complete.
Thus, $S\gamma\in \Gamma_S$. Hence $\gamma\in \{\pm S,\pm I\}$ or $\gamma$ starts with $\pm R$, which contradicts $\gamma\in \Gamma_S$ and $\rho_\gamma\mathcal{F}\not\subset \mathcal{A}$. Therefore, we have $\rho_{\Gamma_S}\mathcal{F}\subset \mathcal{A}$ and hence $\rho_{\Gamma_S} \mathcal{F}=\mathcal{A}$. Thus the lemma was proved.
\end{proof}\par 

We now use the formula \eqref{eqn: formula of tilde m for SL2(R)} and \eqref{eqn: formula of tilde tilde m for SL2(R)} to compute $\widetilde{m}$ on the subgroup $A$. Let
\[
s_0 = \begin{pmatrix}
\sqrt{y} & \dfrac{x}{\sqrt{y}} \\
0 & \dfrac{1}{\sqrt{y}}
\end{pmatrix}, \quad
g = \begin{pmatrix}
\sqrt{g_y} & \dfrac{g_x}{\sqrt{g_y}} \\
0 & \dfrac{1}{\sqrt{g_y}}
\end{pmatrix},
\]
where $x + i y \in \mathcal{F}$ and $g_x + i g_y \in \mathcal{H}$. By Lemma~\ref{Lem: Region of rho_Gamma F}, the condition $\beta((s_0, I), g) \in \Gamma_S$ is equivalent to
\[
\rho_{s_0 g} i \in \rho_{\Gamma_S} \mathcal{F}=\mathcal{A}.
\]
A direct computation shows that
\[
\rho_{s_0 g} i = x + g_x y + i g_y y.
\]
Thus, the condition $\rho_{s_0 g} i \in \mathcal{A}$ is equivalent to
\[
x + g_x y > -\frac{1}{2}, \quad \text{and} \quad \left( x + g_x y + 1 \right)^2 + \left( g_y y \right)^2 > 1.
\]
Define the region
\[
A(\omega_1, \omega_2) = \left\{ x + i y \in \mathcal{H} : x + \omega_1 y > -\frac{1}{2} \text{ and } \left( x + \omega_1 y + 1 \right)^2 + \left( \omega_2 y \right)^2 > 1 \right\}.
\]
Then we obtain the explicit formula
\begin{equation}\label{Eqn: explicit formula of tildetilde m for Hilbert transform}
  \widetilde{\widetilde{m}}(g) = \frac{3}{\pi} \int_{A(g_x, g_y) \cap \mathcal{F}} \frac{1}{y^2} \, \dd x \, \dd y.
\end{equation}\par 

By the bi-$K$-invariance of $\widetilde{m}$ and right $K$-invariance of $\widetilde{\widetilde{m}}$, we may restrict our attention to elements $g \in SL_2(\mathbb{R})$ of the form $g = P_{AN}(ka)$, where 
\[
k = \begin{pmatrix}
\cos\theta & -\sin\theta \\
\sin\theta & \cos\theta
\end{pmatrix}
\quad \text{and} \quad
a = \begin{pmatrix}
r & 0 \\
0 & \dfrac{1}{r}
\end{pmatrix}.
\]
Then
\[
g = P_{AN}(ka) =
\begin{pmatrix}
\dfrac{r}{\sqrt{\cos^2\theta + r^4 \sin^2\theta}} &
\dfrac{(r^4 - 1) \sin\theta \cos\theta}{r \sqrt{\cos^2\theta + r^4 \sin^2\theta}} \\
0 & \dfrac{\sqrt{\cos^2\theta + r^4 \sin^2\theta}}{r}
\end{pmatrix}.
\]
Define
\[
g_x(r, \theta) = \frac{(r^4 - 1) \sin\theta \cos\theta}{\cos^2\theta + r^4 \sin^2\theta}, \quad
g_y(r, \theta) = \frac{r^2}{\cos^2\theta + r^4 \sin^2\theta}.
\]
Then $g = P_{AN}(ka)$ can be written as
\[
g = \begin{pmatrix}
\sqrt{g_y(r, \theta)} & \dfrac{g_x(r, \theta)}{\sqrt{g_y(r, \theta)}} \\
0 & \dfrac{1}{\sqrt{g_y(r, \theta)}}
\end{pmatrix}.
\]
Denote $\widetilde{\widetilde{m}}(g)$ in \eqref{Eqn: explicit formula of tildetilde m for Hilbert transform} by $\widehat{m}(g_x, g_y)$. Since $g_x(r, \theta)$ and $g_y(r, \theta)$ are periodic in $\theta$, we may write
\begin{equation}\label{Eqn: explicit formula of tilde m in a}
\widetilde{m}(a) = \frac{1}{\pi} \int_0^\pi \widehat{m}(g_x(r, \theta), g_y(r, \theta)) \, \dd\theta = \frac{1}{\pi} \int_{-\frac{\pi}{2}}^{\frac{\pi}{2}} \widehat{m}(g_x(r, \theta), g_y(r, \theta)) \, \dd\theta=\frac{3}{\pi^2} \int_{-\frac{\pi}{2}}^{\frac{\pi}{2}}\int_{A(g_x(r,\theta), g_y(r,\theta)) \cap \mathcal{F}} \frac{1}{y^2} \, \dd x \, \dd y \, \dd\theta.
\end{equation}

  \begin{remark}\label{Rmk: decay of symbol tilde m}
  Let $\|g\| = \sup_{v \in \mathbb{R}^2} \frac{\|gv\|_2}{\|v\|_2}$ denote the operator norm of $g \in SL_2(\mathbb{R})$. The function $\widetilde{m}$ defined in~\eqref{Eqn: explicit formula of tilde m in a} is bi-$K$-invariant and satisfies the following pointwise H\"ormander--Mikhlin-type estimate:
\begin{equation} \label{Ineq: HM condition for tilde m}
\|g\|^{|\gamma|} \left| \dd_g^\gamma \widetilde{m}(g) \right| \leq C, \quad \text{for all multi-indices } \gamma \text{ with } |\gamma| \leq 1,
\end{equation}
where $\dd_g^\gamma$ denotes the Lie differential operator corresponding to the multi-index $\gamma = (j_1, \dots, j_k)$ with $|\gamma| = k$, defined by
\[
\dd_g^\gamma \widetilde{m}(g) = \partial_{X_{j_1}} \partial_{X_{j_2}} \cdots \partial_{X_{j_k}} \widetilde{m}(g) = \left( \prod_{1 \leq \ell \leq k}^{\rightarrow} \partial_{X_{j_\ell}} \right) \widetilde{m}(g).
\]
The verification of this estimate is elementary but quite lengthy. We refer the reader to the Appendix for the computation details. It is also shown there that \eqref{Ineq: HM condition for tilde m} does not extend to second-order derivatives.

It is worth noting that this phenomenon reflects the highly noncommutative nature of the lattice. In the Euclidean setting, for the directional Hilbert transform $m(x) = \operatorname{sgn}(\langle x, y \rangle)$ along a direction $y \in \mathbb{R}^d$, its restriction to the lattice $\mathbb{Z}^d$ gives rise to a transferred multiplier $\widetilde{m}$, which is a continualization of $m$ on $\mathbb{R}^d$. However, it is easy to see that for $d \geq 2$, $\widetilde{m}$ fails to satisfy even first-order H\"ormander--Mikhlin estimates. In contrast, the lattice $SL_2(\mathbb{Z}) \subset SL_2(\mathbb{R})$ is nonuniform, which significantly contributes to the improved regularity of the transferred function.

\end{remark}

\printbibliography
\addresseshere

\newpage
\appendix 
\section{Regularity of Lie derivatives of Hilbert transform analogues}
The main goal of this appendix is to establish the following result.

\begin{theorem}\label{thm: decay of Hilbert transform on SL2(R)}
  Let $1 < p < \infty$ and let $m \in L_\infty(SL_2(\mathbb{Z}))$ be the symbol defined in \eqref{eqn: Hilbert transform on PSL_2(Z)}. Let $\widetilde{m} \in L_\infty(SL_2(\mathbb{R}))$ denote the transferred multiplier associated to $m$ via \eqref{eqn: formula of tilde m}. Then $\widetilde{m}$ is bi-$K$-invariant and satisfies the pointwise estimate
  \[
     \norm{g}^{\abs{\gamma}}\abs{\dd_g^\gamma \widetilde{m}(g)}\leq C,\quad \text{for all multi-indices }  \abs{\gamma}\leq 1,
  \]
  where $\norm{g} =\sup_{v\in \mathbb{R}^2}\frac{\norm{gv}_2}{\norm{v}_2}$ is the operator norm of $g \in SL_2(\mathbb{R})$.

  Moreover, $T_{\widetilde{m}}$ is a completely bounded Fourier multiplier on $L_p(\mathcal{L}(SL_2(\mathbb{R})))$.
\end{theorem}

We begin by establishing some notational conventions and preliminary computations. The proof of Theorem~\ref{thm: decay of Hilbert transform on SL2(R)} is presented in Section~\ref{Sec: Proof of the main theorem}, where it relies on Lemma~\ref{Lem: dereasing ratio of derivative of m hat}. In Remark~\ref{Rmk: non-existence of second order}, we show that the H\"ormander--Mikhlin condition fails to hold for $\widetilde{m}$ when $\abs{\gamma} = 2$.

Let
\[
X_1 = \begin{pmatrix}
1 & 0 \\
0 & -1
\end{pmatrix}, \quad 
X_2 = \begin{pmatrix}
0 & 1 \\
0 & 0
\end{pmatrix}, \quad 
X_3 = \begin{pmatrix}
0 & 1 \\
-1 & 0
\end{pmatrix}
\]
be a basis for the Lie algebra $\mathfrak{sl}_2(\mathbb{R})$. The corresponding Lie derivatives of functions on $SL_2(\mathbb{R})$ are defined by
\[
\partial_{X_j} m(g) = \eval{\dv{t}}_{t=0} m\big(g \exp(tX_j)\big),
\]
and, in general, these derivatives do not commute when $j \neq k$. Given an ordered multi-index $\gamma = (j_1, j_2, \dots, j_k)$ with $1 \leq j_i \leq 3$ and $\abs{\gamma}= k \geq 0$, the associated Lie differential operator is defined by
\[
\dd_g^\gamma m(g) = \partial_{X_{j_1}} \partial_{X_{j_2}} \cdots \partial_{X_{j_k}} m(g) = \left( \prod_{1 \leq k \leq |\gamma|}^{\rightarrow} \partial_{X_{j_k}} \right) m(g).
\]

Assume that $m \in L_\infty(SL_2(\mathbb{Z}))$ satisfies $m(\gamma) = m(-\gamma)$ for all $\gamma \in SL_2(\mathbb{Z})$. Let $SL_2(\mathbb{R}) = KAK$ denote the Cartan decomposition of the group. By Proposition~\ref{Prop: widetilde m is bi-K-invariant}, the first-order Lie derivatives of the function $\widetilde{m}$ satisfy
\[
\begin{split}
\partial_{X_j} \widetilde{m}(k'a k) &= \eval{\dv{t}}_{t=0}\widetilde{m}\big(k' a k \exp(t X_j)\big) =\eval{\dv{t}}_{t=0} \widetilde{m}\big(a k \exp(t X_j) k^{-1}\big) \\
&= \eval{\dv{t}}_{t=0} \widetilde{m}\big(a \exp(t (\ad_k X_j))\big) = \partial_{\ad_k X_j} \widetilde{m}(a),
\end{split}
\]
for all $k, k' \in K$ and $a \in A$ with $1 \leq j \leq 3$.
Let
\[
k = \begin{pmatrix}
\cos\theta & -\sin\theta \\
\sin\theta & \cos\theta
\end{pmatrix}.
\]
We compute the adjoint actions:
\begin{align*}
\ad_k X_1 &= kX_1k^{-1} = 
\begin{pmatrix}
\cos(2\theta) & \sin(2\theta) \\
\sin(2\theta) & -\cos(2\theta)
\end{pmatrix}
= \cos(2\theta) X_1 + 2\sin(2\theta) X_2 - \sin(2\theta) X_3, \\
\ad_k X_2 &= kX_2k^{-1} =
\begin{pmatrix}
-\sin\theta \cos\theta & \cos^2\theta \\
-\sin^2\theta & \sin\theta \cos\theta
\end{pmatrix}
= -\sin\theta \cos\theta X_1 + \cos(2\theta) X_2 + \sin^2\theta X_3, \\
\ad_k X_3 &= kX_3k^{-1} = 
\begin{pmatrix}
0 & 1 \\
-1 & 0
\end{pmatrix}
= X_3.
\end{align*}
Therefore, for any $a \in A$ and $k,k' \in K$, we have the following estimate:
\begin{equation}\label{Eqn: estimation of Lie derivation of m tilde}
\left| \partial_{X_j} \widetilde{m}(k' a k) \right| = \left| \partial_{\ad_k X_j} \widetilde{m}(a) \right| \leq \sum_{j=1}^3 \left| \partial_{X_j} \widetilde{m}(a) \right|.
\end{equation}

Now, let $g \in AN$ be written as
\[
g = \begin{pmatrix}
\sqrt{g_y} & \dfrac{g_x}{\sqrt{g_y}} \\
0 & \dfrac{1}{\sqrt{g_y}}
\end{pmatrix},
\]
and denote by $\widehat{m}(g_x, g_y)$ the function $\widetilde{\widetilde{m}}(g)$ defined in \eqref{eqn: formula of tilde tilde m for SL2(R)} when restricted to $AN$.
The Lie derivatives of $\widetilde{\widetilde{m}}$ are given by
\begin{equation}\label{Eqns: Lie derivations of tildetilde m}
\left\{
\begin{aligned}
   \partial_{X_1} \widetilde{\widetilde{m}}(g) &=\eval{\dv{t}}_{t=0} \widehat{m}(g_x, e^{2t}g_y) = \eval{2e^{2t}g_y \pdv{g_y} \widehat{m}(g_x, e^{2t}g_y)}_{t=0} = 2g_y \pdv{g_y} \widehat{m}(g_x, g_y), \\
  \partial_{X_2} \widetilde{\widetilde{m}}(g) &=\eval{\dv{t}}_{t=0} \widehat{m}(g_x + t g_y, g_y) = \eval{ g_y \pdv{g_x} \widehat{m}(g_x + t g_y, g_y)}_{t=0} = g_y \pdv{g_x} \widehat{m}(g_x, g_y), \\
  \partial_{X_3} \widetilde{\widetilde{m}}(g) &=\eval{\dv{t}}_{t=0} \widehat{m}(g_x, g_y) = 0.
\end{aligned}
\right.
\end{equation}
By the mean value theorem, we have
\begin{equation}\label{Eqn: formula of Lie derivation of m tilde}
\begin{split}
  \partial_{X_j} \widetilde{m}(a) &=\eval{\dv{t}}_{t=0} \frac{1}{\pi} \int_{K_+} \widetilde{\widetilde{m}}(k_0 a \exp(t X_j)) \, \dd\nu(k_0) \\
  &= \frac{1}{\pi} \lim_{\delta \to 0} \int_{K_+} \eval{\dv{t}}_{t=\delta} \widetilde{\widetilde{m}}(k_0 a \exp(t X_j)) \, \dd\nu(k_0).
\end{split}
\end{equation}
Assuming that $\left| \dv{t} \widetilde{\widetilde{m}}(k_0 a \exp(t X_j)) \right|$ is integrable over $K_+$ for $t$ in a neighborhood of $0$, the dominated convergence theorem yields
\begin{equation}\label{Eqn: final formula of Lie derivation of m tilde}
  \partial_{X_j} \widetilde{m}(a) = \frac{1}{\pi} \int_{K_+} \partial_{X_j} \widetilde{\widetilde{m}}(k_0 a) \, \dd\nu(k_0).
\end{equation}

To analyze the decay of $\partial_{X_j} \widetilde{m}$, by \eqref{Eqn: estimation of Lie derivation of m tilde}, it suffices to evaluate the integrals in \eqref{Eqn: final formula of Lie derivation of m tilde}.
The decay of $\partial_{X_j} \widetilde{m}$ on the subgroup $A$ is governed by the expression $\widehat{m}(g_x, g_y)$, derived using Equations~\eqref{Eqns: Lie derivations of tildetilde m} and~\eqref{Eqn: final formula of Lie derivation of m tilde}, where $g = P_{AN}(k a)$ for some $k \in K$ and $a \in A$. The following lemma establishes the necessary decay properties for the function $\widetilde{\widetilde{m}}(g)$ defined in~\eqref{eqn: formula of tilde tilde m for SL2(R)}. The proof of this lemma constitutes the main technical work of the next subsection.

\begin{lemma}\label{Lem: dereasing ratio of derivative of m hat}
Let
\[
g = \begin{pmatrix}
  \sqrt{g_y} & \dfrac{g_x}{\sqrt{g_y}} \\
  0 & \dfrac{1}{\sqrt{g_y}}
\end{pmatrix} \in AN,
\]
and denote $\widetilde{\widetilde{m}}(g)$ in~\eqref{Eqn: explicit formula of tildetilde m for Hilbert transform} by $\widehat{m}(g_x, g_y)$. Then the following estimates hold:
\[
0<\pdv{g_x}\widehat{m}(g_x,g_y)\leq
\begin{cases}
C_1 \log\left( \dfrac{1}{|g_x|} \right) + C_2 \log\left( \dfrac{1}{g_y} \right) + C_3, & \text{if } 0 < g_y \leq \dfrac{1}{2}, \\
C_1 \log\left( \dfrac{1}{|g_x|} \right), & \text{if } g_y > \dfrac{2}{\sqrt{3}},
\end{cases}
\]
and
\[
0 \leq \pdv{g_y}\widehat{m}(g_x,g_y) \leq
\begin{cases}
6, & \text{if } 0 < g_y \leq \dfrac{1}{2}, \\
0, & \text{if } g_y > \dfrac{2}{\sqrt{3}}.
\end{cases}
\]
\end{lemma}
\refstepcounter{subsection}
\subsection*{\thesubsection\quad Proof of Lemma \ref{Lem: dereasing ratio of derivative of m hat}}
\phantomsection
Recall the explicit formula for $\widetilde{\widetilde{m}}(g)$ is given by 
\[
 \widetilde{\widetilde{m}}(g) = \frac{3}{\pi} \int_{A(g_x, g_y) \cap \mathcal{F}} \frac{1}{y^2} \, \dd x \, \dd y,
\]
where $A(g_x,g_y)$ is the region in the upper half-plane defined by
\[
A(g_x, g_y) = \left\{ x + i y \in \mathcal{H} : x + g_x y > -\frac{1}{2} \text{ and } \left( x + g_x y + 1 \right)^2 + \left( g_y y \right)^2 > 1 \right\}
\]
and $\mathcal{F}$ is the region in the upper half-plane defined by
\[
\mathcal{F} = \left\{ x + i y \in \mathcal{H} : \abs{x} \leq \frac{1}{2},\ x^2 + y^2 \geq 1 \right\}.
\]
The region $A(g_x, g_y)$ is the intersection of the area outside the ellipse $E(x,y): (x + g_x y + 1)+(g_y y)^2=1$ and the one side of the line $L(x,y): x+g_x y=-\frac{1}{2}$. The region $A(g_x,g_y)\cap \mathcal{F}$, depicted in Figure \ref{fig: A cap F}, represents the overlapped area determined by a parameterized ellipse, a line, and the fundamental domain $\mathcal{F}$. \par 

\begin{figure}[h]
\centering
\includegraphics[scale=0.4]{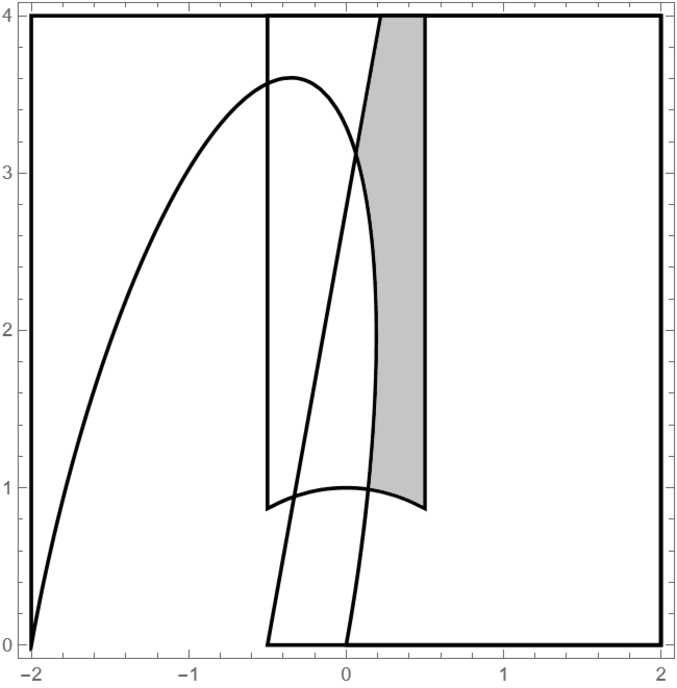}
\caption{The illustration of region $A(g_x,g_y)\cap \mathcal{F}$.}\label{fig: A cap F}
\end{figure}
The region $A(g_x,g_y)\cap \mathcal{F}$ changes as described below: For a fixed $g_y>0$, the ellipse and line undergo a clockwise rotation as $g_x$ decreases from positive to negative infinity. Figure \ref{fig: change of A cap F by g_x for small gy} illustrates these changes for $g_y<\frac{1}{2}$. These figures are also the boundary conditions for $g_x$.
\begin{figure}[h]
\centering
\includegraphics[scale=0.4]{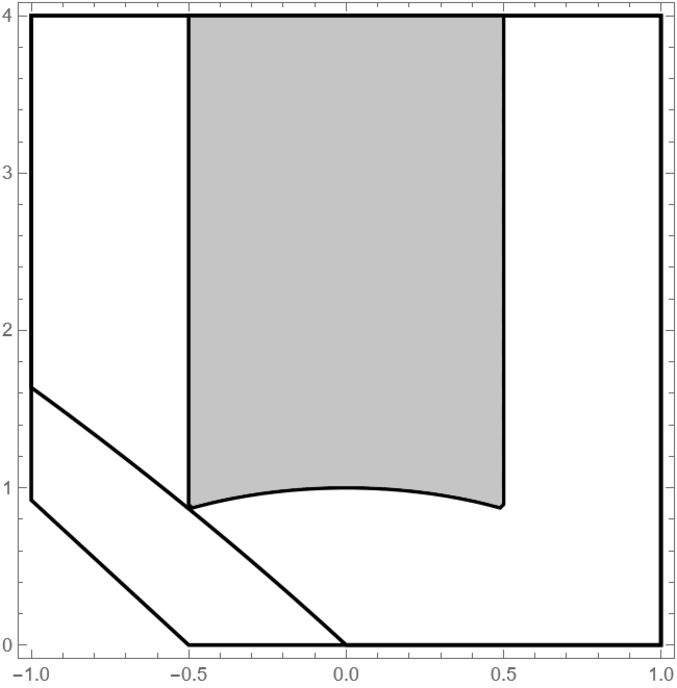}
\includegraphics[scale=0.4]{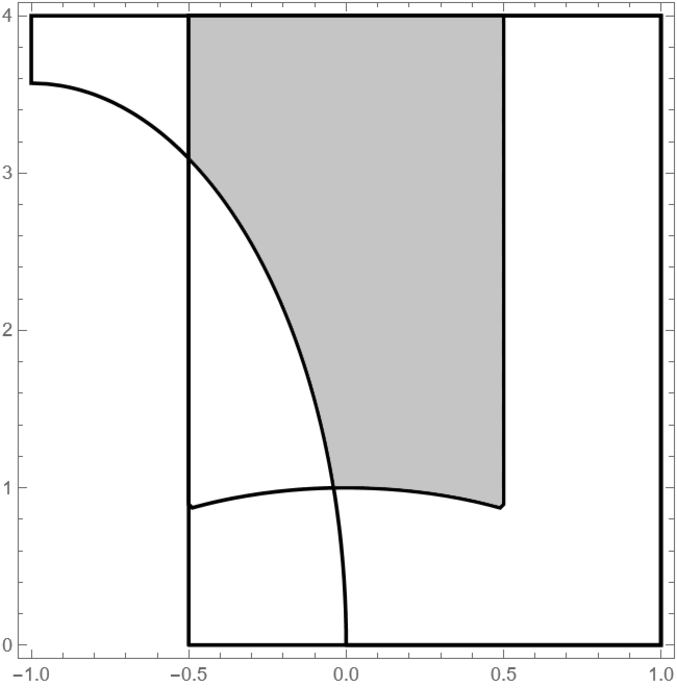}
\includegraphics[scale=0.4]{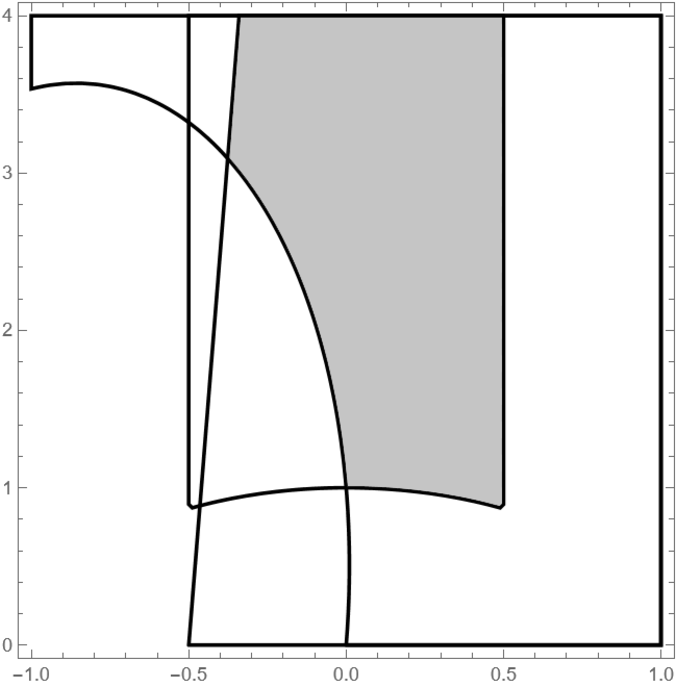}
\includegraphics[scale=0.4]{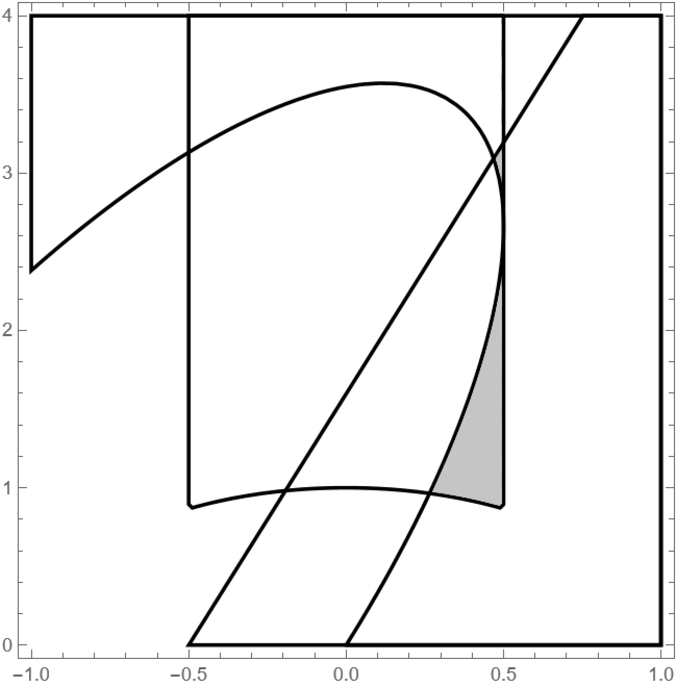}
\includegraphics[scale=0.4]{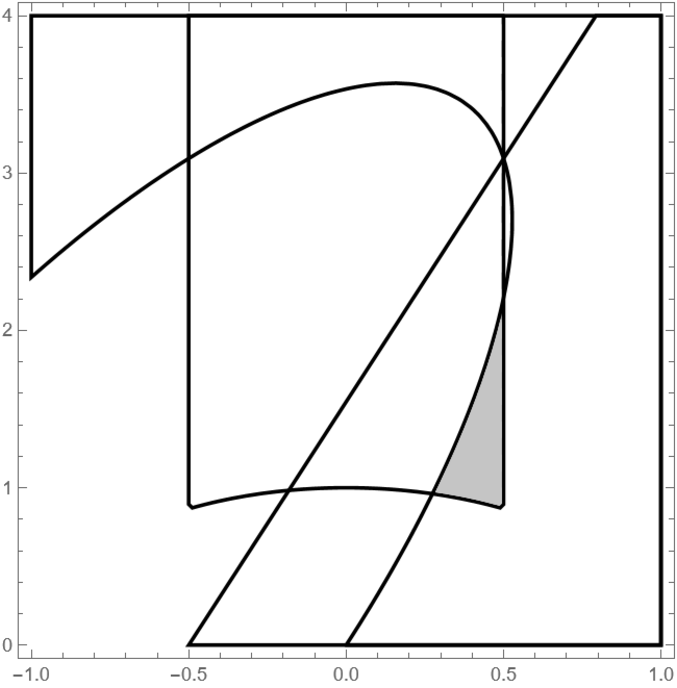}
\includegraphics[scale=0.4]{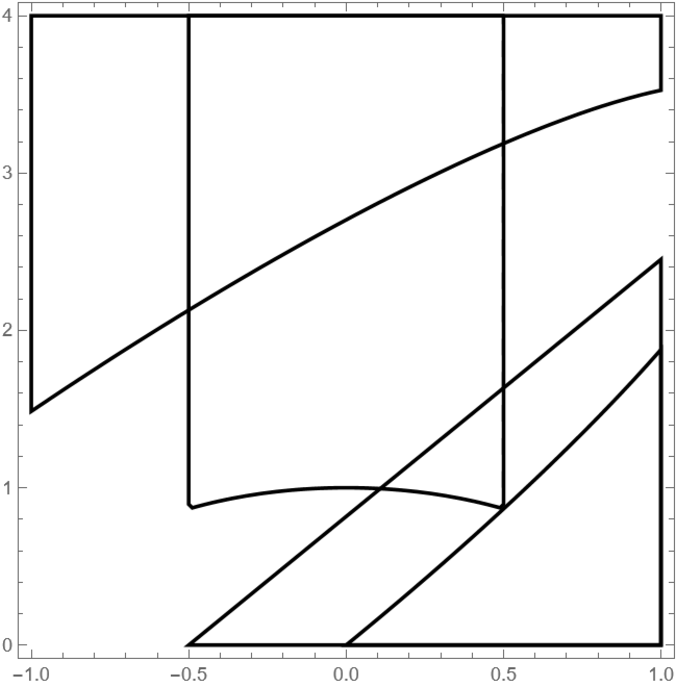}
\caption{The illustration of region $A(g_x,g_y)\cap \mathcal{F}$ as $g_x$ decreases for $g_y<\frac{1}{2}$.}\label{fig: change of A cap F by g_x for small gy}
\end{figure}\par
For $g_y>\frac{2}{\sqrt{3}}$, the ellipse $E$ does not intersect with the region $\mathcal{F}$. Figure \ref{fig: change of A cap F by g_x for large gy} illustrates these changes for $g_y>\frac{2}{\sqrt{3}}$.
\begin{figure}[h]
\centering
\includegraphics[scale=0.4]{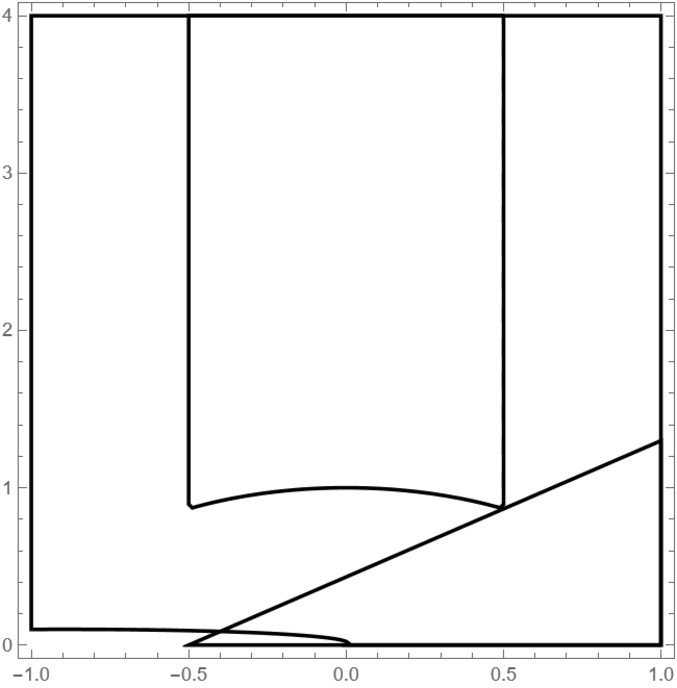}
\includegraphics[scale=0.4]{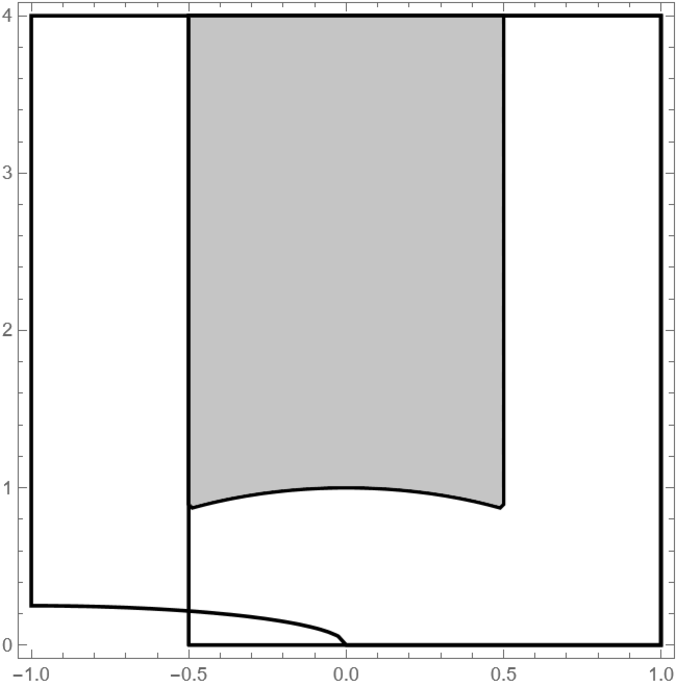}
\caption{The illustration of region $A(g_x,g_y)\cap \mathcal{F}$ as $g_x$ increases with $g_y>\frac{2}{\sqrt{3}}$.}\label{fig: change of A cap F by g_x for large gy}
\end{figure}
\par
For fixed $g_x$, $g_y$, we denote the function of the upper (lower) part of the ellipse by
\begin{align*}
E_{y,upper}(x,g_x,g_y)&=\frac{\sqrt{g_x^2 - x (x + 2) g_y^2} - (x + 1) g_x}{g_x^2 + g_y^2},\\
 E_{y,lower}(x,g_x,g_y)&=-\frac{\sqrt{g_x^2 - x (x + 2) g_y^2} + (x + 1) g_x}{g_x^2 + g_y^2}.
\end{align*}
Also, we denote the function of the right (left) part of the ellipse by
\begin{align*}
E_{x,right}(y, g_x, g_y) &= -y g_x + \sqrt{1 - y^2 g_y^2} - 1, \\
E_{x,left}(y, g_x, g_y) &= -y g_x - \sqrt{1 - y^2 g_y^2} - 1. \\
\end{align*}
Finally, for easy to compute, denote  
\begin{align*}
L_y(x, g_x, g_y) &= -\frac{1 + 2 x}{2 g_x},\\
L_x(y, g_x, g_y) &= -y g_x - \frac{1}{2}.
\end{align*}
\par 

These functions help delineate the integration domain for $\widehat{m}(g_x,g_y)$ when calculating the integral over $A(g_x, g_y) \cap \mathcal{F}$. Solving these equations provides the necessary boundaries to evaluate the integral explicitly. We list the values of $g_x$ at the boundary conditions for $g_y<\frac{1}{2}$ in Figure \ref{fig: change of A cap F by g_x for small gy}:
\begin{enumerate}
\item $g_x=b_2(g_y)=\frac{-1+\sqrt{4-3g_y^2}}{\sqrt{3}}$: The unique $g_x$ such that the point $(-\frac{1}{2},\frac{\sqrt{3}}{2})$ lies on the right part of the ellipse $E$.
\item $g_x=b_3(g_y)=0$: The unique $g_x$ such that the line $L$ is perpendicular to the $x$-axis.
\item $g_x=b_4(g_y)=-1+\sqrt{1-g_y^2}$: The unique $g_x$ such that the point $(0,1)$ lies on the right part of the ellipse $E$.
\item $g_x=b_5(g_y)=-\frac{\sqrt{5}g_y}{2}$: The unique $g_x$ such that the right part of the ellipse $E$ tangent to the straight line $x=\frac{1}{2}$.
\item $g_x=b_6(g_y)=-\frac{2g_y}{\sqrt{3}}$: The unique $g_x$ such that the line $L$, the right part of the ellipse $E$ and the straight line $x=\frac{1}{2}$ meet at one point. 
\item $g_x=b_7(g_y)=-\frac{3-\sqrt{4-3g_y^2}}{\sqrt{3}}$: The unique $g_x$ such that the point $(\frac{1}{2},\frac{\sqrt{3}}{2})$ lies on the right part of the ellipse $E$.
\end{enumerate}
And we list the explicit values of $g_x$ at these boundary conditions for $g_y>\frac{2}{\sqrt{3}}$ in Figure \ref{fig: change of A cap F by g_x for large gy}.
\begin{enumerate}
\item $g_x=b_8(g_y)=-\frac{2}{\sqrt{3}}$: The unique $g_x$ such that the point $(\frac{1}{2},\frac{\sqrt{3}}{2})$ lies on the line $L$.
\item $g_x=b_9(g_y)=0$: The unique $g_x$ such that the line $L$ is perpendicular to the $x$-axis.
\end{enumerate}
We compute $\widehat{m}(g_x,g_y)$ for $g_y<\frac{1}{2}$ in Case 1 to Case 7, $\widehat{m}(g_x,g_y)$ for $g_y>\frac{2}{\sqrt{3}}$ in Case 8.\par 
\noindent\textbf{Case 1: $g_x \in (b_2(g_y), \infty)$}

In this case, the overlap region is depicted in Figure \ref{fig:case1}.

\begin{figure}[h]
\centering
\includegraphics[scale=0.3]{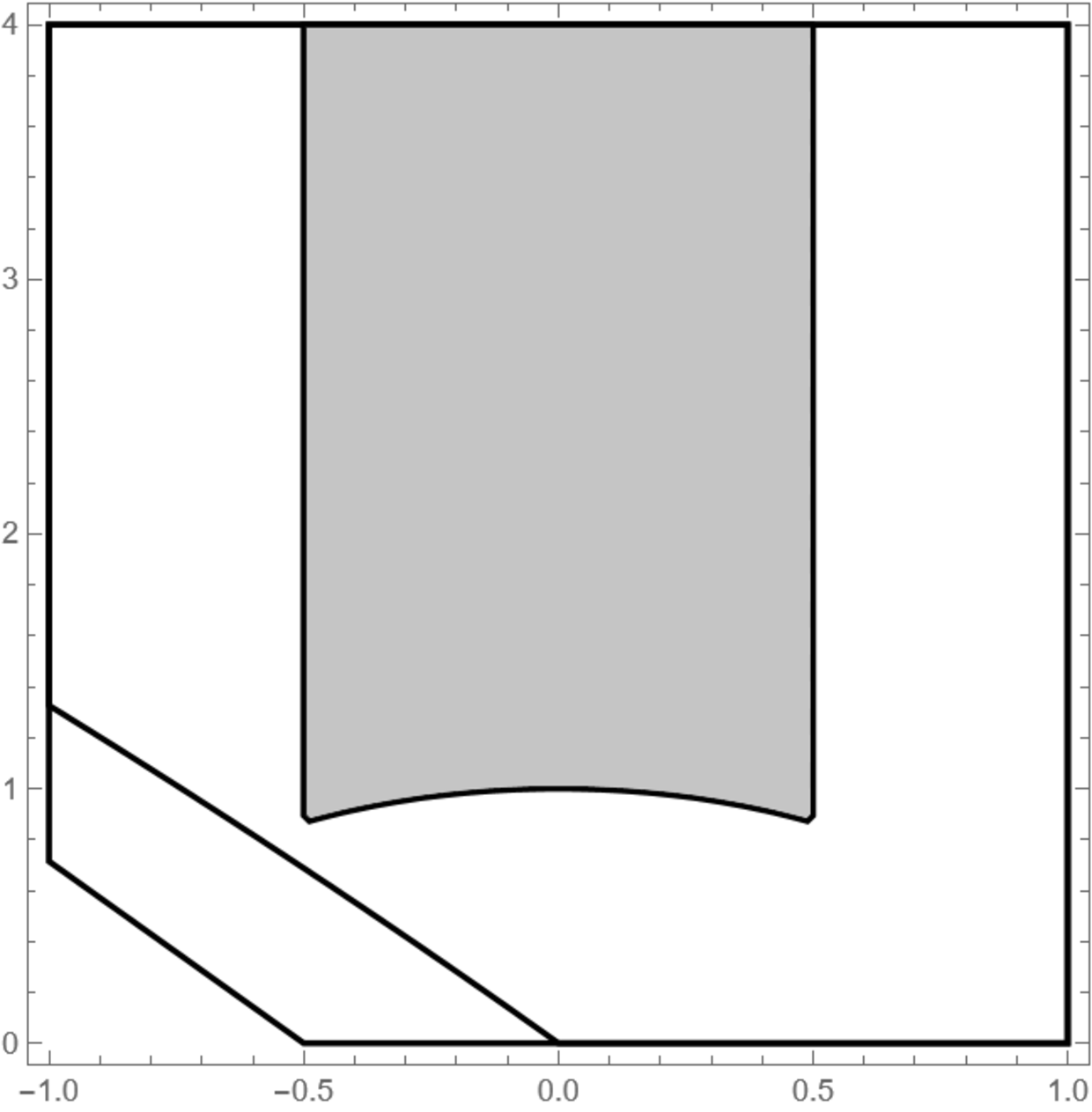}
\caption{Case 1}\label{fig:case1}
\end{figure}

We have $\widehat{m}(g_x, g_y) = 1$ is a constant function. So $\pdv{g_y}\widehat{m}(g_x, g_y) = \pdv{g_x}\widehat{m}(g_x, g_y)=0$.\par

\noindent\textbf{Case 2: $g_x \in (b_3(g_y), b_2(g_y))$}

In this case, the overlap region is depicted in Figure \ref{fig:case2}.
\begin{figure}[h]
\centering
\includegraphics[scale=0.3]{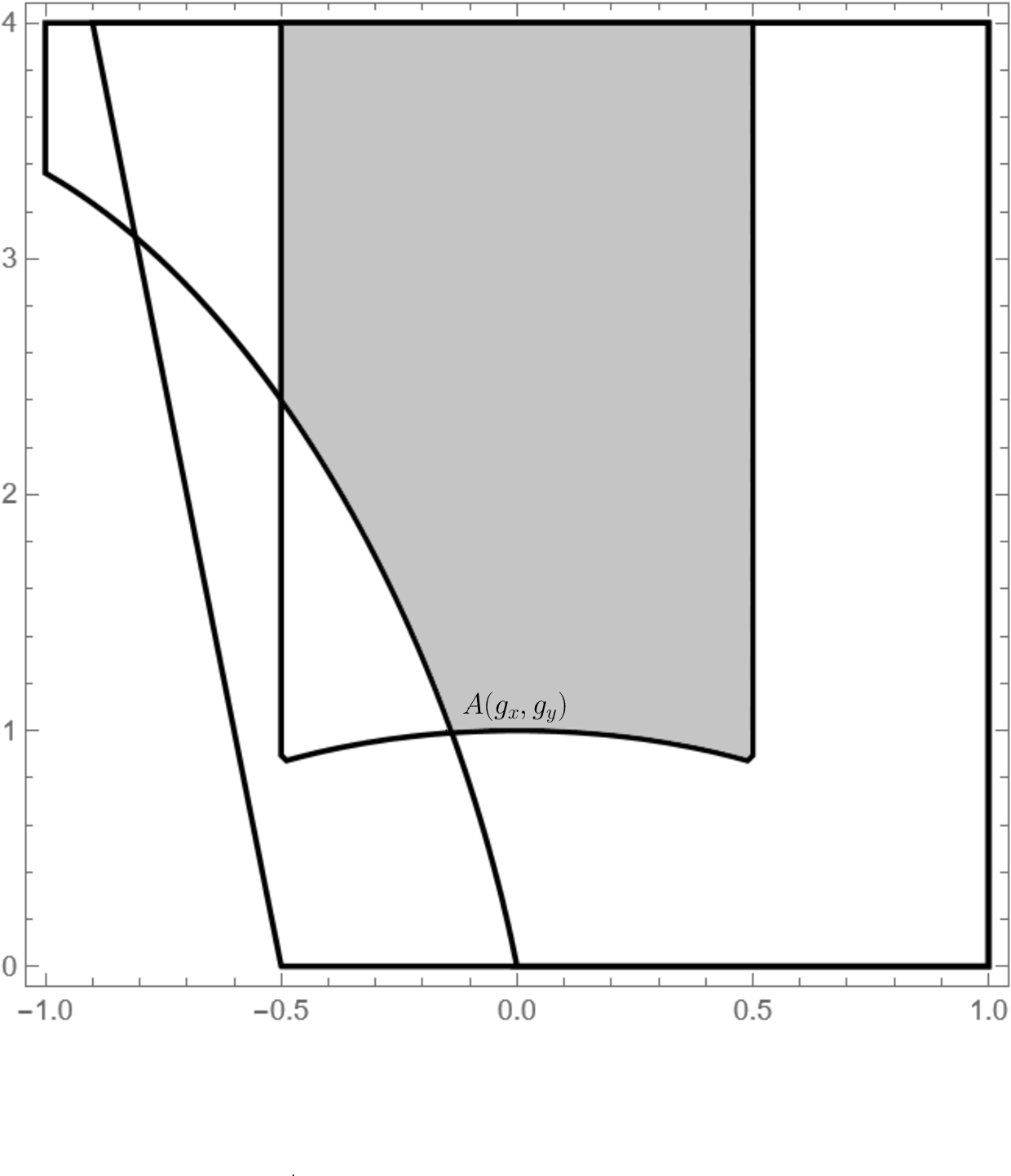}
\caption{Case 2}\label{fig:case2}
\end{figure}

Denote the $x$-coordinate of the intersection point of the ellipse $E$ and the circle $x^2+y^2=1$ by  $A_x(g_x, g_y)$. Since \(E_{y,upper}(x, g_x, g_y)\) is concave with respect to $x$, we have
\[
E_{y,upper}(x, g_x, g_y) \leq \left( \eval{\pdv{x} E_{y,upper}(x, g_x, g_y)}_{x=0} \right) x = -\frac{x}{g_x}.
\]
So the solution of the equation \(-\frac{x}{g_x} = \sqrt{1 - x^2}\) gives an upper bound of \(A_x(g_x, g_y)\),
\[
A_x(g_x, g_y) \leq -\frac{g_x}{\sqrt{1 + g_x^2}} \leq -\frac{g_x}{\sqrt{1 + \frac{1}{3}}} = -\frac{\sqrt{3}g_x}{2}.
\]
The second inequality is due to $g_x^2 \leq \frac{1}{3}$ for $g_x\in (0,b_2(g_y))$.\par
For $g_x\in (b_3(g_y),b_2(g_y))$, we have
\[
\begin{split}
\widehat{m}(g_x, g_y) &= 1 - \frac{3}{\pi} \int_{-\frac{1}{2}}^{A_x(g_x, g_y)} \int_{\sqrt{1 - x^2}}^{E_{y,upper}(x, g_x, g_y)} \frac{1}{y^2} \dd y \dd x \\
&= 1 + \frac{3}{\pi} \int_{-\frac{1}{2}}^{A_x(g_x, g_y)}  \frac{1}{E_{y,upper}(x, g_x, g_y)} - \frac{1}{\sqrt{1 - x^2}} \dd x \\
&= 1 + \frac{3}{\pi} \int_{-\frac{1}{2}}^{A_x(g_x, g_y)} F_2(x, g_x, g_y) \dd x.
\end{split}
\]

It is easy to compute that $\int \pdv{g_x}F_2(x,g_x,g_y)\dd x=\pdv{g_x}\int F_2(x,g_x,g_y)\dd x+C$ and $\int \pdv{g_y}F_{2}(x,g_x,g_y)\dd x=\pdv{g_y}\int F_2(x,g_x,g_y)\dd x+C$. Since $F_2(A_x(g_x,g_y),g_x,g_y)=0$, we have 
\[
\pdv{g_x}\widehat{m}(g_x, g_y)=\frac{3}{\pi}\int_{-\frac{1}{2}}^{A_x(g_x,g_y)}\pdv{g_x}F_2(x,g_x,g_y)\dd x,
\]
\[
\pdv{g_y}\widehat{m}(g_x, g_y)=\frac{3}{\pi}\int_{-\frac{1}{2}}^{A_x(g_x,g_y)}\pdv{g_y}F_{2}(x,g_x,g_y)\dd x.
\]\par

In the region of integration in this case, we have \(x \in (-\frac{1}{2}, 0 )\) and \(g_x > 0\). Therefore, \(\pdv{g_x}F_2(x, g_x, g_y)\) can be estimated by
\[
0 < \pdv{g_x}F_2(x, g_x, g_y) = -\frac{ \frac{g_x}{\sqrt{g_x^2 - g_y^2 x (x + 2)}} + x + 1 }{ x (x + 2) } \leq -\frac{ 1 + x + 1 }{ x (x + 2) } = -\frac{1}{x}.
\]
So we have 
\begin{equation}\label{ineq: estimation of m hat 2 by gx}
0<\pdv{g_x}\widehat{m}(g_x, g_y)\leq \frac{3}{\pi}\int_{-\frac{1}{2}}^{ \frac{-\sqrt{3}g_x}{2} } -\frac{1}{x}\dd x= -\frac{3\log(3)}{2\pi}+\frac{3}{\pi}\log(\frac{1}{g_x}).
\end{equation}
Now we estimate $\pdv{g_y}\widehat{m}(g_x, g_y)$. This estimation will occur many times in subsequent cases. The main estimation we use is that 
\begin{equation}\label{ineq: estimation of m hat 2 by gy}
\begin{split}
0<\pdv{g_y}\widehat{m}(g_x, g_y)=&\frac{3}{\pi}\int_{x_l}^{x_u}\pdv{g_y}F_{2}(x, g_x, g_y)\dd x = \frac{3}{\pi}\int_{x_l}^{x_u}\frac{1}{\sqrt{(\frac{g_x}{g_y} )^2 - x(x+2)}}\dd x \\
= &\frac{3}{\pi}\left( \eval{ -\tan^{-1}\left( \frac{\sqrt{(\frac{g_x}{g_y})^2 - 2x - x^2}}{1+x} \right) }_{x_l}^{x_u} \right) \leq 3,
\end{split}
\end{equation}
for any $x_l, x_u$ in the domain of the integrand of the above integration.

\noindent\textbf{Case 3: $g_x\in (b_4(g_y),b_3(g_y))$}\par

In this case, the overlap region is depicted in Figure \ref{fig:case3}.
\begin{figure}[h]
\includegraphics[scale=0.3]{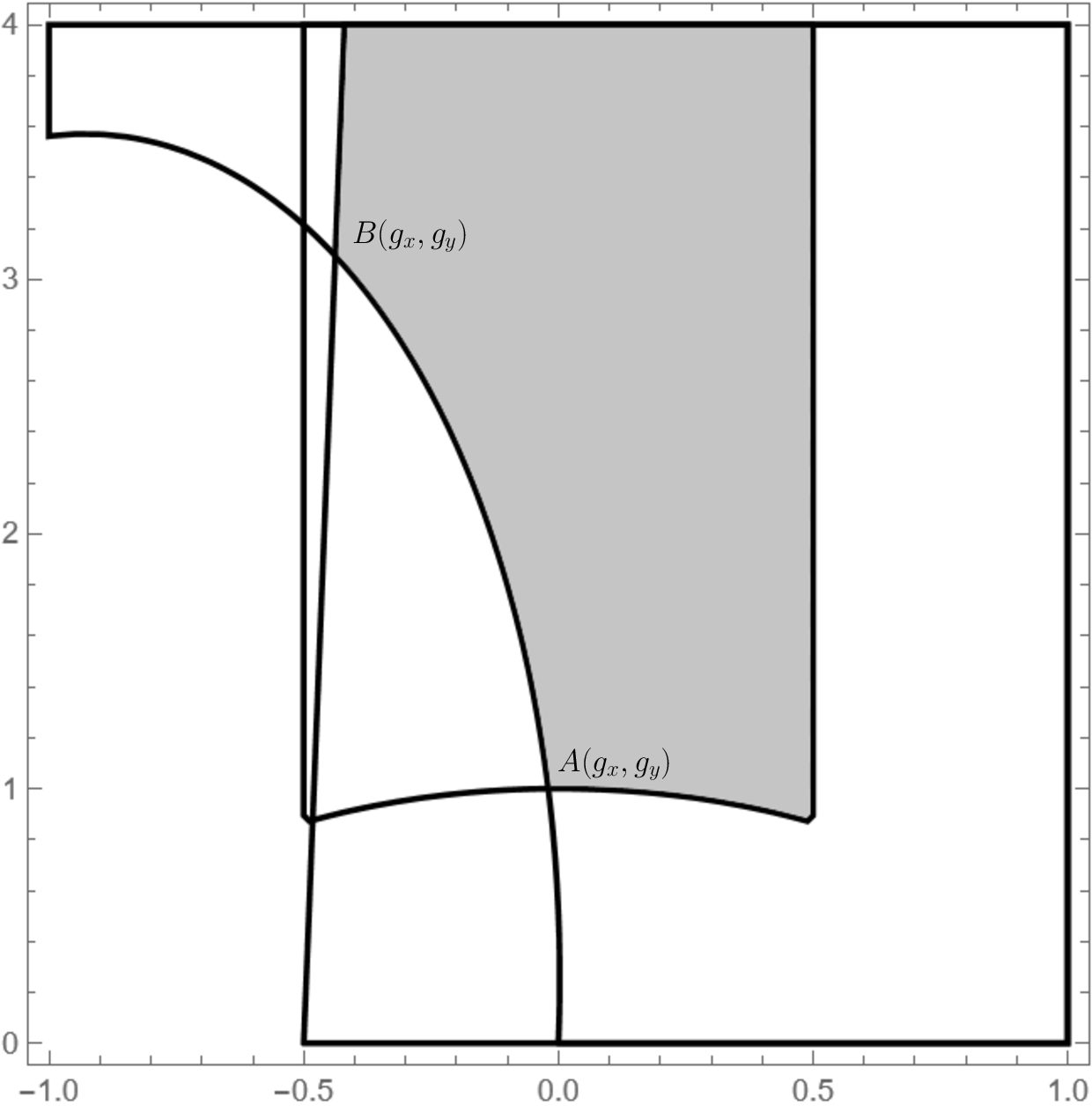}
\centering
\caption{Case 3}\label{fig:case3}
\end{figure}\par 
Let $B_x(g_x,g_y)$ be the $x$-coordinate of the intersection point of the ellipse $E$ and the line $L$, i.e.,
\[
L_y(B_x(g_x,g_y),g_x,g_y)=E_{y,upper}(B_x(g_x,g_y),g_x,g_y).
\]
This yields
\[
B_x(g_x,g_y)=-\frac{\sqrt{3} g_x+g_y}{2 g_y}>-\frac{1}{2}.
\]
For $g_x\in (b_4(g_y),b_3(g_y))$, we have
\[
\begin{split}
\widehat{m}(g_x,g_y)=&\frac{3}{\pi}\int_{B_x(g_x,g_y)}^{\frac{1}{2}}\int_{\sqrt{1-x^2}}^{L_y(x,g_x,g_y)}\frac{1}{y^2}\dd y\dd x-\frac{3}{\pi}\int_{B_x(g_x,g_y)}^{A_x(g_x,g_y)}\int_{\sqrt{1-x^2}}^{E_{y,upper}(x,g_x,g_y)}\frac{1}{y^2}\dd y\dd x\\
=&-\frac{3}{\pi}\int_{B_x(g_x,g_y)}^{\frac{1}{2}}\frac{1}{L_y(x,g_x,g_y)}-\frac{1}{\sqrt{1-x^2}}\dd x+\frac{3}{\pi}\int_{B_x(g_x,g_y)}^{A_x(g_x,g_y)}\frac{1}{E_{y,upper}(x,g_x,g_y)}-\frac{1}{\sqrt{1-x^2}}\dd x\\
=&\frac{3}{\pi}\int_{B_x(g_x,g_y)}^{\frac{1}{2}}F_{3,1}(x,g_x,g_y)\dd x+\frac{3}{\pi}\int_{B_x(g_x,g_y)}^{A_x(g_x,g_y)}F_{2}(x,g_x,g_y)\dd x\\
=& \widehat{m}_{3,1}(g_x,g_y)+\widehat{m}_{3,2}(g_x,g_y).
\end{split}
\]\par
It is easy to compute that $\int \pdv{g_x}F_{3,1}(x,g_x,g_y)\dd x=\pdv{g_x}\int F_{3,1}(x,g_x,g_y)\dd x+C$ and $\int \pdv{g_y}F_{3,1}(x,g_x,g_y)\dd x=\pdv{g_y}\int F_{3,1}(x,g_x,g_y)\dd x+C=C$. Since 
\[
F_{3,1}(B_x(g_x,g_y),g_x,g_y)+F_{2}(B_x(g_x,g_y),g_x,g_y)=0,
\]
we have
\[
\begin{split}
\pdv{g_x}\widehat{m}(g_x, g_y)=&\frac{3}{\pi}\int_{B_x(g_x,g_y)}^{\frac{1}{2}}\pdv{g_x}F_{3,1}(x,g_x,g_y)\dd x+\frac{3}{\pi}\int_{B_x(g_x,g_y)}^{A_x(g_x,g_y)}\pdv{g_x}F_{2}(x,g_x,g_y)\dd x\\
=&\frac{3}{\pi}\int_{B_x(g_x,g_y)}^{\frac{1}{2}}\frac{2}{1+2x}\dd x+\frac{3}{\pi}\int_{B_x(g_x,g_y)}^{A_x(g_x,g_y)}\pdv{g_x}F_{2}(x,g_x,g_y)\dd x,\\
\end{split}
\]
\[
\pdv{g_y}\widehat{m}(g_x, g_y)=\frac{3}{\pi}\int_{B_x(g_x,g_y)}^{A_x(g_x,g_y)}\pdv{g_y}F_{2}(x,g_x,g_y)\dd x.\\
\]\par
We have
\begin{equation}
0<\int_{B_x(g_x,g_y)}^{\frac{1}{2}}\frac{2}{1+2x}\dd x\leq\log(-\frac{2g_y}{\sqrt{3}g_x})\leq \log(-\frac{1}{\sqrt{3}g_x}).
\end{equation}

Let $b(g_x,g_y)=\frac{g_y^2}{g_x^2}$. We have
\[
b(g_x,g_y)\geq \frac{g_y^2}{b_4(g_y)^2}= \frac{g_y^2}{(1-\sqrt{1-g_y^2})^2}\geq 7+4\sqrt{3}.
\]
So 
\[
\pdv{g_x}F_{2}(x,g_x,g_y)=\frac{-\frac{1}{\sqrt{1-(\frac{g_y}{g_x})^2 x (x+2)}}+x+1}{-x(x+2)}=\frac{-\frac{1}{\sqrt{1-b(g_x,g_y) x (x+2)}}+x+1}{-x(x+2)}>0.
\]

We have
\[
\begin{split}
0<\int_{B_x(g_x,g_y)}^{A_x(g_x,g_y)}\pdv{g_x}F_{2}(x,g_x,g_y)\dd x\leq &\int_{-\frac{1}{2}}^{0}\frac{-\frac{1}{\sqrt{1-b(g_x,g_y) x (x+2)}}+x+1}{-x (x+2)}\dd x\\
=&(\log(\frac{1+b(g_x,g_y)}{2})-\log(\frac{2}{3}(\sqrt{4+3b(g_x,g_y)}-1)))\\
\leq &\log(b(g_x,g_y))= 2\log(-\frac{g_y}{g_x})\leq 2\log(-\frac{1}{2g_x}).
\end{split}
\]
Therefore, we have 
\begin{equation}\label{ineq: estimation of m hat 3 by gx}
0<\pdv{g_x}\widehat{m}(g_x, g_y)\leq \frac{9}{\pi}\log \left(-\frac{1}{g_x}\right)-\frac{3\log (3)}{2\pi}-\frac{3}{\pi}\log (4).
\end{equation}
Also, by estimation \eqref{ineq: estimation of m hat 2 by gy}, we have \par 
\begin{equation}\label{ineq: estimation of m hat 3 by gy}
0<\pdv{g_y}\widehat{m}(g_x, g_y)=\frac{3}{\pi}\int_{B_x(g_x,g_y)}^{A_x(g_x,g_y)}\pdv{g_y}F_{2}(x,g_x,g_y)\dd x\leq 3.
\end{equation}

\noindent\textbf{Case 4: $g_x\in (b_5(g_y),b_4(g_y))$}\par
In this case, the overlap region is depicted in Figure \ref{fig:case4}.\par
\begin{figure}[h]
\includegraphics[scale=0.3]{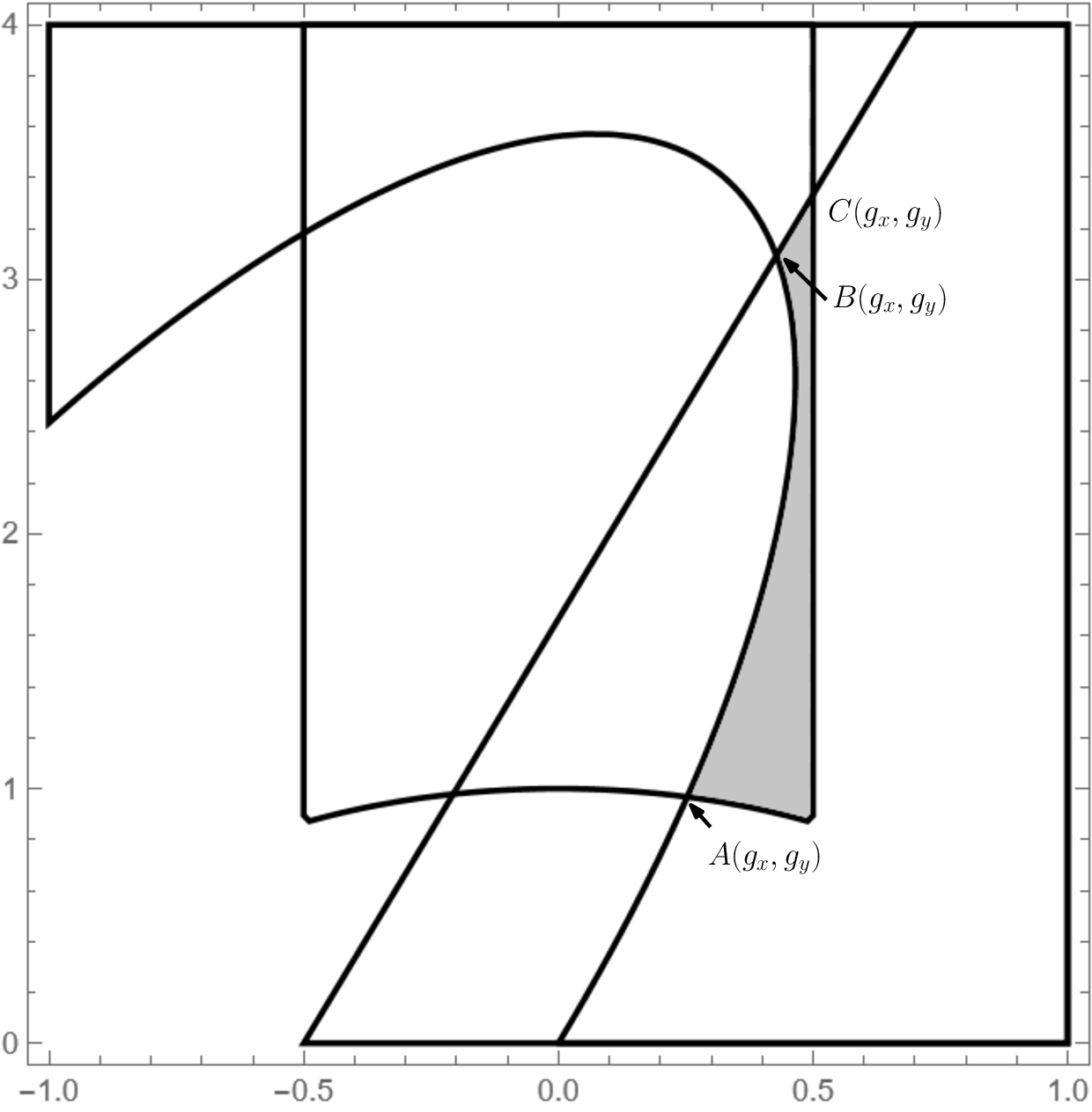}
\centering
\caption{Case 4}\label{fig:case4}
\end{figure}
Let $A_y(g_x, g_y)$ be the $y$-coordinate of the intersection point of the ellipse $E$ and the circle $x^2+y^2=1$, $B_y(g_x,g_y)$ be the $y$-coordinate of the intersection point of the ellipse $E$ and the line $L$, i.e.,
\[
L_x(B_y(g_x,g_y),g_x,g_y)=E_{x,right}(B_y(g_x,g_y),g_x,g_y),
\]
and $C_y(g_x,g_y)$ be the $y$-coordinate of the intersection point of the line $L$ and line $x=\frac{1}{2}$, i.e.,
\[
L_x(C_y(g_x,g_y),g_x,g_y)=\frac{1}{2}.
\]
These yield 
\[
A_y(g_x,g_y)\geq \frac{\sqrt{3}}{2},\quad B_y(g_x,g_y)=\frac{\sqrt{3}}{2g_y}>\frac{\sqrt{3}}{2},\quad C_y(g_x,g_y)=-\frac{1}{g_x}>\frac{\sqrt{3}}{2}.
\]
For $g_x\in (b_5(g_y),b_4(g_y))$, 
\[
\begin{split}
\widehat{m}(g_x,g_y)=&\frac{3}{\pi}\int_{\frac{\sqrt{3}}{2}}^{A_y(g_x,g_y)}\int_{\sqrt{1-y^2}}^{\frac{1}{2}}\frac{1}{y^2}\dd x\dd y+\frac{3}{\pi}\int_{A_y(g_x,g_y)}^{B_y(g_x,g_y)}\int_{E_{x,right}(y,g_x,g_y)}^{\frac{1}{2}}\frac{1}{y^2}\dd x\dd y\\
&~~+\frac{3}{\pi}\int_{B_y(g_x,g_y)}^{C_y(g_x,g_y)}\int_{L_x(y,g_x,g_y)}^{\frac{1}{2}}\frac{1}{y^2}\dd x\dd y\\
=&\frac{3}{\pi}\int_{\frac{\sqrt{3}}{2}}^{A_y(g_x,g_y)}\frac{\frac{1}{2}-\sqrt{1-y^2}}{y^2}\dd y+\frac{3}{\pi}\int_{A_y(g_x,g_y)}^{B_y(g_x,g_y)}\frac{\frac{1}{2}-E_{x,right}(y,g_x,g_y)}{y^2}\dd y\\
&~~+\frac{3}{\pi}\int_{B_y(g_x,g_y)}^{C_y(g_x,g_y)}\frac{\frac{1}{2}-L_x(y,g_x,g_y)}{y^2}\dd y\\
=&\frac{3}{\pi}\int_{\frac{\sqrt{3}}{2}}^{A_y(g_x,g_y)}F_{4,1}(y,g_x,g_y)\dd y+\frac{3}{\pi}\int_{A_y(g_x,g_y)}^{B_y(g_x,g_y)}F_{4,2}(y,g_x,g_y)\dd y\\
&~~+\frac{3}{\pi}\int_{B_y(g_x,g_y)}^{C_y(g_x,g_y)}F_{4,3}(y,g_x,g_y)\dd y.\\
\end{split}
\]

It is easy to compute that
\[
\begin{cases}
\int \pdv{g_x}F_{4,1}(x,g_x,g_y)\dd x=\pdv{g_x}\int F_{4,1}(x,g_x,g_y)\dd x+C=C,\\
\int \pdv{g_y}F_{4,1}(x,g_x,g_y)\dd x=\pdv{g_y}\int F_{4,1}(x,g_x,g_y)\dd x+C=C,\\
\int \pdv{g_x}F_{4,2}(x,g_x,g_y)\dd x=\pdv{g_x}\int F_{4,2}(x,g_x,g_y)\dd x+C=\log(x)+C,\\
\int \pdv{g_y}F_{4,2}(x,g_x,g_y)\dd x=\pdv{g_y}\int F_{4,2}(x,g_x,g_y)\dd x+C=\tan ^{-1}\left(\frac{g_y x}{\sqrt{1-g_y^2 x^2}}\right)+C,\\
\int \pdv{g_x}F_{4,3}(x,g_x,g_y)\dd x=\pdv{g_x}\int F_{4,3}(x,g_x,g_y)\dd x+C=\log(x)+C,\\
\int \pdv{g_y}F_{4,3}(x,g_x,g_y)\dd x=\pdv{g_y}\int F_{4,3}(x,g_x,g_y)\dd x+C=C.
\end{cases}
\]
Since
\[
\begin{cases}
F_{4,1}(A_y(g_x,g_y),g_x,g_y)=F_{4,2}(A_y(g_x,g_y),g_x,g_y),\\
F_{4,2}(B_y(g_x,g_y),g_x,g_y)=F_{4,3}(B_y(g_x,g_y),g_x,g_y),\\
F_{4,3}(C_y(g_x,g_y),g_x,g_y)=0,
\end{cases}
\]
we have
\[
\begin{split}
\pdv{g_x}\widehat{m}(g_x, g_y)=&\frac{3}{\pi}\int_{A_y(g_x,g_y)}^{B_y(g_x,g_y)}\pdv{g_x}F_{4,2}(y,g_x,g_y)\dd y+\frac{3}{\pi}\int_{B_y(g_x,g_y)}^{C_y(g_x,g_y)}\pdv{g_x}F_{4,3}(y,g_x,g_y)\dd y\\
=&\frac{3}{\pi}\int_{A_y(g_x,g_y)}^{C_y(g_x,g_y)}\frac{1}{y}\dd y,
\end{split}
\]
\[
\pdv{g_y}\widehat{m}(g_x, g_y)=\frac{3}{\pi}\int_{A_y(g_x,g_y)}^{B_y(g_x,g_y)}\pdv{g_y}F_{4,2}(y,g_x,g_y)\dd y.
\]
\par
We have 
\begin{equation}\label{ineq: estimation of m hat 4 by gx}
0<\pdv{g_x}\widehat{m}(g_x, g_y)= \frac{3}{\pi}\int_{A_y(g_x,g_y)}^{C_y(g_x,g_y)}\frac{1}{y}\dd y    \leq\frac{3}{\pi}\int_{\frac{\sqrt{3}}{2}}^{C_y(g_x,g_y)}\frac{1}{y}\dd y=\frac{3}{\pi}\log(-\frac{2}{\sqrt{3}g_x}).
\end{equation}
And similar to the estimation in inequality \eqref{ineq: estimation of m hat 2 by gy}, we have
\begin{equation}\label{ineq: estimation of m hat 4 by gy}
0<\pdv{g_y}\widehat{m}(g_x, g_y)=\frac{3}{\pi}\int_{A_y(g_x,g_y)}^{B_y(g_x,g_y)}\pdv{g_y}F_{4,2}(x,g_x,g_y)\dd y=\frac{3}{\pi}\eval{\tan ^{-1}\left(\frac{g_y x}{\sqrt{1-g_y^2 x^2}}\right)}_{A_y(g_x,g_y)}^{B_y(g_x,g_y)}\leq 3.
\end{equation}

\noindent\textbf{Case 5: $g_x\in (b_6(g_y),b_5(g_y))$}\par
In this case, the overlap region is depicted in Figure \ref{fig:case5}.
\begin{figure}[h]
\includegraphics[scale=0.3]{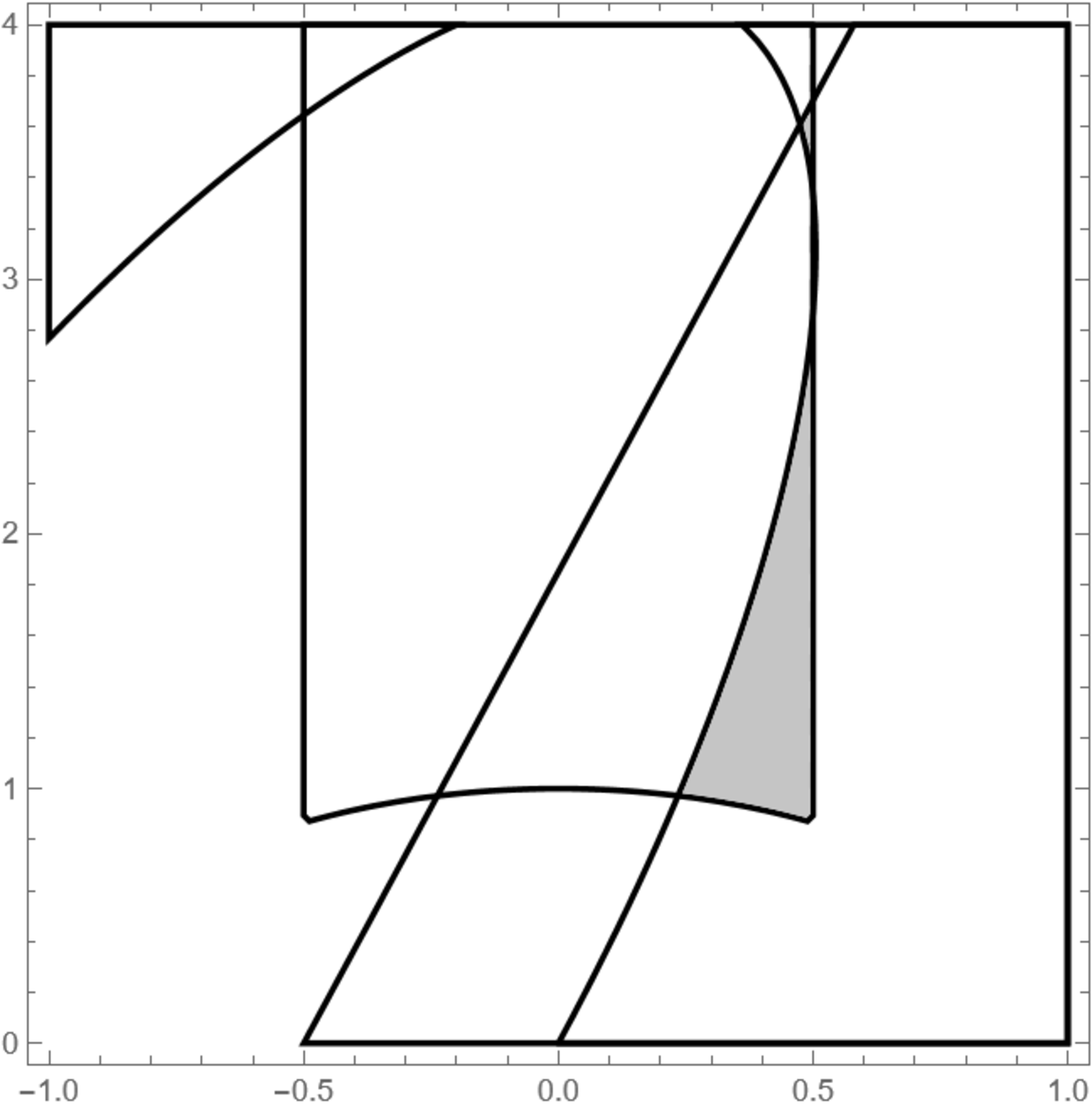}
\includegraphics[scale=0.3]{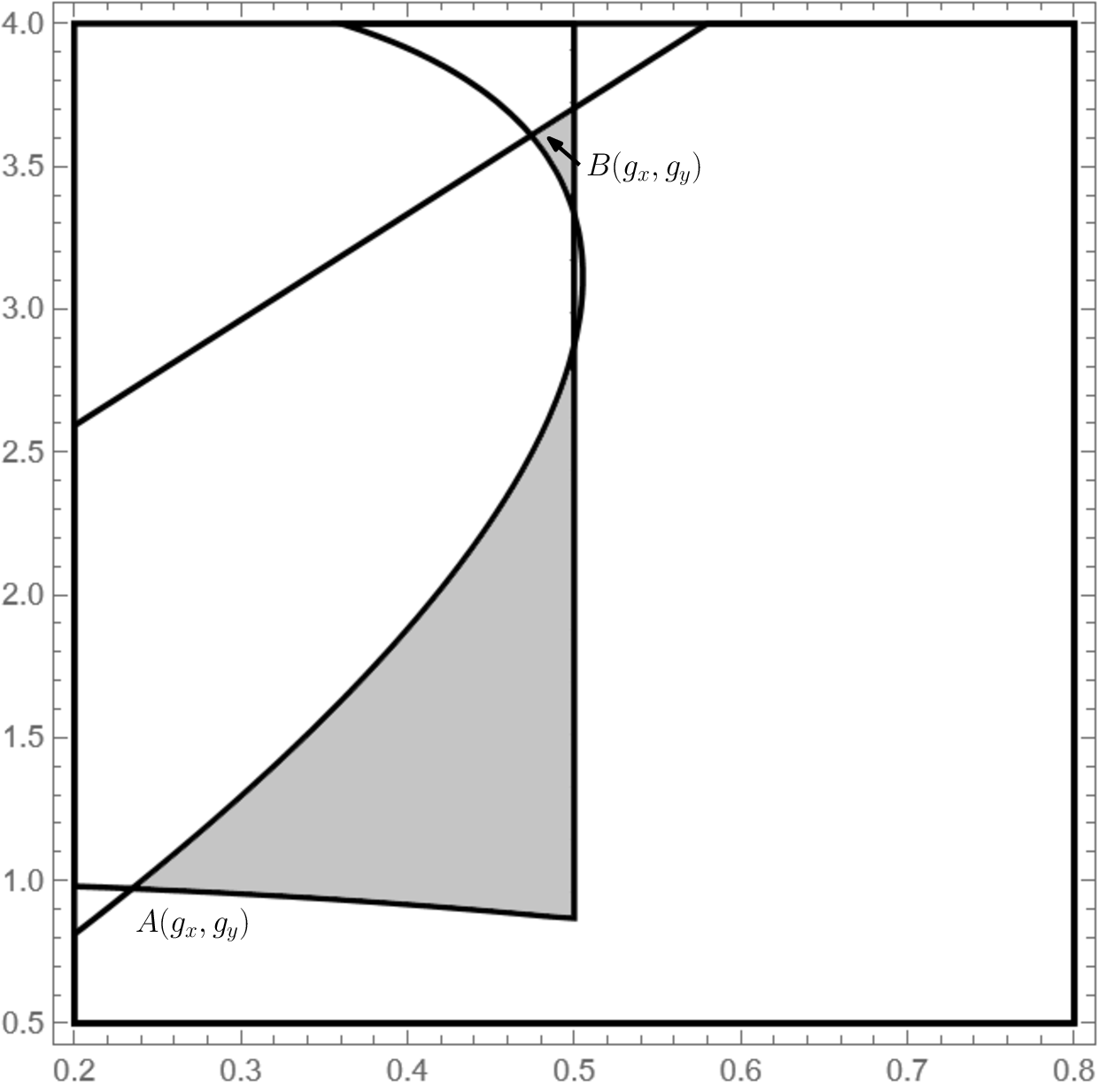}
\centering
\caption{Case 5}\label{fig:case5}
\end{figure}\par

Since the ellipse $E$ is convex, we have 
\[
E_{y,lower}(x,g_x,g_y)\leq \frac{B_y(g_x,g_y)}{L_x(B_y(g_x,g_y),g_x,g_y)}x=-\frac{\sqrt{3}}{\sqrt{3}g_x+g_y}x.
\]
Therefore, the solution of the equation $-\frac{\sqrt{3}}{\sqrt{3}g_x+g_y}x=\sqrt{1-x^2}$ gives a lower bound of $A_x(g_x,g_y)$,
\[
A_x(g_x,g_y)\geq  \frac{1}{\sqrt{\frac{3}{(\sqrt{3}g_x+g_y)^2}+1}}.
\]
Since $\sqrt{3}g_x+g_y\in (-2g_y+g_y,-\frac{\sqrt{15}}{2}g_y+g_y)$, we have 
\[
\frac{ (\sqrt{3}g_x+g_y)^2}{3}\geq \frac{(1-\frac{\sqrt{15}}{2})^2}{3}g_y^2>\frac{1}{4}g_y^2,
\]
and hence
\[
A_x(g_x,g_y)\geq \frac{1}{\sqrt{\frac{4}{g_y^2}+1}}\geq \frac{g_y}{4}.
\]\par
The lower bound of $B_x(g_x,g_y)$ can be estimated by
\[
\begin{split}
B_x(g_x,g_y)=&-\frac{\sqrt{3}g_x+g_y}{2g_y}\geq \frac{1}{4}(\sqrt{15}-2).
\end{split}
\]

For $g_x\in (b_6(g_y),b_5(g_y))$, $\widehat{m}$ is
\[
\begin{split}
\widehat{m}(g_x,g_y)=&\frac{3}{\pi}\int_{A_x(g_x,g_y)}^{\frac{1}{2}}\int_{\sqrt{1-x^2}}^{E_{y,lower}(x,g_x,g_y)}\frac{1}{y^2}\dd y\dd x+\frac{3}{\pi}\int_{B_x(g_x,g_y)}^{\frac{1}{2}}\int_{E_{y,upper}(x,g_x,g_y)}^{L_y(y,g_x,g_y)}\frac{1}{y^2}\dd y\dd x\\
=&-\frac{3}{\pi}\int_{A_x(g_x,g_y)}^{\frac{1}{2}}\frac{1}{E_{y,lower}(x,g_x,g_y)}-\frac{1}{\sqrt{1-x^2}}\dd x-\frac{3}{\pi}\int_{B_x(g_x,g_y)}^{\frac{1}{2}}\frac{1}{L_y(x,g_x,g_y)}-\frac{1}{E_{y,upper}(x,g_x,g_y)}\dd x\\
=&\frac{3}{\pi}\int_{A_x(g_x,g_y)}^{\frac{1}{2}}F_{5,1}(x,g_x,g_y)\dd x+\frac{3}{\pi}\int_{B_x(g_x,g_y)}^{\frac{1}{2}}F_{5,2}(x,g_x,g_y)\dd x.
\end{split}
\]

It is easy to compute that
\[
\begin{cases}
\int \pdv{g_x}F_{5,1}(x,g_x,g_y)\dd x=\pdv{g_x}\int F_{5,1}(x,g_x,g_y)\dd x+C,\\
\int \pdv{g_y}F_{5,1}(x,g_x,g_y)\dd x=\pdv{g_y}\int F_{5,1}(x,g_x,g_y)\dd x+C,\\
\int \pdv{g_x}F_{5,2}(x,g_x,g_y)\dd x=\pdv{g_x}\int F_{5,2}(x,g_x,g_y)\dd x+C,\\
\int \pdv{g_y}F_{5,2}(x,g_x,g_y)\dd x=\pdv{g_y}\int F_{5,2}(x,g_x,g_y)\dd x+C.\\
\end{cases}
\]

Since
\[
F_{5,1}(A_x(g_x,g_y))=0,\quad F_{5,2}(B_x(g_x,g_y))=0,
\]
we have
\[
\pdv{g_x}\widehat{m}(g_x, g_y)=\frac{3}{\pi}\int_{A_x(g_x,g_y)}^{\frac{1}{2}}\pdv{g_x}F_{5,1}(x,g_x,g_y)\dd x+\frac{3}{\pi}\int_{B_x(g_x,g_y)}^{\frac{1}{2}}\pdv{g_x}F_{5,2}(x,g_x,g_y)\dd x,\\
\]
\[
\pdv{g_y}\widehat{m}(g_x, g_y)=\frac{3}{\pi}\int_{A_x(g_x,g_y)}^{\frac{1}{2}}\pdv{g_y}F_{5,1}(x,g_x,g_y)\dd x+\frac{3}{\pi}\int_{B_x(g_x,g_y)}^{\frac{1}{2}}\pdv{g_y}F_{5,2}(x,g_x,g_y)\dd x.\\
\]
We have 
\[
\pdv{g_x}F_{5,1}(x,g_x, g_y)=\frac{-\frac{g_x}{\sqrt{g_x^2-g_y^2 x (x+2)}}+x+1}{x (x+2)}=\frac{\frac{1}{\sqrt{1-b(g_x,g_y) x (x+2)}}+x+1}{x (x+2)}>0.
\]
Since $b(g_x,g_y)=\frac{g_y^2}{g_x^2}\in (\frac{3}{4},\frac{4}{5})\subset (0,\frac{4}{5})$, we have $ 1+x-\sqrt{1-bx(x+2)}>0$ and hence
\[
\begin{split}
\int_{A_x(g_x,g_y)}^{\frac{1}{2}}\pdv{g_x}F_{5,1}(x,g_x,g_y)\dd x\leq&\int_{A_x(g_x,g_y)}^{\frac{1}{2}}\frac{\frac{1}{\sqrt{1-b(g_x,g_y) x (x+2)}}+x+1}{x (x+2)}\dd x\\
=&\eval{\log(1+x-\sqrt{1-b(g_x,g_y)x(x+2)})}_{A_x(g_x,g_y)}^{\frac{1}{2}}.
\end{split}
\]\par
Let $f(b,x)=1+x-\sqrt{1-bx(x+2)}$. For $b>0$ and $x\in (0,\frac{1}{2})$, we have 
\[
{\pdv{f}{b}}(b,x)=\frac{x (x+2)}{2 \sqrt{1-b x (x+2)}}>0,\quad {\pdv{f}{x}}(b,x)=\frac{b(x+1)}{\sqrt{1-b x (x+2)}}+1>0.
\]
So we have
\[
\begin{split}
\eval{\log(f(b(g_x,g_y),x))}_{A_x(g_x,g_y)}^{\frac{1}{2}}\leq &\log(f(\frac{4}{5},\frac{1}{2}))-\log(f(\frac{3}{4},\frac{g_y}{4}))\\
=&\log (12)-\log \left(2 g_y-\sqrt{64-3 g_y (g_y+8)}+8\right)\\
=&\left(\log \left(\frac{24}{7}\right)+\log (\frac{1}{g_y})\right)-\frac{3 g_y}{32}+O\left(g_y^2\right).
\end{split}
\]
Therefore,
\[
\begin{split}
\pdv{g_x}F_{5,2}(x,g_x,g_y)=&\frac{2}{1+2x}-\frac{1+x-\frac{1}{\sqrt{1-b(g_x,g_y) x (x+2)}}}{x (x+2)}\\
\geq& \frac{2}{1+2x}-\frac{1+x-1}{x (x+2)}>0.
\end{split}
\]
Thus, we have 
\[
\begin{split}
0<\int_{B_x(g_x,g_y)}^{\frac{1}{2}}\pdv{g_x}F_{5,2}(x,g_x,g_y)\dd x=& \int_{B_x(g_x,g_y)}^{\frac{1}{2}}\frac{2}{1+2x}-\frac{1+x-\frac{1}{\sqrt{1-b(r,\theta) x (x+2)}}}{x (x+2)}\dd x\\
=&\eval{\left[ \log(f(b(r,\theta),x))+\log(\frac{1+2x}{x(x+2)}) \right]}_{B_x(g_x,g_y)}^{\frac{1}{2}}\\
\leq&\log(f(\frac{4}{5},\frac{1}{2}))-\log(f(\frac{3}{4},\frac{1}{4} \left(\sqrt{15}-2\right)))\\
=&\log \left(\frac{12}{2 \sqrt{15}-\sqrt{55-12 \sqrt{15}}+4}\right).
\end{split}
\]
The last inequality holds since $\log(\frac{1+2x}{x(x+2)})$ decreases in $(0,\frac{1}{2})$. Therefore, 
\begin{equation}\label{ineq: estimation of m hat 5 by gx}
0<\pdv{g_x}\widehat{m}(g_x,g_y)\leq \frac{3}{\pi}\log \left(\frac{288}{7(2 \sqrt{15}-\sqrt{55-12 \sqrt{15}}+4)}\right)+\frac{3}{\pi}\log(\frac{1}{g_y})-\frac{9g_y}{32\pi}+O(g_y^2)
\end{equation}
Also, 
\[
\pdv{g_y}F_{5,2}(x,g_x,g_y)=\pdv{g_y}F_{5,1}(x,g_x,g_y)=\frac{1}{\sqrt{(\frac{g_x}{g_y})^2-x(x+2)}}=\pdv{g_y}F_{2}(x,g_x,g_y).
\]
Therefore, using estimation \eqref{ineq: estimation of m hat 2 by gy}, we have 
\[
0<\int_{A_x(g_x,g_y)}^{\frac{1}{2}} \pdv{g_y}F_{5,1}(x,g_x,g_y)\dd x\leq \pi,
\]
and
\[
0<\int_{B_x(g_x,g_y)}^{\frac{1}{2}} \pdv{g_y}F_{5,2}(x,g_x,g_y)\dd x\leq \pi.
\]
Therefore, we have
\begin{equation}\label{ineq: estimation of m hat 5 by gy}
0<\pdv{g_y}\widehat{m}(g_x, g_y)\leq 6.
\end{equation}

\noindent\textbf{Case 6: $g_x\in (b_7(g_y),b_6(g_y))$}\par
In this case the overlap region is depicted in Figure \ref{fig:case6}.
\begin{figure}[h]
\includegraphics[scale=0.3]{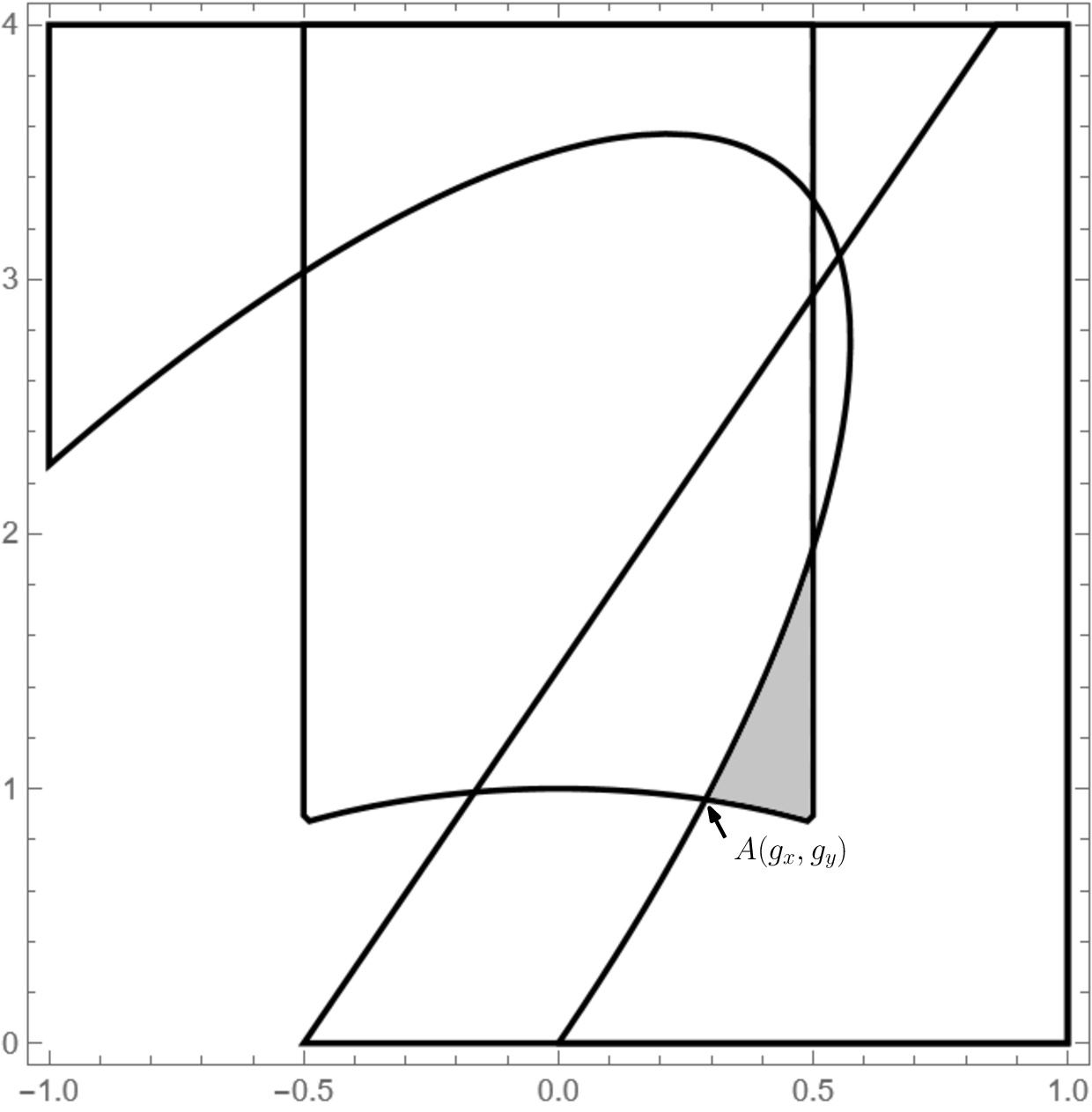}
\centering
\caption{Case 6}\label{fig:case6}
\end{figure}\par 
For $g_x\in (b_7(g_y),b_6(g_y))$, 
\[
\begin{split}
\widehat{m}(g_x,g_y)=\frac{3}{\pi}\int_{A_x(g_x,g_y)}^{\frac{1}{2}}F_{5,1}(x,g_x,g_y)\dd x.
\end{split}
\]
We have 
\[
\pdv{g_x}\widehat{m}(g_x, g_y)=\int_{A_x(g_x,g_y)}^{\frac{1}{2}}\pdv{g_x}F_{5,1}(x,g_x,g_y)\dd x,
\]
\[
\pdv{g_y}\widehat{m}(g_x, g_y)=\int_{A_x(g_x,g_y)}^{\frac{1}{2}}\pdv{g_y}F_{5,1}(x,g_x,g_y)\dd x.
\]

Since $g_x\in (b_7(g_y),b_6(g_y))\subset  (b_7(g_y),b_5(g_y))$, from Case 5, we know $A_x(g_x,g_y)>  \frac{g_y}{4}$ and $b(g_x,g_y)=\frac{g_y^2}{g_x^2}\in (0,\frac{4}{5})$. Therefore, we have  
\begin{equation}\label{ineq: estimation of m hat 6 by gx}
\begin{split}
0<\pdv{g_x}\widehat{m}(g_x,g_y)\leq &\frac{3}{\pi} \int_{A_x(g_x,g_y)}^{\frac{1}{2}}\pdv{g_x}F_{5,1}(x,g_x,g_y)\dd x\\
\leq &\frac{3}{\pi}(\log(f(\frac{4}{5},\frac{1}{2}))-\log(f(0,\frac{g_y}{4})))=\frac{3}{\pi}\log (6)+\frac{3}{\pi}\log (\frac{1}{g_y}).
\end{split}
\end{equation}

Also, like in Case 5, we have
\begin{equation}\label{ineq: estimation of m hat 6 by gy}
0<\pdv{g_y}\widehat{m}(g_x, g_y)=\frac{3}{\pi}\int_{B_x(g_x,g_y)}^{\frac{1}{2}} \pdv{g_y}F_{5,1}(x,g_x,g_y)\dd x\leq 3.
\end{equation}

\noindent\textbf{Case 7: $g_x\in (-\infty,b_7(g_y))$}\par
In this case the overlap region is depicted in Figure \ref{fig:case7}.
\begin{figure}[h]
\includegraphics[scale=0.3]{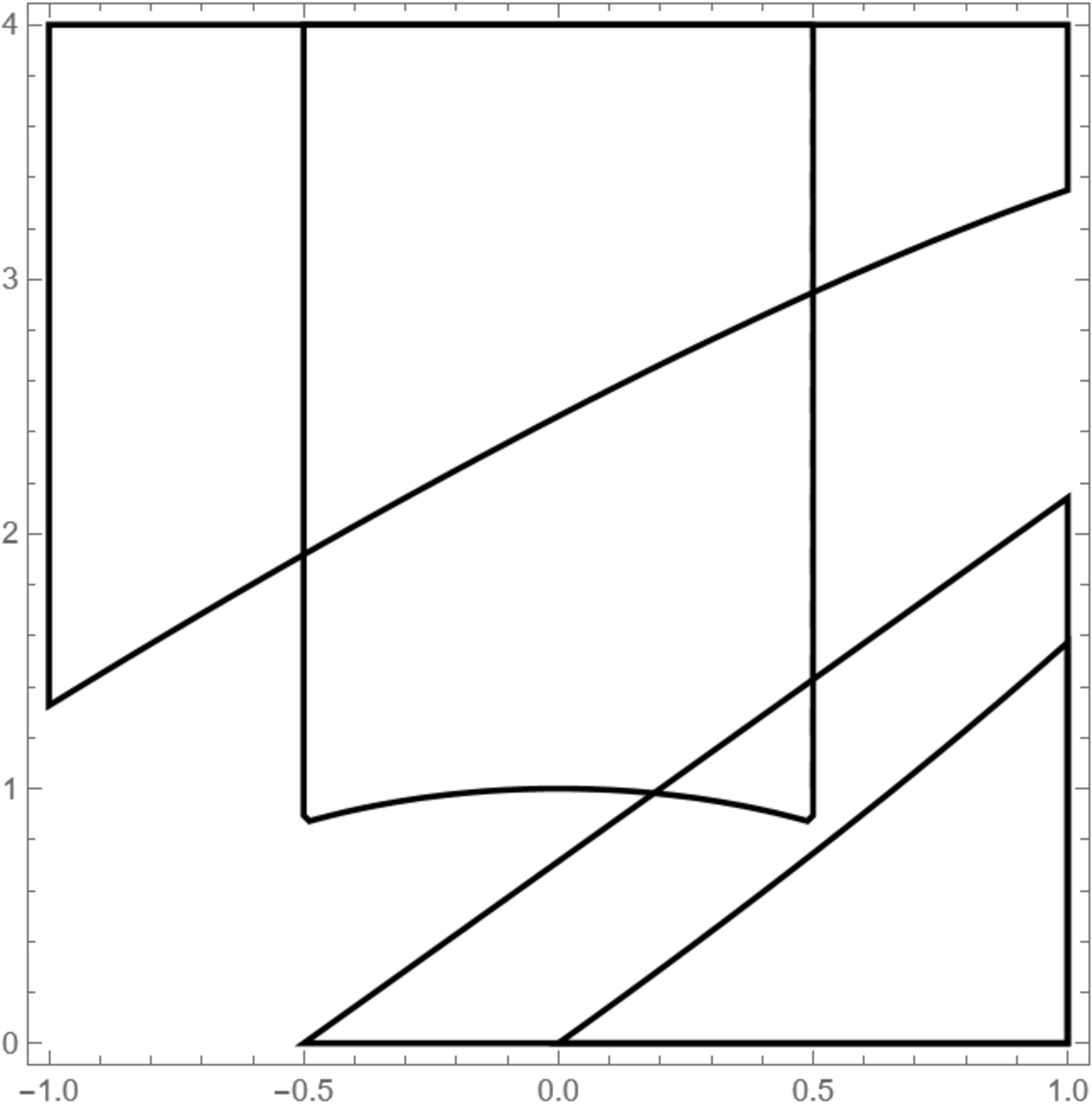}
\centering
\caption{Case 7}\label{fig:case7}
\end{figure}\par
We have $\widehat{m}(g_x,g_y)=0$ and thus $\pdv{g_x}\widehat{m}(g_x, g_y)=\pdv{g_y}\widehat{m}(g_x, g_y)=0$.\par

\noindent\textbf{Case 8: $g_x\in (-\frac{2}{\sqrt{3}},0)$}\par
In this case the overlap region is depicted in Figure \ref{fig:case8}.
\begin{figure}[h]
\includegraphics[scale=0.3]{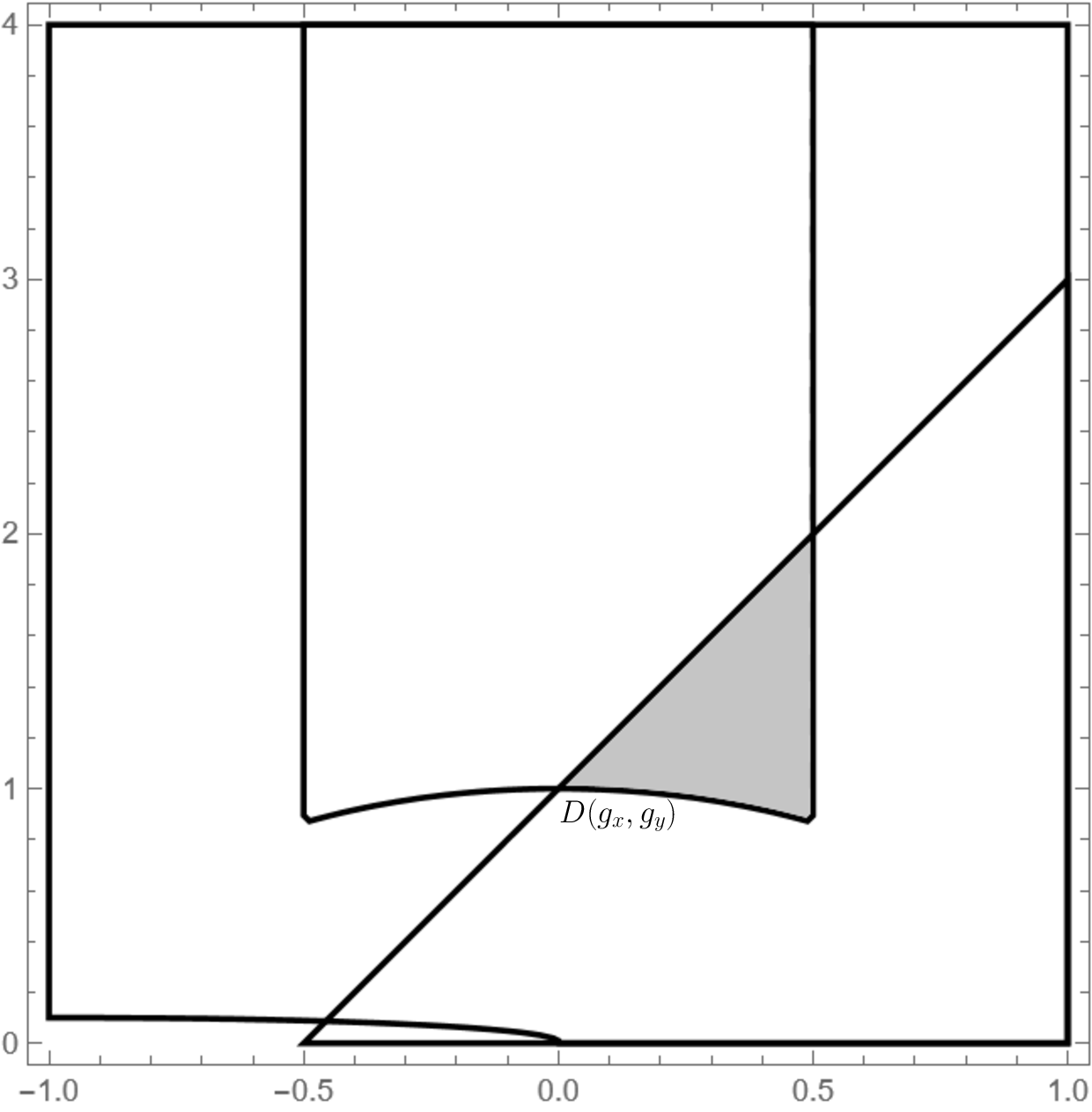}
\centering
\caption{Case 8}\label{fig:case8}
\end{figure}
\par

Let $D_x(g_x,g_y)$ be the $x$-coordinate of the intersection point of the line $L$ and the circle $x^2+y^2=1$, i.e., 
\[
L_y(D_x(g_x,g_y),g_x,g_y)=\sqrt{1-D_x(g_x,g_y)^2}.
\]
This yields
\[
D_x(g_x,g_y)=-\frac{g_x \sqrt{4 g_x^2+3}+1}{2 g_x^2+2}.
\]
For $g_x\in (-\frac{2}{\sqrt{3}},0)$,
\[
\begin{split}
\widehat{m}(g_x,g_y)=&\frac{3}{\pi}\int_{D_x(g_x,g_y)}^{\frac{1}{2}}\int_{\sqrt{1-x^2}}^{L_y(x,g_x,g_y)}\frac{1}{y^2}\dd y\dd x=-\frac{3}{\pi}\int_{D_x(g_x,g_y)}^{\frac{1}{2}}\frac{1}{L_y(x,g_x,g_y)}-\frac{1}{\sqrt{1-x^2}}\dd x\\
=&\frac{3}{\pi}\int_{D_x(g_x,g_y)}^{\frac{1}{2}}F_8(x,g_x,g_y)\dd x.
\end{split}
\]

Since $F_8(D_x(g_x,g_y),g_x,g_y)=0$, we have
\[
\pdv{g_x}\widehat{m}(g_x, g_y)=\pdv{g_x}\int_{D_x(g_x,g_y)}^{\frac{1}{2}}F_8(x,g_x,g_y)\dd x=\eval{\log(1+2x)}_{D_x(g_x,g_y)}^{\frac{1}{2}}=\log \left(\frac{2 \left(g_x^2+1\right)}{g_x \left(g_x-\sqrt{4 g_x^2+3}\right)}\right),
\]
\[
\pdv{g_y}\widehat{m}(g_x, g_y)=\pdv{g_y}\int_{D_x(g_x,g_y)}^{\frac{1}{2}}F_8(x,g_x,g_y)\dd x=0.
\]
For $g_x\in (-\frac{2}{\sqrt{3}},0)$, we have $-\frac{2}{\sqrt{3}}<\frac{2 \left(g_x^2+1\right)}{g_x-\sqrt{4 g_x^2+3}}<-1$. So we have
\begin{equation}\label{ineq: estimation of m hat 8 by gx}
\pdv{g_x}\widehat{m}(g_x, g_y)\leq\log(-\frac{2}{\sqrt{3}g_x}).
\end{equation}
Combining all estimations of $\pdv{g_x}\widehat{m}(g_x, g_y)$ and $\pdv{g_x}\widehat{m}(g_x, g_y)$ in Case 1 to Case 8, the proof of Lemma \ref{Lem: dereasing ratio of derivative of m hat} is finished.
\refstepcounter{subsection}
\subsection*{\thesubsection\quad Proof of Theorem \ref{thm: decay of Hilbert transform on SL2(R)}}
\phantomsection
\label{Sec: Proof of the main theorem}

First we prove the integrability of $\dv{t}\widetilde{\widetilde{m}}(k_0a\exp(tX_j))$, i.e., the integrability of $2e^{2t}g_y\pdv{g_y}\widehat{m}(g_x,e^{2t}g_y)$ and $g_y\pdv{g_x}\widehat{m}(g_x+tg_y,g_y)$ for small $t$. Since $g_y(r,\theta)\leq \max(\frac{1}{r^2},r^2)$, by Lemma \ref{Lem: dereasing ratio of derivative of m hat}, we only need to prove the integrability of $\log(\frac{1}{\abs{g_x(r,\theta)+t g_y(r,\theta)}})$ and $\log(\frac{1}{\abs{g_x(r,\theta)}})$.\par
We have 
\[
\abs{g_x(r,\theta)+tg_y(r,\theta)}=\frac{\abs{\left(r^4-1\right) \sin \theta \cos\theta+r^2 t}}{r^4 \sin ^2\theta+\cos ^2\theta}\geq \frac{1-r^4}{2}\abs*{\sin (2\theta)-\frac{2r^2 t}{1-r^4}}
\]
for $r$ close to $0$. Let $\sin(2\omega)=\frac{2r^2 t}{1-r^4}$, we have 
\[
\begin{split}
&\int_{-\frac{\pi}{2}}^{\frac{\pi}{2}}\log(\frac{1}{\abs{g_x(r,\theta)+t g_y(r,\theta)}})\dd\theta\\
\leq& \pi\log(\frac{2}{1-r^4})-\int_{-\frac{\pi}{2}}^{\frac{\pi}{2}}\log(\abs{\sin(2\theta)-\sin(2\omega)})\dd\theta\\
=&\pi\log(\frac{2}{1-r^4})-\log(2)-\int_{-\frac{\pi}{2}}^{\frac{\pi}{2}}\log(\abs{\cos(\theta+\omega)})\dd\theta-\int_{-\frac{\pi}{2}}^{\frac{\pi}{2}}\log(\abs{\sin(\theta-\omega)})\dd\theta\\
=&\pi\log(\frac{2}{1-r^4})-\log(2)-\int_{-\frac{\pi}{2}}^{\frac{\pi}{2}}\log(\abs{\cos\theta})\dd\theta-\int_{-\frac{\pi}{2}}^{\frac{\pi}{2}}\log(\abs{\sin\theta})\dd\theta\\
=&\pi\log(\frac{2}{1-r^4})-\log(2)+2\pi\log(2).\\
\end{split} 
\]
Let $t=0$, we have $\omega=0$ and hence
\[
\int_{-\frac{\pi}{2}}^{\frac{\pi}{2}}\log(\frac{1}{\abs{g_x(r,\theta)}})\dd\theta\leq\pi\log(\frac{2}{1-r^4})-\log(2)+2\pi\log(2).
\]
Therefore, by dominated convergence theorem, we have 
\begin{equation}
\begin{cases}
\partial_{X_1} \widetilde{m}(a)=\frac{1}{\pi}\int_{-\frac{\pi}{2}}^{\frac{\pi}{2}} 2g_y(r,\theta)\pdv{g_y}\widehat{m}(g_x(r,\theta),g_y(r,\theta))\dd \theta,\\
\partial_{X_2} \widetilde{m}(a)=\frac{1}{\pi}\int_{-\frac{\pi}{2}}^{\frac{\pi}{2}} g_y(r,\theta)\pdv{g_x}\widehat{m}(g_x(r,\theta),g_y(r,\theta))\dd \theta.
\end{cases}
\end{equation}
The ratio $\frac{g_x(r,\theta)}{g_y(r,\theta)}=\frac{(r^4-1)\cos\theta\sin\theta}{r^2}$ indicates that as $\theta$ varies from $-\frac{\pi}{2}$ to $\frac{\pi}{2}$, the ellipse $E$ and line $L$ first rotate counter-clockwise, then clockwise, and finally counter-clockwise again when $r$ is close to $0$. The transformation $(r,\theta) \mapsto (g_x(r,\theta),g_y(r,\theta))$ is invertible for $\theta \in \left(-\frac{\pi}{4}, \frac{\pi}{4}\right)$. Figure \ref{fig: rotate of theta} illustrates these changes as $\theta$ increases.
\begin{figure}[h]
\centering
\includegraphics[scale=0.4]{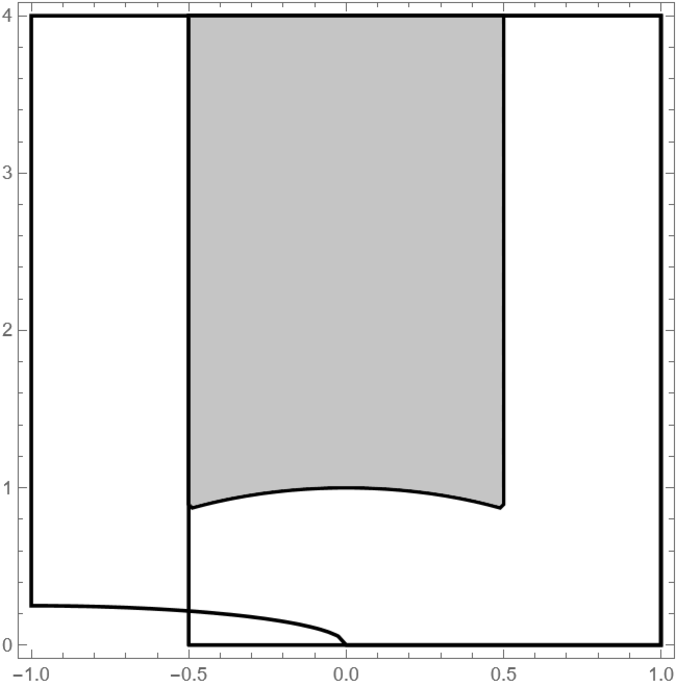}
\includegraphics[scale=0.4]{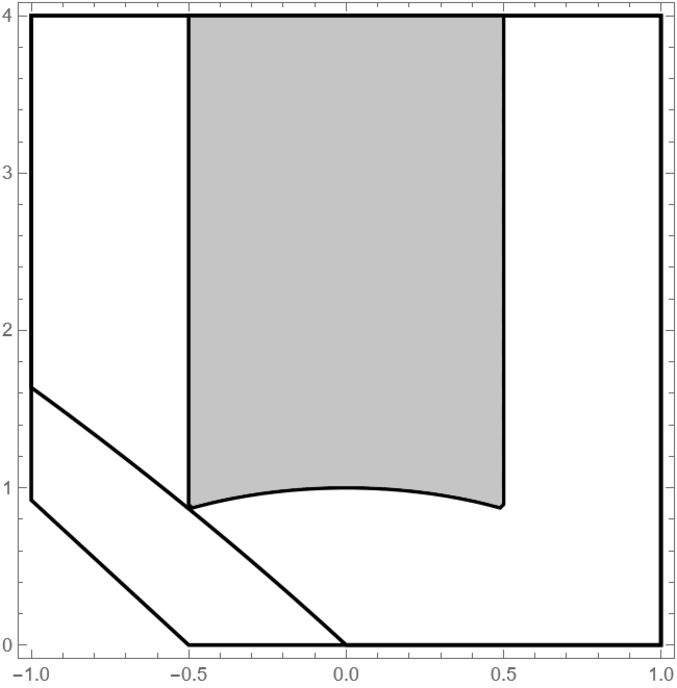}
\includegraphics[scale=0.4]{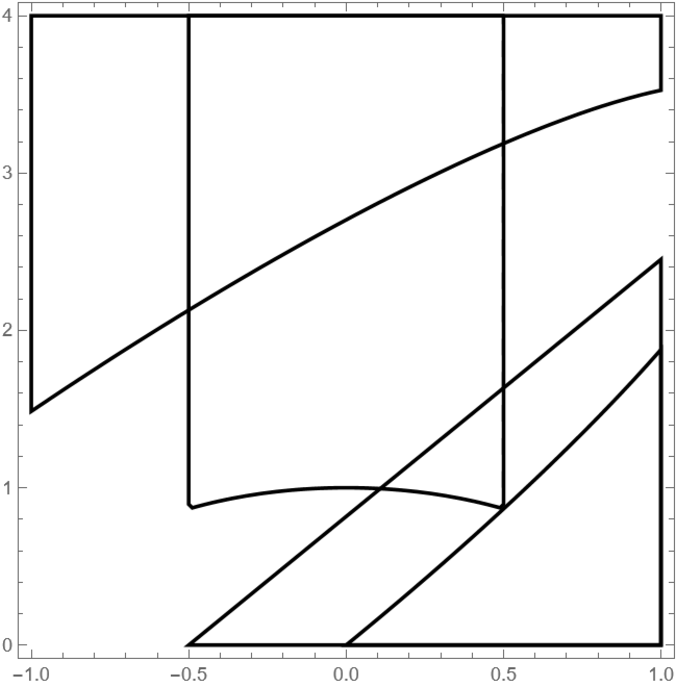}
\includegraphics[scale=0.4]{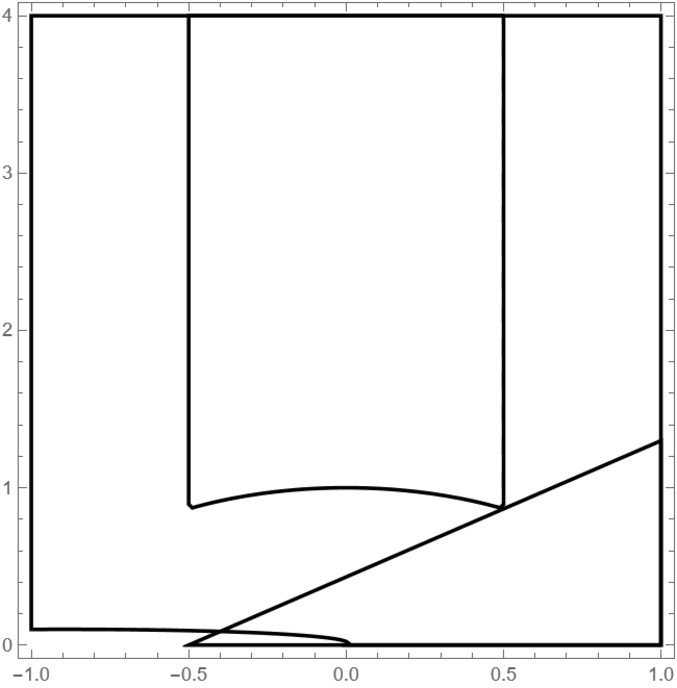}
\includegraphics[scale=0.4]{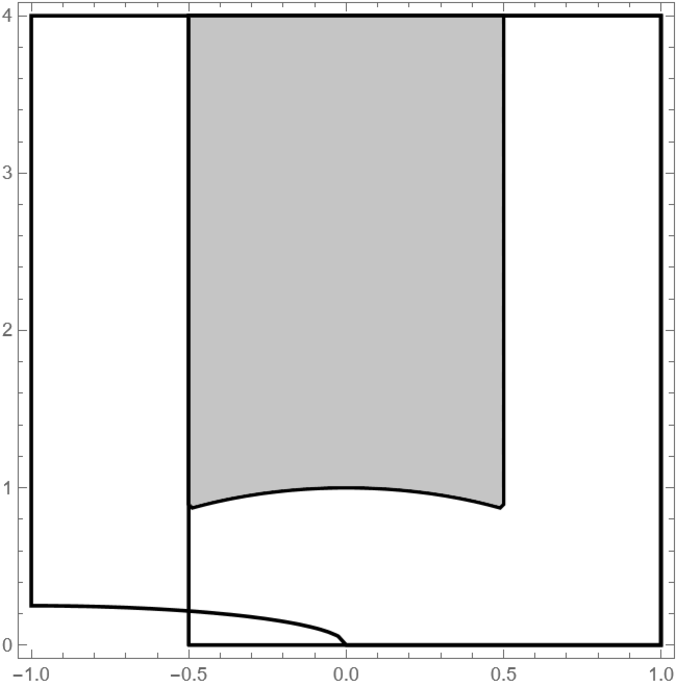}
\caption{The illustration of $A(g_x(r,\theta),g_y(r,\theta))\cap \mathcal{F}$ as $\theta$ increases.}\label{fig: rotate of theta}
\end{figure}\par

Solving the boundaries $g_x(r,\theta)=b_2(g_y(r,\theta))$, $g_x(r,\theta)=b_7(g_y(r,\theta))$ and $g_x(r,\theta)=b_8(g_y(r,\theta))$ for $r$ close to $0$, we obtain the corresponding boundaries 
\[
\begin{cases}
\theta_2(r)=-\frac{\pi}{6},\\
\theta_7(r)=\cos^{-1}\left(\sqrt{\frac{16r^4-3\sqrt{9r^8-66r^4+9}-12}{28(r^4-1)}}\right),\\
\theta_8(r)=\tan^{-1}\left(\frac{4\sqrt{3}}{-3r^4+\sqrt{9r^8-66r^4+9}+3}\right),
\end{cases}
\]
and 
\[
\theta_7(0)=\frac{\pi}{6}, \quad \theta_8(0)=\frac{\pi}{2}.
\]

For $\theta\in (-\frac{\pi}{4},\frac{\pi}{4})$ and $r$ close to $0$, we have estimations:
\[
\abs{g_x(r,\theta)}=\frac{\abs{(r^4-1)\cos\theta \sin\theta}}{\cos^2\theta +r^4\sin^2\theta}\in (\frac{\abs{\sin(2\theta)}}{4},\abs{\sin(2\theta)}),
\]
\[
g_y(r,\theta)=\frac{r^2}{\cos^2\theta+r^4\sin^2\theta}\in (r^2,2r^2).
\]

For $\theta\in (\theta_8(r),\frac{\pi}{2})$ and $r$ close to $0$, we have 
\[
\begin{split}
g_x(r,\theta)=&\frac{\left(r^4-1\right) \sin\theta \cos\theta}{r^4 \sin ^2\theta+\cos ^2\theta}=\frac{\left(r^4-1\right)}{r^4 \tan\theta+\cot\theta}\in\left( -\frac{1}{r^4\tan\theta},-\frac{1}{4r^4\tan\theta} \right),
\end{split}
\]
by the estimation $r^4\tan\theta \geq r^4 \tan(\theta_8(r))\to \frac{\sqrt{3}}{2}$ as $r\to 0$. And
\[
g_y(r,\theta)=\frac{r^2}{\cos^2\theta+r^4\sin^2\theta}\in \left( \frac{1}{2r^2},\frac{2}{r^2} \right),
\]
by estimations $\sin\theta\geq \sin(\theta_8(r))\to 1$ as $r\to 0$ and $\cos^2\theta\leq \cos^2(\theta_8(r))=O(r^8)$.\par
By the above estimations of $g_x(r,\theta)$ and $g_y(r,\theta)$ and Lemma \ref{Lem: dereasing ratio of derivative of m hat}, we have 
\[
\begin{split}
&\int_{-\frac{\pi}{2}}^{\frac{\pi}{2}}g_y(r,\theta)\pdv{g_x}\widehat{m}(g_x(r,\theta),g_y(r,\theta))\dd\theta\\
=&\int_{-\frac{\pi}{2}}^{\theta_7(r)}g_y(r,\theta)\pdv{g_x}\widehat{m}(g_x(r,\theta),g_y(r,\theta))\dd\theta+\int_{\theta_8(r)}^{\frac{\pi}{2}}g_y(r,\theta)\pdv{g_x}\widehat{m}(g_x(r,\theta),g_y(r,\theta))\dd\theta\\
\lesssim & r^2\left( \int_{-\frac{\pi}{2}}^{\theta_7(r)}\log(\frac{1}{\sin(2\theta)})+\log(\frac{1}{r^2})\dd\theta \right)+\frac{1}{r^2}\left(  \int_{\theta_8(r)}^{\frac{\pi}{2}}\log(4r^4\tan\theta)\dd\theta \right)\\
\lesssim & r^2\log(\frac{1}{r^2})=O(r)
\end{split}
\]
and 
\[
\int_{-\frac{\pi}{2}}^{\frac{\pi}{2}}g_y(r,\theta)\pdv{g_y}\widehat{m}(g_x(r,\theta),g_y(r,\theta))\dd\theta=\int_{-\frac{\pi}{2}}^{\theta_7(r)}g_y(r,\theta)\pdv{g_x}\widehat{m}(g_x(r,\theta),g_y(r,\theta))\dd\theta\leq 12\pi r^2.
\]
Let $f_j(r)=\partial_{X_j}\widetilde{m}(a)$. Since $\partial_{X_j}\widetilde{m}(a)$ is left-$K$-invariant, we have $f_j(r)=f_j(\frac{1}{r})$.
By $\norm{kak'}=\norm{a}=\max(r,\frac{1}{r})$, we have 
\[
\limsup_{\norm{kak'}\to \infty}\norm{kak'}\abs{ \partial_{X_j} \widetilde{m}(kak')}\leq \limsup_{\norm{a}\to \infty}\norm{a}(\abs{\partial_{X_1} \widetilde{m}(a)}+\abs{\partial_{X_2} \widetilde{m}(a)})=\limsup_{r\to 0^+}\frac{1}{r}(\abs{f_1(r)}+\abs{f_2(r)}).
\]
Therefore, Theorem \ref{thm: decay of Hilbert transform on SL2(R)} is proved since $f_i(r)=O(r)$ as $r\to 0^+$.

\begin{remark}
By a more careful analysis, $ \int_{-\frac{\pi}{2}}^{\frac{\pi}{2}}g_y(r,\theta)\pdv{g_x}\widehat{m}(g_x(r,\theta),g_y(r,\theta))\dd\theta$ is also in the order of $O(r^2)$.
\end{remark}

\begin{remark}\label{Rmk: non-existence of second order}
Take $X_i=X_j=X_2$. By mean value theorem we have 
\[
\begin{split}
  \partial_{X_2}\partial_{X_2} \widetilde{m}(a)=&\eval{\dv{t}}_{t=0}\partial_{X_2} \widetilde{m}(a\exp(tX_2))\\
  =&\eval{\dv{t}}_{t=0}\eval{\dv{l}}_{l=0}\widetilde{m}(a\exp(tX_2)\exp(lX_2))\\
  =&\frac{1}{\nu(K_+)}\lim_{\delta\to 0}\lim_{\varepsilon\to 0}\int_{K_+}\eval{\dv{t}}_{t=\delta}\eval{\dv{l}}_{l=\varepsilon}\widetilde{\widetilde{m}}(k_0a\exp(tX_2)\exp(lX_2))\dd\nu(k_0)\\
  =&\frac{1}{\pi}\lim_{\delta\to 0}\lim_{\varepsilon\to 0}\int_{-\frac{\pi}{2}}^{\frac{\pi}{2}}g_y(r,\theta)^2\pdv[2]{g_x}\widehat{m}(g_x(r,\theta)+(\delta+\varepsilon)g_y(r,\theta),g_y(r,\theta))\dd\theta.\\
\end{split}
\]
We will show that there is an interval $I\subset (-\frac{\pi}{2},\frac{\pi}{2})$ such that
\[
g_y(r,\theta)^2\pdv[2]{g_x}\widehat{m}(g_x(r,\theta)+(\delta+\varepsilon)g_y(r,\theta),g_y(r,\theta))>0,
\]
and 
\[
\int_I g_y(r,\theta)^2\pdv[2]{g_x}\widehat{m}(g_x(r,\theta),g_y(r,\theta))=\infty.
\]
Therefore, by Fatou's lemma we have 
\[
\lim_{\delta\to 0}\lim_{\varepsilon\to 0}\int_{I}g_y(r,\theta)^2\pdv[2]{g_x}\widehat{m}(g_x(r,\theta)+(\delta+\varepsilon)g_y(r,\theta),g_y(r,\theta))\dd\theta   \geq \int_I g_y(r,\theta)^2\pdv[2]{g_x}\widehat{m}(g_x(r,\theta),g_y(r,\theta))=\infty.
\]
So $\partial_{X_2}\partial_{X_2} \widetilde{m}(a)$ is infinity or does not exist.

We have 
\[
\begin{split}
\pdv[2]{g_x}\widehat{m}(g_x, g_y)=&-\frac{2}{1+2D_x(g_x,g_y)}D_x^{(1,0)}(g_x,g_y)=\frac{\frac{g_x}{\sqrt{4 g_x^2+3}}-1}{g_x^3+g_x}.
\end{split}
\]
For $g_x(r,\theta)\in (-\frac{2}{\sqrt{3}},0)$, we have $4 \left(\sqrt{3}-2\right)\leq\frac{\frac{g_x}{\sqrt{4 g_x^2+3}}-1}{g_x^2+1}\leq -\frac{3}{5}$. So 
\[
\begin{split}
-\frac{3}{5g_x}\leq \pdv[2]{g_x}\widehat{m}(g_x, g_y)\leq \frac{4 \left(\sqrt{3}-2\right)}{g_x}.
\end{split}
\]

Therefore, we have 
\[
\begin{split}
\int_{\theta_8(r)}^{\frac{\pi}{2}}g_y(r,\theta)^2 \pdv[2]{g_x}\widehat{m}(g_x(r,\theta),g_y(r,\theta))\dd \theta\geq &\frac{3}{5}\int_{\theta_8(r)}^{\frac{\pi}{2}}\frac{r^4\tan\theta}{4r^4}\dd\theta\\
=&\infty.
\end{split}
\]
\end{remark}

\end{document}